\pdfoutput=1
\RequirePackage{ifpdf}
\ifpdf 
\documentclass[pdftex]{sigma}
\else
\documentclass{sigma}
\fi

\usepackage{bm} 

\input xy \xyoption{all}

\usepackage{enumerate}

\newcommand{\grB}{{$\N$-graded manifold}}
\newcommand{\grBs}{{$\N$-graded manifolds}}

\newcommand{\ZM}{{VBC}} %

\newcommand{\Zakrzewski}{{vector bundle comorphism}}
\newcommand{\BIGZakrzewski}{{Vector bundle comorphism}}
\newcommand{\degree}{{order}}

\newcommand{\thh}[1]{#1^{\mathrm{th}}} 
\newcommand{\st}[1]{#1^{\mathrm{st}}} 

\newcommand{\R}{\mathbb{R}}
\newcommand{\Z}{\mathbb{Z}}
\newcommand{\N}{\mathbb{N}}
\newcommand{\D}{\mathbb{D}}

\newcommand{\wt}[1]{\widetilde{#1}} 
\newcommand{\wh}[1]{\widehat{#1}} 

\newcommand{\und}{\underline}

\newcommand{\ideal}[1]{\langle {#1} \rangle}

\newcommand{\graf}{\operatorname{graph}}

\newcommand{\pa}{\partial}
\newcommand{\Sec}{\operatorname{Sec}}
\newcommand{\rank}{\operatorname{rank}}

\newcommand{\id}{\operatorname{id}}

\newcommand{\pt}{\operatorname{pt}}
\newcommand{\dd}{\mathrm{d}}

\newcommand{\V}{\mathbf{V}} 
\newcommand{\X}{\mathfrak{X}}
\newcommand{\Tau}{\mathcal{T}} 

\newcommand{\lra}{\longrightarrow}
\newcommand{\ra}{\rightarrow}
\newcommand{\mZM}{\Rightarrow}

\newcommand{\eps}{\varepsilon}
\newcommand{\del}{\delta}

\newcommand{\Cross}{{\sf X}} 
\newcommand{\plus}{\boldsymbol{\pmb{+}}}
\newcommand{\minus}{\boldsymbol{\pmb{-}}}
\newcommand{\rhoL}{\rho_L} 
\newcommand{\rhoR}{\rho_R}
\newcommand{\A}{\mathcal{A}} 
\newcommand{\E}[1]{E^{[#1]}} 
\newcommand{\prealg}{pre-algebroid} 
\newcommand{\Prealg}{Pre-algebroid}

\newcommand{\T}{\mathrm{T}} 
\newcommand{\TT}[1]{\mathrm{T}^{#1}} 

\newcommand{\tclass}[2]{\bm{\mathrm{t}}^{#1}#2} 

\newcommand{\dbydt}{\frac{\mathrm{d}}{\mathrm{d}t}}
\newcommand{\dbydtk}[1]{\frac{\mathrm{d}^{#1}}{\mathrm{d}t^{#1}}}

\newcommand{\core}[1]{\widehat{#1}} 
\newcommand{\Core}{\mathbf{C}}

\newcommand{\catZM}{{\mathcal{VBC}}} 
\newcommand{\catGB}{\mathcal{GM}}
\newcommand{\catHA}{\mathcal{HA}} 
\newcommand{\VB}{\mathcal{VB}} 

\newcommand{\G}{\mathcal{G}} 
\newcommand{\Cf}{\mathcal{C}^\infty}
\newcommand{\g}{\mathfrak{g}} 
\newcommand{\Adm}{\operatorname{Adm}} 
\newcommand{\VA}{\mathcal{VA}} 

\newcommand{\VF}{\mathfrak{X}}

\newdir{|>}{!/4,5pt/@{|}*:(1,-.2)@^{>}*:(1,+.2)@_{>}}
\def\relto{\rightarrow\!\!\vartriangleright} 

\def\<#1>{\left\langle #1\right\rangle}
\def\(#1){\left( #1\right)}

\numberwithin{equation}{section}

\newtheorem{thm}{Theorem}[section]
\newtheorem{prop}[thm]{Proposition}
\newtheorem{lem}[thm]{Lemma}

\theoremstyle{definition}
\newtheorem{df}[thm]{Definition}
\newtheorem{ex}[thm]{Example}

\newtheorem{rem}[thm]{Remark}

\begin{document}

\allowdisplaybreaks

\newcommand{\arXivNumber}{1708.03174}

\renewcommand{\PaperNumber}{135}

\FirstPageHeading

\ShortArticleName{Higher-Order Analogs of Lie Algebroids via Vector Bundle Comorphisms}

\ArticleName{Higher-Order Analogs of Lie Algebroids\\ via Vector Bundle Comorphisms}

\Author{Micha{\l} J\'{O}\'{Z}WIKOWSKI~$^\dag$ and Miko{\l}aj ROTKIEWICZ~$^\ddag$}

\AuthorNameForHeading{M.~J\'o\'zwikowski and M.~Rotkiewicz}

\Address{$^\dag$~Institute of Mathematics, Polish Academy of Sciences,\\
\hphantom{$^\dag$}~\'{S}niadeckich 8, 00-656 Warszawa, Poland}
\EmailD{\href{mailto:m.jozwikowski@mimuw.edu.pl}{m.jozwikowski@mimuw.edu.pl}}
\URLaddressD{\url{https://www.impan.pl/~mjoz/}}

\Address{$^\ddag$~Faculty of Mathematics, Informatics and Mechanics, University of Warsaw,\\
\hphantom{$^\ddag$}~Banacha 2, 02-097 Warszawa, Poland}
\EmailD{\href{mailto:mrotkiew@mimuw.edu.pl}{mrotkiew@mimuw.edu.pl}}
\URLaddressD{\url{https://www.mimuw.edu.pl/~mrotkiew/}}

\ArticleDates{Received January 10, 2018, in final form December 12, 2018; Published online December 29, 2018}

\Abstract{We introduce the concept of a higher algebroid, generalizing the notions of an algebroid and a higher tangent bundle. Our ideas are based on a description of (Lie) algebroids as vector bundle comorphisms -- differential relations of a special kind. In our approach higher algebroids are vector bundle comorphism between graded-linear bundles satisfying natural axioms. We provide natural examples and discuss applications in geometric mechanics.}


\Classification{58A20; 58A50; 70G65; 58E30; 70H50}

{\small \setcounter{tocdepth}{2}
\tableofcontents}

\section{Introduction}\label{sec:intro}

\looseness=-1 The main goal of this paper is to introduce a concept of a (general) \emph{higher algebroid}\footnote{In this paper under the name \emph{algebroid} we understand generalizations of the notion of a Lie algebroid obtained by weakening its axioms. In the geometric mechanics literature these generalizations are commonly known as \emph{almost-Lie} (no Jacobi identity), \emph{skew} (no Jacobi and no anchor-bracket compatibility), or \emph{general} (no Jacobi, no anchor-bracket compatibility, no skew-symmetry of the bracket) algebroids (the latter are often simply called \emph{algebroids}). Throughout the paper we shall refer to these objects as to \emph{$($Lie$)$ algebroids} meaning ``Lie algebroids and their above-mentioned generalizations''. They should not be confused with \emph{Courant algebroids} or other similarly named concepts.\label{foot:algebroid}} and study basic properties and applications of this notion. We work essentially within the framework of the theory of \grBs\ and the theory of differential relations (\Zakrzewski s).

{\bf Why higher algebroids?} Lie algebroids and their generalizations proved to be a fruitful area of study in the last three decades, either on their own right, as an offspring of Poisson geometry, and, particularly, in geometric mechanics. The latter direction was originated by Weinstein \cite{Weinstein_1996} and, following seminal papers of Mart\'{i}nez \cite{Martinez_geom_form_mech_lie_alg_2001, Martinez_lagr_mech_lie_alg_2001}, it developed into many different sub-branches (see a survey paper~\cite{Cortes_Inn_survey_2006} for a detailed discussion of the available literature).

From our point of view one of the most spectacular achievements of the algebroid-oriented study of mechanics is the recognition of the geometric structures standing behind variational calculus. The relation between (Lie) algebroids and variations is, perhaps, best described in two papers \cite{GG_var_calc_gen_alg_2008, GGU_geom_mech_alg_2006}, which themselves were inspired by much earlier works of Tulczyjew \cite{Tulczyjew_1976,Tulczyjew_1976a}. What makes the mentioned works particularly interesting is that the authors were able to identify the true geometric essence of variational calculus, which keeps working despite getting rid of several most natural assumptions, such as the Jacobi identity and even the skew-symmetry of the Lie algebroid bracket. In fact, putting aside the existence of real-life examples, it seems that the concept of a general algebroid \cite{JG_PU_Lie_alg_P-N_str_1997, JG_PU_Algebroids_1999} is the uttermost geometric reality behind first-order variational calculus.

From the above perspective it is most natural to look for similar geometric structures related with variational calculus of higher order. Thus we would like to introduce a geometric object, a~\emph{$($Lie$)$ higher algebroid}, which is present whenever a variational problem involving higher velocities is considered. We have already tried to address this question from the perspective of the classical groupoid--algebroid reduction~\cite{MJ_MR_models_high_alg_2015}, yet now we propose a more abstract and a more general approach.

{\bf Towards a proper definition -- objects.} In order one, the (Lie) algebroid structure ``lives'' on a vector bundle, which can be often interpreted as the velocity bundle of some physical system, i.e., as a quotient (in a suitable sense) of some tangent bundle $\T M$. The structure of the tangent bundle is thus our most important example of a (Lie) algebroid. Obviously, in higher-order variational calculus the role of $\T M$ is taken over by its higher analogue $\TT{k} M$, the bundle of $k$-velocities.\footnote{$\T^k M$ is the bundle of $k$-jets of curves in $M$ and it should not be mistaken with $\T^{(k)} M = \T(\T(\ldots (\T M)\ldots))$~-- the iterated tangent bundle.} Locally, $\TT{k} M$ can be characterized by specifying the points, velocities, accelerations, and higher derivatives of curves on $M$. Thus, having the perspective of variational calculus in mind, it is reasonable to expect that a proper object hosting a higher algebroid structure would be a bundle of some sort with fibre coordinates sharing a similar derivative-like nature as coordinates on $\TT{k} M$. This expectation is met by a class of fibre bundles with a typical fibre diffeomorphic with~$\R^n$, however, in general, transition functions need not to be linear but are polynomials which are homogeneous with respect to a prescribed gradation of coordinate functions (see Section~\ref{sec:gr_bund} for a detailed discussion). They are a~special case of the concept of a \emph{graded manifold} introduced by Voronov~\cite{Voronov_2002}.
Namely, they are \emph{purely even non-negatively graded manifolds} in the sense of Voronov, meaning that local coordinates are graded by non-negative integers. We follow Voronov's idea that such a non-negatively graded manifold, which expands into a tower of fibrations with polynomial transition functions (of a special form), should be regarded as a generalization of a vector bundle~\cite{Voronov_2002}, see also \cite{Voronov_Q_mfds_high_lie_alg_2010, Voronov_2012_napl, Voronov_2012}. (For this, we use notation typical for (vector) bundles~-- see Section~\ref{sec:gr_bund}.)
Properties of such bundles and their application in theoretical mechanics where also considered in several papers \cite{BGG_2015,BGR, JG_MR_gr_bund_hgm_str_2011}. For brevity, we shall use ``$\N$-graded'' for ``non-negatively graded'' (thinking that zero is included into the natural numbers) or simply ``graded'' if no other gradings will appear. Note that we consider only the even case, though the results are expected to hold in the super setting too.

{\bf Towards a proper definition -- structures.} The structure of an order one (Lie) algebroid on a~vector bundle $\sigma\colon E\ra M$ is most commonly introduced as a bi-linear bracket operation on the space of sections of its host bundle $\sigma$. It is quite obvious that this definition has no direct gene\-ra\-li\-zation to the higher-order case as, in particular, the higher tangent bundle \mbox{$\tau^k_M\colon \TT{k} M\ra M$}, which is an obvious candidate to host a higher (Lie) algebroid structure, admits no canonical bracket operation on its space of sections. Fortunately, the (Lie) algebroid structure on $\sigma$ has several other characterizations equivalent to the standard definition, perhaps more suitable for a~generalization. To mention just a few, we may describe an algebroid on~$\sigma$ as a certain homological vector field on a graded supermanifold~$E[1]$, a de Rham-like derivative in the space of $E$-differential forms (i.e., sections of $\bigwedge^{\bullet} E^\ast$), a certain linear 2-tensor field on the dual bundle~$\sigma^\ast$, a morphism $\eps\colon \T^\ast E\ra \T E^\ast$ of double vector bundles, or as a differential relation of a special kind (a \Zakrzewski ) $\kappa\subset\T E\times\T E$.

From the perspective of variational calculus, the latter characterization is the one most appealing to us.\footnote{Let us remark that, the double vector bundle morphism $\eps\colon \T^\ast E\ra \T E^\ast$ favored in the so-called Tulczyjew approach to geometric mechanics \cite{GG_var_calc_gen_alg_2008, GGU_geom_mech_alg_2006} is just the dual of $\kappa$.} This important role of $\kappa$ can be easily explained. For the tangent algebroid structure on $\tau_M\colon \T M\ra M$, $\kappa$ is, in fact, the canonical isomorphism $\kappa_M\colon \T\T M\ra \T\T M$ interchanging the two vector bundle structures on $\T\T M$ -- cf.\ Example~\ref{ex:tang_alg}. Its role in variational calculus is to present `a jet of curves as a curve of jets' (i.e., a~variation of the tangent lift of a trajectory in $M$ is constructed from the tangent lift of the curve of virtual displacements by means of $\kappa_M$). The situation in the higher-order case is no different and, perhaps, even easier to understand. Namely on a~higher-tangent bundle $\tau_M^k\colon \TT{k} M\ra M$, we have the canonical isomorphism $\kappa^k_M\colon \TT{k} \T M\ra \T\TT{k} M$ which, similarly as before, allows to express a variation of a~$\thh{k}$-tangent lift of a trajectory in $M$ as a $\thh{k}$-tangent lift of the related generator (a curve of virtual displacements) \cite{MJ_MR_higher_var_calc_2013}. When performing the standard Lie groupoid -- Lie algebroid reduction the character (but not the role) of $\kappa$ changes~-- the reduced object is no longer a differentiable map but, in general, only a differential relation (but of a special kind). Again there is no essential difference between the first- and higher-order cases \cite{MJ_MR_models_high_alg_2015}.

Summing up, geometrically relation $\kappa\subset\T E\times\T E$ is responsible for the construction of admissible variations of curves on a (Lie) algebroid on $E$. Its presence can be easily motivated by the groupoid -- algebroid reduction procedure: when performing a reduction of a variational problem it is not enough to reduce the trajectories, but also the variations! An this is how~$\kappa$ appears. Variational calculus on (Lie) algebroids can be, however, successfully developed in abstract terms (with a prominent role of $\kappa$) without referring to the reduction procedure~\cite{GG_var_calc_gen_alg_2008}.

All the above suggests that the language of differential relations is suitable to speak about higher (Lie) algebroids (at least from the perspective of applications in variational calculus). The first order-case \cite{GG_var_calc_gen_alg_2008} and the known higher-order examples from \cite{MJ_MR_prototypes} advocate that these relations should have a special nature, namely they should be \emph{\Zakrzewski s} (\emph{\ZM s} in short; see Definition~\ref{def:zm}). This assumption, in principle, allows to uniquely characterize admissible variations in terms of their generators (virtual displacements). Let us remark that the important role of \Zakrzewski s in the theory of Lie groupoids and Lie algebroids was recognized also in~\cite{Cattaneo_Dherin_Weinstein_2012}.

{\bf Higher (Lie) algebroids, examples and applications.} Summing up the above heuristic considerations, we postulate a \emph{higher $($Lie$)$ algebroid} to be a \grB\ $\sigma^k\colon E^k\ra M$ equipped with a \Zakrzewski\ $\kappa^k\subset\TT{k} E^1\times \T E^k$, which also should be naturally graded (Definition~\ref{def:higher_algebroid}). Here $E^1$ is the canonical reduction of $E^k$ to a first-order bundle, i.e., to a vector bundle.

For $k=1$ we recover the standard notion of a (Lie) algebroid. In fact, we devote entire Section~\ref{sec:algebroids_as_zm} to carefully formulate the theory of (the first order) Lie algebroids and their generalizations in the language of \ZM s. This has a two-fold purpose. First of all, the language of differential relations was so far hardly used in the theory of (Lie) algebroids and most of the results of this section are new (yet rather straightforward, so no true originality can be claimed by us). Secondly, this first-order formulation has a direct generalization to the higher order case, with most of the definitions being completely analogous.

 Apart from the first-order case, important examples of such structures are provided by higher tangent bundles $\TT{k} M$ with the isomorphisms~$\kappa^k_M$ mentioned above. Another important class of examples is given by the \emph{prolongations of almost-Lie algebroids} introduced in \cite{Col_deDiego_seminar_2011} and \cite{MJ_MR_models_high_alg_2015} (see also~\cite{Saunders_2004}). Prolongations naturally appear in higher-order variational calculus as reductions of higher tangent bundles of Lie groupoids. In our previous publications \cite{MJ_MR_prototypes,MJ_MR_models_high_alg_2015} (see also \cite{Martinez_2015}) we study several concrete cases. Additionally some examples related to Lie groups are discussed in Section~\ref{sec:examples}.

We further argue for the usefulness of the concept of a higher (Lie) algebroid introduced in this work by providing two particular applications. First of all, given a higher algebroid structure on $\sigma^k\colon E^k\ra M$ we were able to define a whole family of \emph{algebroid lifts} of sections of~$E^1$ to vector fields on~$E^k$. These generalize the well-know constructions of the vertical and the complete lifts of a section in the presence of a (Lie) algebroid structure. Secondly, in Section~\ref{sec:var_calc} we show that the geometric formulation of higher-order variational calculus is possible within the framework of higher (Lie) algebroids (from this perspective the algebroid lifts are directly related with conservation laws). In fact, we show that the latter works fine also for a more general class of \emph{higher \prealg s}, thus further generalization of our theory is possible. We refer to a recent publication \cite{Colombo_2017} and the references therein for numerous concrete interesting examples of higher-order variational problems.

{\bf Alternative approaches.} In the literature there have been a few attempts to introduce higher analogs of (Lie) algebroids. The version of a higher analog of Lie algebroids that we develop is quite distinct from the work of Mackenzie~\cite{Mackenzie_2011} and Voronov~\cite{Voronov_Q_mfds_high_lie_alg_2010, Voronov_2012}, though we make use of some of their concepts and methods. Voronov \cite{Voronov_Q_mfds_high_lie_alg_2010} proposed that higher algebroids should be understood as homological vector fields of weight one on a non-negatively graded supermanifold (generalizing Vaintrob's description of Lie algebroids as $Q$-manifolds \cite{Vaintrob_1997}). The language of $Q$-manifolds is central in the whole bracket geometry. Connections between higher (Lie) algebroids in the sense of this work and Voronov's ones or $k$-fold Lie algebroids of Mackenzie are not clear and need further studies.

A much more down-to-earth idea is to define higher (Lie) algebroids as bundles of higher-jets of admissible curves on a standard Lie algebroid. In this paper we refer to this construction as the \emph{prolongation of an algebroid} and briefly discuss it in Section~\ref{ssec:prolongations}. Such objects appeared for the first time in a seminar talk by Colombo and de Diego \cite{Col_deDiego_seminar_2011}, however without a full recognition of the relevant geometric structure. The latter was studied by us in \cite{MJ_MR_models_high_alg_2015} and successfully applied in higher-order variational problems on Lie algebroids and Lie groupoids \cite{MJ_MR_prototypes} (see also \cite{Martinez_2015}). Prolongations of algebroids in the above sense constitute a particular subclass of higher (Lie) algebroids in the understanding of this paper, perhaps the most important one as it contains reductions of higher tangent bundles of Lie groupoids.

The most recent attempt is due to Bruce, Grabowska and Grabowski. In \cite{Bruce_Grabowska_Grabowski_2016} they constructed a \emph{linearisation functor}, which, to a \grB\ $\sigma^k\colon E^k\ra M$ of \degree\ $k$, canonically associates a manifold $l\big(E^k\big)$ equipped with a graded-linear structure of a bi-\degree\ $(k-1,1)$ over $E^1$ and $E^{k-1}$, respectively. Now the higher (Lie) algebroid structure on $\sigma^k$ is defined as a (Lie) algebroid structure on the vector bundle $l\big(E^k\big)\ra E^{k-1}$ compatible with the second grading. The idea here is to mimic the canonical inclusion $\TT{k} M\hookrightarrow \T \TT{k-1} M$, which makes the higher tangent bundle $\TT{k}M$ a sub-object of the tangent Lie algebroid of $\TT{k-1}M$. In this way the construction of~\cite{Bruce_Grabowska_Grabowski_2016} admits higher tangent bundles and also prolongations of algebroids as examples of higher (Lie) algebroids, similarly to our approach. However, these two notions of a higher algebroid are different, as can be easily seen from coordinate calculations already in \degree\ 2. The construction of \cite{Bruce_Grabowska_Grabowski_2016} can be adapted to develop higher-order mechanics on \grBs\ \cite{BGG_2015}, equivalent to our results~\cite{MJ_MR_prototypes} in comparable cases. So perhaps the difference here is of a philosophical nature. In the classical case: should we treat $\thh{k}$-order variational calculus on a manifold $M$ as the first-order theory on~$\TT{k-1}M$ thus passing through unnecessary degrees of freedom (as would be in the spirit of \cite{Bruce_Grabowska_Grabowski_2016}) or do we prefer to work directly with $\TT{k} M$, as is in our formalism.

{\bf Organization of the paper.} In Section~\ref{sec:algebroids_as_zm} we reformulate the theory of (Lie) algebroids in the language of \Zakrzewski s. Section~\ref{sec:gr_bund} contains basic information concerning \grBs\ and weighted structures. In Section~\ref{sec:high_alg}, using the result of the previous two sections, we formulate the definition of a higher algebroid, provide a few natural examples and study in detail higher algebroid lifts and the Lie axiom in Section~\ref{ssec:higher_lie}. In Section~\ref{sec:var_calc} we develop the framework of higher-order variational calculus on higher algebroids, putting emphasis on admissible variations and their relation with conservation laws. Section~\ref{sec:examples} contains a study of higher algebroid structures inherited from higher tangent bundles of a Lie group $G$. In particular, we characterize all higher subalgebroids and describe a class of quotients of $\TT{k} G/G$. In Section~\ref{sec:final} we sketch a few perspectives of future research. Finally, in Appendix~\ref{sec:app} we have hidden most of the proofs of more technical results appearing in this work.

\subsection{Double vector bundles}
We shall frequently work with \emph{double vector bundles}, geometric objects sharing two (compatible) vector bundle structures. A typical example is the total space $\T E$ of the tangent bundle of a~vector bundle $\sigma\colon E\ra M$. The two vector bundle structures are: the tangent projection $\tau_E\colon \T E\ra E$, and $\T \sigma\colon \T E\ra \T M$, the tangent lift of~$\sigma$. The foundations of the theory of double vector bundles were laid by J.~Pradines~\cite{Pradines1977} (see also \cite[Chapter~9]{Mackenzie_lie_2005} for basic definitions, examples and historical remarks). Due to the results of~\cite{JG_MR_hi_VB_2009} we may formulate a definition of a~double vector bundle in the following way.
\begin{df}[double vector bundle]\label{df:DVB} A structure of a \emph{double vector bundle} (\emph{DVB}, in short) on a manifold $D$ is a pair of vector bundles $\sigma_A\colon D \ra A$, $\sigma_B\colon D \ra B$ such that for any $t, s\in\R$ and $x\in D$
\begin{gather}\label{eqn:DVB}
t\cdot_A s \cdot_B x = s\cdot_B t \cdot_A x,
\end{gather}
where $\cdot_A$ and $\cdot_B$ denote the multiplications by scalars in $\sigma_A$ and $\sigma_B$, respectively.

A \emph{morphism of DVBs} $(D,\sigma_A,\sigma_B)$ and $(D',\sigma_{A'},\sigma_{B'})$ is a smooth map $\phi\colon D\ra D'$ which is linear as a map from $\sigma_A$ to $\sigma_A'$ and, simultaneously, as a map from $\sigma_B$ to $\sigma_{B'}$.
\end{df}
The bases $A$, $B$ of a double vector bundle as above carry induced vector bundle structures over the common base $M$ giving rise to a diagram
\begin{gather*}
\xymatrix{D\ar[d]_{\sigma_A} \ar[rr]^{\sigma_B} && B\ar[d]^{\sigma^B_M} \\
A \ar[rr]^{\sigma^A_M} && M}
\end{gather*}
consisting of four vector bundle projections. All $M$, $A$, $B$ can be seen as submanifolds of $D$ via the zero section embeddings. In particular, the zero section $0_A\colon A\ra D$ defines a vector bundle structure on $A$ as a substructure of the vector bundle structure on~$\sigma_B\colon D\ra B$. Double vector bundles are often said to be `vector bundles in the category of vector bundles' what can be understood as the condition that all four structure maps of~$\sigma_A$ (the bundle projection~$\sigma_A$, the zero section~$0_A$, the scalar multiplication~$\cdot_A$ and the addition $+_A$) are vector bundle morphism with respect to the vector bundle structure of $\sigma_B$. The notion of a~double vector bundle naturally generalizes to the notion of a multi-\grB\ (see Section~\ref{sec:gr_bund}).

\section{(Lie) algebroids as \Zakrzewski s}\label{sec:algebroids_as_zm}

In this section we characterize Lie algebroids and their relaxed versions (such as generalized algebroids in the sense of Grabowski--Urba\'nski~\cite{JG_PU_Algebroids_1999}) by means of differential relations of a special kind (\Zakrzewski s, \ZM\ in short). This will be the cornerstone of our definition of higher algebroids in Section~\ref{sec:high_alg}.

\subsection{The category of \Zakrzewski s}

Following \cite{HM_1993, Mackenzie_lie_2005} we will recall a definition of a \emph{comorphism} between vector bundles (\ZM, in short). Later we shall study basic properties of \ZM s, i.e., the corresponding map on sections, morphisms between \ZM s, and \ZM s compatible with an additional linear structure.

{\bf \BIGZakrzewski s.}
\begin{df}[\Zakrzewski, \ZM ]\label{def:zm} A \emph{\Zakrzewski } (\emph{\ZM }, in short) from a vector bundle $\sigma_1\colon E_1\ra M_1$ to a vector bundle $\sigma_2\colon E_2\ra M_2$ is a relation $r\subset E_1 \times E_2$ of a~special form. Namely, there exists an ordinary vector bundle morphism $\phi\colon E_2^\ast\ra E_1^\ast$ covering a smooth map $\und{\phi}\colon M_2\ra M_1$ such that $r = \phi^\ast$, i.e., $r$ is the union of graphs of linear maps $r_y\colon (E_1)_{\und{\phi}(y)} \ra (E_2)_{y}$ such that $\phi_y\colon (E_2^\ast)_{y} \ra (E_1^\ast)_{\und{\phi}(y)}$ is the dual of the linear map $r_y$, where $y$ varies in $M_2$. We also say that the \ZM\ $r$ \emph{covers} the base map $\und{r}:=\und{\phi}$. Usually we shall denote a \ZM\ as $r\colon \sigma_1\relto \sigma_2$ or as a diagram
\begin{gather*}
\xymatrix{E_1 \ar[d]^{\sigma_1} \ar@{-|>}[r]^{r} & E_2 \ar[d]^{\sigma_2} \\ %
M_1 & M_2. \ar[l]_{\und{r}} }
\end{gather*}
In particular, the \ZM\ $r$ is a vector subbundle of $\sigma_1\times \sigma_2$.
\end{df}

There is a general concept of comorphisms between fibre bundles, (Lie) algebroid comorphisms and the corresponding (Lie) groupoid comorphisms. A full presentation of these ideas can be found in \cite{Cattaneo_Dherin_Weinstein_2012} and \cite{Mackenzie_lie_2005} with details on the history of these and related notions. The concept of a comorphism can be dated back at least to Bourbaki~\cite{Bourbaki_1971} and has been appearing ever since, sometimes under different names, see, e.g., \cite{Guillemin_Sternberg_1977,Higgins_Mackenzie_1990,Huebschmann_1990}. We would like to acknowledge that \ZM\ is the simplest example of a groupoid morphism in the sense of Zakrzewski \cite{Zakrz_quant_class_pseud_I_1990,Zakrz_quant_class_pseud_II_1990}. In the appendix of~\cite{MJ_MR_models_high_alg_2015} we studied some properties of \ZM s (using the name \emph{Zakrzewski morphism} proposed in this context in \cite{GG_var_calc_gen_alg_2008}).

\begin{rem}[\ZM s and maps on sections]\label{rem:zm_sections}
An important property of a \ZM\ is that it induces a mapping between the corresponding spaces of sections. Namely, given a \ZM\ $r\colon \sigma_1 \relto \sigma_2$ as in Definition \ref{def:zm} and a sections $s\in \Sec_{M_1}(E_1)$ we define $\hat{r}(s)\in \Sec_{M_2}(E_2)$ by
\begin{gather*}
\hat{r}(s) (y):= r_y(s(\und{r}(y))),\qquad \text{for every $y\in M_2$.}
\end{gather*}
In other words, the value of $\hat{r}(s)$ at a given point $y\in M_2$ is the unique element of the fibre~$(E_2)_y$ which is $r$-related to the value of $s$ at $\und{r}(y)$.
By the linearity of $r$, for every two sections $s,s'\in \Sec_{M_1}(E_1)$ we have $\hat{r}(s+s')=\hat{r}(s)+\hat{r}(s')$. Moreover, by construction, for any $f\in\Cf(M_1)$
\begin{gather}\label{e:hat_r_axiom}
\hat{r}(f\cdot s) = \und{r}^\ast(f) \cdot \hat{r}(s) .
\end{gather}

Conversely, any ${\mathbb R}$-linear map $\Sec_{M_1}(E_1)\ra \Sec_{M_2}(E_2)$ satisfying \eqref{e:hat_r_axiom} for some underlying map $\und{r}\colon M_2\ra M_1$ gives rise to a \ZM\ $r\colon \sigma_1\relto \sigma_2$. Checking this property is left to the reader.
\end{rem}

{\bf Morphisms between \ZM s.} A \ZM\ can describe a single geometric object. As we shall see shortly (Lemma~\ref{lem:kappa_from_sigma} and Proposition~\ref{prop:sigma_from_kappa}), a~(Lie) algebroid structure on a bundle $\sigma\colon E\ra M$ is fully encoded by a \ZM\ $\kappa$ from the tangent lift $\T\sigma\colon \T E\ra\T M$ to the tangent bundle \mbox{$\tau_E\colon \T E\ra E$}. This suggests that the standard notion of a (Lie) algebroid morphism should be expressible by a, properly defined, morphism between the corresponding \Zakrzewski s. Now we are going to introduce such a~notion under the name $\catZM$-morphism, i.e., a~morphism in the category $\catZM$ of \Zakrzewski s.

\begin{df}[$\catZM$-morphism]\label{def:zm_morphism} Let $(\sigma_i\colon E_i\to M_i)$ and $(\sigma_i'\colon E_i'\to M_i')$, $i=1, 2$, be four VBs and let $r$ and $r'$ be the following \Zakrzewski s:
\begin{gather}\label{def:ZMrr}\begin{split} &
\xymatrix{E_1 \ar[d]^{\sigma_1} \ar@{-|>}[r]^{r} & E_2 \ar[d]^{\sigma_2} \\ %
M_1 & M_2, \ar[l]_{\und{r}} }
\qquad
\xymatrix{E'_1 \ar[d]^{\sigma_1'} \ar@{-|>}[r]^{r'} & E'_2 \ar[d]^{\sigma_2'} \\ %
M_1' & M_2'. \ar[l]_{\und{r}'} }\end{split}
\end{gather}
We say that a pair $(\phi_1, \phi_2)$ of VB morphisms $\phi_i\colon E_i\to E_i'$ covering smooth maps $\und{\phi_i}\colon M_i\to M_i'$, $i=1,2$, is a \emph{$\catZM$-morphism from $r$ to $r'$} (and write $(\phi_1,\phi_2)\colon r\mZM r'$) if the following diagram
\begin{gather}\label{d:ZMtoZM}\begin{split}&
\xymatrix{
& E_1\ar[ld]_{\phi_1} \ar[dd]^<<<<<{\sigma_1} \ar@{-|>}[rr]^{r} && E_2 \ar[ld]_{\phi_2}\ar[dd]^{\sigma_2} \\
E_1'\ar[dd]_{\sigma_1'} \ar@{-|>}[rr]^<<<<<<<<<{r'} && E_2'\ar[dd]_>>>>>>>{\sigma_2'} & \\
& M_1\ar[ld]_{\und{\phi_1}} && M_2 \ar[ld]_{\und{\phi_2}} \ar[ll]_>>>>>{\und{r}} \\
M_1' && M_2' \ar[ll]_{\und{r}'} &
}\end{split}
\end{gather}
is commutative in the following sense:
\begin{enumerate}[(i)]
\item \label{item:def_morphism_ZM_1}The base maps commute, i.e., $\und{\phi_1}\circ \und{r} = \und{r}'\circ \und{\phi_2}\colon M_2\to M_1'$.
\item \label{item:def_morphism_ZM_2}At the level of fibres, for any $y\in M_2$ the following compositions of linear maps
\begin{gather*}
E_{1, x}\xrightarrow{r_{y}} E_{2,y}\xrightarrow{\phi_2} E_{2, y'}' \qquad\text{and}\qquad
E_{1, x}\xrightarrow{\phi_1} E_{1, x'}'\xrightarrow{r'_{y'}} E_{y'}'
\end{gather*}
coincide. Here $x=\und{r}(y)$, $y' = \und{\phi_2}(y)$, $x' = \und{r}'(y')=\und{\phi_1}(x)$ are determined by $y\in M_2$.
\end{enumerate}

It is easy to see that $\catZM$-morphisms can be naturally composed, and that \ZM s with $\catZM$-morphisms form a \emph{category of \Zakrzewski s}, denoted $\catZM$.
\end{df}

\begin{rem}[dualization of \ZM s] By dualizing the diagram~(\ref{d:ZMtoZM}) we get a diagram of the same type:
\begin{gather*}
\xymatrix{
& E_1^\ast \ar[dd]^<<<<<{\sigma_1^\ast} && E_2^\ast \ar[ll]_{r^\ast} \ar[dd]^{\sigma_2^\ast} \\
E_1^{'\ast}\ar@{-|>}[ru]^{\phi_1^\ast} \ar[dd]_{(\sigma_1')^\ast}  && E_2^{'\ast}\ar[ll]_<<<{r^{'\ast}} \ar@{-|>}[ru]^{\phi_2^\ast} \ar[dd]_>>>>>>>{(\sigma_2')^\ast} & \\
& M_1\ar[ld]_{\und{\phi_1}} && M_2 \ar[ld]_{\und{\phi_2}} \ar[ll]_>>>>>{\und{r}} \\
M_1' && M_2'. \ar[ll]_{\und{r}'} &
}
\end{gather*}
It is clear that $(\phi_1, \phi_2)\colon r\mZM r'$ is a $\catZM$-morphism if and only if $(r'^{\ast}, r^\ast)\colon \phi_2^\ast\mZM \phi_1^\ast$ is so. Thus the notion of a morphism in the category $\catZM$ does not reduce to the notion of a vector bundle morphism.
\end{rem}

\begin{ex}[phase lift] Let $\und{\phi}\colon M \ra M'$ be a smooth map, and consider the phase lift of $\und{\phi}$ which is a \ZM\ $\T^\ast\und{\phi}\colon \T^\ast M' \relto \T^\ast M$ covering $\und{\phi}$. Then the tangent lift of $\T^\ast\und{\phi}$ is a \ZM\ as well and $(\T \tau^\ast_{M'}, \tau_{\T^\ast M})\colon \T \T^\ast \und{\phi} \mZM \T^\ast\und{\phi}$ is a morphism in the category $\catZM$.
\end{ex}

Another useful formulation of the conditions presented in the definition of a $\catZM$-morphism is possible:
\begin{prop}[characterization of $\catZM$-morphisms]\label{p:ZM-morphism} Let $\phi_i\colon E_i\ra E_i'$ $(i=1, 2)$ be vector bundle morphisms and let $r, r'$ be \ZM s as in Definition~{\rm \ref{def:zm_morphism}}. Then $(\phi_1,\phi_2)\colon r\mZM r'$ is a $\catZM$-morphism if and only if for any $r$-related vectors $X\in E_1$, $Y\in E_2$ their images $\phi_1(X)$ and $\phi_2(Y)$ are $r'$-related. Equivalently, $(\phi_1\times \phi_2)(r)\subset r'$.
\end{prop}
\begin{proof}Assume that $(\phi_1,\phi_2)\colon r\mZM r'$ is a $\catZM$-morphism and take $X$ and $Y$ such that $X\sim_r Y$. Denote $x:=\sigma_1(X)$, $y:=\sigma_2(Y)$, $x':= \und{\phi_1}(x)$ and $y':= \und{\phi_2}(y)$. According to Definition~\ref{def:zm_morphism} $r'_{y'}(\phi_1(X)) = \phi_2(Y)$, in other words $\phi_1(X)$ and $\phi_2(Y)$ are $r'$-related. The proof the in reverse direction is analogous and is left to the reader.
\end{proof}

{\bf Linear \ZM s.} Of our special interest will be \emph{linear \ZM s}, i.e., \ZM s compatible with an additional linear structure. Later in Section~\ref{sec:gr_bund} (Definition~\ref{def:wZM}) we will generalize this notion to the concept of \emph{weighted \ZM s}, which are compatible with an additional graded structure. We shall follow a general scheme already presented in the definitions of a \emph{linear} Poisson structure, $\VB$-groupoids, $\VB$-algebroids, weighted algebroids etc.~\cite{BGG, Mackenzie_lie_2005}.

\begin{df}[linear \ZM ] \label{def:linear_ZM} Let $r$ be a \ZM\ from a vector bundle $\sigma\colon E\ra M$ to a vector bundle $ \sigma'\colon E'\ra M'$. Assume that the total spaces $E$ and $E'$ carry linear structures $\tau\colon E\ra N$ and $\tau'\colon E'\ra N'$ compatible with $\sigma$ and $\sigma'$, respectively (i.e., $E$ and $E'$ are DVBs in the sense of Definition~\ref{df:DVB}). We say that~$r$ is a~\emph{linear \ZM } if $r\subset E\times E'$ is a vector subbundle of $\tau\times\tau'$.
 \end{df}

Note that a linear \ZM\ $r\subset E\times E'$ defined as above is a vector subbundle of both $\sigma\times \sigma'$ and $\tau\times\tau '$. It follows that $r$ projects to linear relations in $M\times M'$ and $N\times N'$. In particular, the base map $\und{r}\colon M'\ra M$ is a VB morphism (cf.~\eqref{cond:rho_l} and~\eqref{cond:rho_r} in Lemma~\ref{lem:kappa_core} below). Particular examples of linear \ZM s (not all) are provided by dualizing DVB morphisms.

Recall that if $\sigma\colon E\ra M$ is a vector bundle, then the tangent space $\T E$ carries two compatible vector bundle structures $\T\sigma\colon \T E\ra\T M$ and $\tau_E\colon \T E\ra E$. In the remaining part of this paragraph we shall study linear \ZM s intertwining these two VB structures. The common kernel of  $\T \sigma$ and $\tau _E$ (called the \emph{core} $\Core(\T E)$ of DVB $(\T E, \T \sigma, \tau_E)$) is the
subbundle  $C=\V_M E \simeq E$ of the vertical bundle $\V E = \ker \T \sigma$ consisting of vertical vectors based at $M\subset E$. The two additive structures of $\T E$ coincide on the core $C$. Note that the core $C$  acts on $\T E$ by an addition of vertical vectors, i.e., for every $A\in \T_a E$ and every $e\in C_{\sigma(a)}\simeq E_{\sigma(a)}$ we define
\begin{gather*} A\plus e:=A+_{\tau_E}\V_a e\in \T_a E,\end{gather*}
where $\V_a\colon \V_{\sigma(a)} E \rightarrow \V_a E \subset \T_a E$ are canonical isomorphisms.
Note that the above action does not affects the two VB projections $\T\sigma$ and $\tau_E$ on $\T E$, i.e., $\T\sigma(A\plus e)=\T\sigma(A)$ and $\tau_E(A\plus e)=\tau_E(A)$. Moreover, the following useful identity holds
\begin{gather}\label{eqn:T_fa}
\T(f\cdot a) v = f\cdot_{\T\sigma} (\T a) v \plus v(f)\cdot a,
\end{gather}
where $f\in \Cf(M)$, $v\in \VF(M)$ and $a\in\Sec(E)$.

\begin{rem}[local form of a class of linear \ZM s]\label{rem:local_kappa_a}
Let us write a general (local) form of a linear \ZM\ $\kappa$ intertwining the two compatible VB structures on $\T E$. For simplicity and further applications we also assume that $\kappa$ projects onto $\graf(\id_M)$ under the bundle projection $\tau_M\circ \T\sigma \times \sigma \circ \tau_E\colon \T E\times \T E\ra M\times M$, i.e., $\kappa$ relates only elements in the same fibre over~$M$. (Note that $\tau_M\circ \T\sigma = \sigma \circ \tau_E$.) Let $(x^a, y^i)$ be local linear coordinates on $E$. We use the standard notation $(x^a, y^i,\dot x^b, \dot y^j)$ for induced coordinates on the tangent bundle $\T E$ and, to describe $\kappa \subset \T E\times \T E$, we shall underline coordinates in the second copy of $\T E$.

 As $\kappa$ is a double vector subbundle of $(\T E\times \T E, \T\sigma \times \tau_E, \tau_E\times \T \sigma)$, the graph of the base map $\rhoL:=\und{\kappa}\colon E\ra \T M$ of $\kappa$ is a vector subbundle in $E\times \T M$, i.e., it is a graph of a VB morphism. Thus it maps a vector $(\und{x}^a,\und{y}^i)$ from $E$ to a vector of the form $(x^a=\und{x}^a, \dot x^b=Q^b_i(x)\und{y}^i)$ in $\T M$. Now, since for every fixed $a\in E$ relation $\kappa_a\colon (\T E)_{\rhoL(a)}\ra \T_a E$ is a linear map, it turns out that $\kappa$ is given by a linear mapping of coordinates $(y^i,\dot{y}^j)$ to coordinates $(\und{\dot x}^b,\und{\dot y}^j)$ with coefficients depending on coordinates $(\und{x}^a,\und{y}^i)$ in $E$. Furthermore, we know that $\kappa$ is bi-homogeneous. Calculation on weights: $w(y^i)=(1,0)=w(\dot{\und{x}}^a)$, $w(\dot{x}^a)=(0,1)=w(\und{y}^i)$, $w(\dot{y}^i)=(1,1)=w(\dot{\und{y}}^i)$ ensures us that $\kappa$ is determined by equations of the form
\begin{gather}\label{eqn:kappa_local}
\kappa\colon \
\begin{cases}
x^a=\und{x}^a, \\
\dot{x}^b = Q^b_i(x)\, \und{y}^i, \\
 \dot{\und{x}}^b = \wt{Q}^b_i(x) y^i, \\
 \dot{\und{y}}^k = \alpha^k_i(x) \dot{y}^i + Q^k_{ij}(x) \und{y}^iy^j,
	\end{cases}
\end{gather}
where $x = (x^a)$ and $Q^b_i(x)$, $\wt{Q}^b_i(x)$, $\alpha^j_i(x)$ and $Q^k_{ij}(x)$ are some smooth functions defined locally on $M$.

 Note that if we additionally assume that $\kappa$ induces the identity on the core bundle, then the coefficients $\alpha^j_i(x)$ are simply constants $\delta^j_i$. We shall show later (see Lemma~\ref{lem:kappa_from_sigma}) that such linear \ZM s correspond to (Lie) algebroid structures. The anchor maps (left and right) are locally given by functions $Q^b_i(x)$ and $\wt{Q}^b_i(x)$, respectively, while functions $Q^k_{ij}(x)$ encode the bracket operation in a given basis of sections of the bundle $\sigma$, dual to functions $(y^i)$.
\end{rem}

\begin{lem}[a class of linear \ZM s]\label{lem:kappa_core} Let $\sigma\colon E\ra M$ be a vector bundle and let $\kappa\colon \T\sigma\relto \tau_E$ be a linear \ZM . Assume additionally that $\kappa$ induces the identity on the core bundles, i.e., $\kappa\cap(C\times C)=\graf(\id_C)$. Then:
\begin{enumerate}[$(i)$]\itemsep=0pt
	\item\label{cond:rho_l} The base map $\rhoL:=\und{\kappa}\colon E\ra \T M$ is a VB morphism covering the identity map $\id_M$.
	\item\label{cond:rho_r} Relation $\kappa':=(\tau_E\times\T\sigma)(\kappa)\subset E\times\T M$, which is the $\tau_E\times\T\sigma$-projection of $\kappa$, is a graph of a VB morphism $\rhoR\colon E\ra \T M$ covering the identity map $\id_M$.
	\item\label{cond:core_action} Relation $\kappa$ respects the action of the core $C$, i.e., if $(A,B)\in \kappa$ then
	$(A\plus c,B\plus c)\in\kappa$, whenever the addition of $c\in C$ makes sense.
	\item\label{cond:kappa_inverse} The inverse relation $\kappa^T$ is a linear \ZM\ from $\T\sigma$ to $\tau_E$ over the base map $\rhoR\colon E\ra \T M$ which also induces the identity on the cores.
\end{enumerate}
\end{lem}

\begin{proof}The assertion has a rather straightforward geometric justification based essentially on the bi-homogeneity of $\kappa$. For brevity, however, we prefer the following local argument.

Since $\kappa$ is the identity on the cores, we conclude that $\kappa\cap (M\times M) = \Delta_M\subset M\times M$ is a diagonal, and hence $\kappa$ relates only the elements in the same fibres over $M$. Now we are precisely in the situation described in Remark~\ref{rem:local_kappa_a} and the assertion follows easily from the local description~\eqref{eqn:kappa_local}. \end{proof}

\subsection{Description of (Lie) algebroids in terms of \ZM s}\label{ssec:alg_as_zm}
In this part we shall recall the definition of a general algebroid, and latter rephrase it in the language of \Zakrzewski s introduced above. A possibility of such a reformulation is of course well-recognized in the literature since the very introduction of the concept of a~general algebroid~\cite{JG_PU_Algebroids_1999} (see also~\cite{Leon_Marrero_Martinez_2005,Martinez_2008}). However, this topic was never systematically studied for its own sake. In particular, we are not aware that the axioms of a (Lie) algebroid were ever directly formulated in terms of the corresponding \ZM s. Despite this, we do not claim any originality in this area, as such a formulation is straightforward and natural. Our goal is rather to show the consistency and naturality of the approach to (Lie) algebroids based on differential relations (the notions of a subalgebroid, a morphism between algebroids, the Lie axiom, various specific types of algebroids, etc. are intrinsically defined within the category of \Zakrzewski s). In consequence, we prepare the ground for a later definition of a higher algebroid in Section~\ref{sec:high_alg}.

{\bf General algebroids.} General algebroids were introduced by Grabowski and Urba\'nski \mbox{\cite{JG_PU_Lie_alg_P-N_str_1997,JG_PU_Algebroids_1999}} as double vector bundle morphisms of a special kind. Their approach was motivated by the study of the geometry of mechanics and variational calculus originated by Tulczyjew \cite{Tulczyjew_1976,Tulczyjew_1976a}. Skew algebroids, almost-Lie algebroids and Lie algebroids may be regarded as special subclasses of this general notion.

\begin{df}[general algebroid]\label{def:algebroid}
A \emph{general algebroid} structure on a vector bundle \mbox{$\sigma\colon E\!\ra\! M$} is given by a bilinear \emph{bracket} $[\cdot,\cdot]$ on the space of smooth sections of $\sigma$, together with a pair of vector bundle maps (\emph{left} and \emph{right anchors}) $\rhoL,\rhoR\colon E\ra \T M$ over the identity on $M$ such that
\begin{gather}\label{eqn:bracket}
[f\cdot a,g\cdot b]=f\rhoL(a)(g)\cdot b-g\rhoR(b)(f)\cdot a+fg\cdot[a,b]
\end{gather}
for every sections $a,b\in\Sec_M(E)$ and every smooth functions $f,g\in C^\infty(M)$.

In addition
\begin{enumerate}[(i)]\itemsep=0pt
\item\label{cond:skew} If the bracket is skew-symmetric, i.e., $[a,b]=-[b,a]$ (in particular, left and right anchors coincide, i.e., $\rhoL=\rhoR=:\rho$ and we speak simply about the \emph{anchor}) we call $\sigma$ a \emph{skew algebroid}.
\item\label{cond:al} If $\sigma$ is a skew algebroid and the anchor $\rho$ maps the algebroid bracket $[\cdot,\cdot]$ to the Lie bracket of vector fields on $M$, i.e.,
\begin{gather}
\label{eqn:al_algebroid}
\rho[a,b]=[\rho(a),\rho(b)]_{\T M} ,
\end{gather}
we call $\sigma$ an \emph{almost-Lie algebroid} (AL algebroid, in short).
\item\label{cond:lie} If $\sigma$ is an almost-Lie algebroid and the bracket $[\cdot,\cdot]$ satisfies the Jacobi identity, we call $\sigma$ a \emph{Lie algebroid}.
\end{enumerate}
\end{df}

 We will refer to the objects from the above definition as to \emph{$($Lie$)$ algebroids}, i.e., writing that, say, $(\sigma,\rhoL,\rhoR,[\cdot,\cdot])$ is a (Lie) algebroid, means that it is either a Lie algebroid, an AL algebroid, a skew algebroid, or just a general algebroid. By using the word ``(Lie)'' we emphasize that all the above concepts were derived from the standard notion of the Lie algebroid by relaxing its axioms. Moreover, this should prevent possible confusions with Courant algebroid and other notions of an algebroid present in the literature.

{\bf From (Lie) algebroids to \ZM s.} We will now construct a canonical Zakrzewski morphism related with a given (Lie) algebroid structure. Recall Lemma~\ref{lem:kappa_core} describing the structure of a~class of linear \ZM s intertwining the two VB structures on the total space of the tangent bundle $\T E$ of a~vector bundle $\sigma\colon E\ra M$. It turns out that an algebroid structure on $\sigma$ induces such a~\ZM .

\begin{lem}[from (Lie) algebroids to \ZM s]\label{lem:kappa_from_sigma} Let $(\sigma\colon E\ra M,[\cdot,\cdot], \rhoL,\rhoR)$ be a $($Lie$)$ algebroid. Formula
\begin{gather}\label{eqn:kappa_from_bracket}
\kappa_a\left[\T b(\rhoL(a))\right]:=\T a(\rhoR(b))\plus [a,b],
\end{gather}
where $a,b\in \Sec_M(E)$ are arbitrary sections, extends to the unique relation $\kappa\subset\T E\times \T E$ such that
\begin{enumerate}[$(i)$]\itemsep=0pt
	\item\label{cond:kappa_biliner} $\kappa$ is a linear \Zakrzewski\ from $\T\sigma\colon \T E\ra \T M$ to the tangent bundle $\tau_E\colon \T E\ra E$ covering the left anchor $\rhoL\colon E\ra \T M$,
	\item\label{cond:core} $\kappa$ induces the identity on the core bundle, i.e., $\kappa\cap(C\times C)=\graf(\id_C)$.
\end{enumerate}
Moreover, the $(\tau_E\times \T\sigma)$-projection of $\kappa$ is the right anchor $\rhoR\colon E\ra \T M$.
\end{lem}
 The proof is given in Appendix~\ref{sec:app}.

\begin{rem}[the inverse of $\kappa$]\label{rem:kappa_T} Note that, according to Lemma~\ref{lem:kappa_core}\eqref{cond:kappa_inverse}, the inverse relation~$\kappa^T$ is also a linear
\ZM\ (this time over the right anchor $\rhoR$) inducing the identity on the cores.
It clearly satisfies (we use the properties \eqref{cond:kappa_biliner} and \eqref{cond:core} of $\kappa$ to transform equality \eqref{eqn:kappa_from_bracket}) the condition
\begin{gather*} \kappa_b^T\left[\T a(\rhoR(b))\right]=\T b(\rhoL(a))\minus [a,b],\end{gather*}
i.e., passing from $\kappa$ to $\kappa^T$ corresponds to changing the bracket $[a,b]$ to $[a,b]^T:=-[b,a]$ (note that this operation interchanges left and right anchors). Obviously $\kappa=\kappa^T$ if the bracket is skew-symmetric.
\end{rem}

\begin{ex}[the tangent algebroid] \label{ex:tang_alg}The tangent bundle $\tau_M\colon \T M \ra M$ of a manifold $M$ carries a canonical Lie algebroid structure. The left and right anchors are the identity maps $\rhoL=\rhoR=\id_{\T M}$ and the bracket $[\cdot,\cdot]_{\T M}$ is the commutator of vector fields.

In this case the corresponding \ZM\ is, in fact, an isomorphism of DVB $\kappa_M\colon \T\tau_M\ra \tau_{\T M}$. If $(x^a)$ are local coordinates on $M$, $(x^a, \dot x^b)$ induced coordinates on $\T M$ and $(x^a, \dot x^b, \delta x^c,\delta\dot x^d)$ induced coordinates on $\T\T M$, $\kappa_M$ is expressed as
\begin{gather*} \kappa_M\colon \ \big(x^a, \dot x^b, \delta x^c,\delta\dot x^d\big)\longmapsto \big(x^a, \delta x^b, \dot x^c,\delta\dot x^d\big).\end{gather*}
Note that $\kappa_M$ interchanges the two VB structures on $\T \T M$, $\T\tau_M$ and $\tau_{\T M}$.
\end{ex}

\begin{ex}[the tangent lift of a (Lie) algebroid]\label{ex:tang_lift} If $(\sigma\colon E\ra M, \rhoL,\rhoR,[\cdot,\cdot])$ is a (Lie) algebroid structure, then $\T\sigma\colon \T E\ra \T M$ also carries a canonical algebroid structure $(\T\sigma, \dd_\T\rhoL$, $\dd_\T\rhoR, [\cdot,\cdot]_{\dd_\T})$ called the \emph{tangent lift of the algebroid} structure on $\sigma$. This structure is determined by conditions{\samepage
\begin{gather*} \dd_\T\rhoL=\kappa_M\circ \T\rhoL,\qquad \dd_\T\rhoR=\kappa_M\circ \T\rhoR\qquad\text{and}\qquad [\T a,\T b]_{\dd_\T}=\T[a,b] \end{gather*}
for any sections $a,b\in\Sec_M(E)$.}

In this case (see \cite{JG_PU_Algebroids_1999}) the \ZM\ $\dd_\T\kappa$, corresponding to the considered algebroid structure on $\T\sigma$, is the tangent lift of the \ZM\ $\kappa$ corresponding to the initial algebroid structure on $\sigma$ composed with two canonical isomorphisms $\kappa_E$, i.e.,
\begin{gather*}\dd_\T\kappa=\kappa_E\circ\T \kappa \circ\kappa_E.\end{gather*}
 This construction has a natural generalization to higher tangent lifts $\TT{k}\sigma\colon \TT{k} E\ra\TT{k} M$. It will be discussed in detail in the second paragraph of Section~\ref{ssec:higher_lie}.
\end{ex}

{\bf From \ZM s to algebroids.} In fact, relation $\kappa$ introduced in Lemma~\ref{lem:kappa_from_sigma} completely characterizes the (Lie) algebroid structure on $\sigma$.

\begin{prop}[from \ZM s to (Lie) algebroids]\label{prop:sigma_from_kappa}Let $\sigma\colon E\ra M$ be a vector bundle. A~linear \Zakrzewski\ $\kappa$ from $\T\sigma\colon \T E\ra \T M$ to the tangent bundle $\tau_E\colon \T E\ra E$, which induces the identity on the core bundles, i.e., $\kappa\cap(C\times C)=\graf(\id_C)$, provides $\sigma$ with the unique general algebroid structure.

The left anchor $\rhoL$ $($resp., the right anchor $\rhoR)$ is given by the base map of $\kappa$ $($resp., the base map of $\kappa^T)$ and the bracket is given by
\begin{gather}\label{eqn:bracket_from_kappa}
\V_a[a,b]:=\kappa_a\left[\T b(\rhoL(a))\right]-\T a(\rhoR(b)),
\end{gather}
for any sections $a,b\in\Sec_M(E)$.
\end{prop}
 The proof is given in Appendix~\ref{sec:app}.

\begin{rem}[alternative definition of a general algebroid] By the results of Proposition~\ref{prop:sigma_from_kappa} and Lemma~\ref{lem:kappa_from_sigma} we can equivalently define a general algebroid structure $(\sigma,\rhoL,\rhoR,[\cdot,\cdot])$ on $\sigma\colon E\ra M$ as a pair $(\sigma,\kappa)$, where $\kappa\colon \T\sigma\relto\tau_E$ is a linear \ZM\ satisfying natural properties. In what follows we shall often refer to this description.
\end{rem}

\begin{rem}[the dual of $\kappa$]\label{rem:kappa_dual} The fact that $\kappa$ induces the identity on the core implies that its dual (which is a proper vector bundle morphism, and even a DVB morphism by the linearity of~$\kappa$)
\begin{gather*}\xymatrix{
\T^\ast E \ar[d]_{\tau_E^*} \ar[r]^{\kappa^\ast} & \T E^\ast\ar[d]^{\T \sigma^\ast}\\
E \ar[r]^{\rhoL} & \T M
}
\end{gather*}
covers the identity $\id_{E^\ast}$ under the projections $\T^\ast E\ra E^\ast$ and $\T E^\ast\ra E^\ast$ (the core of a DVB becomes a side bundle under dualization \cite{Konieczna_Urbanski_1999}). In other words, $\kappa^\ast\colon \T E^\ast\ra \T^\ast E\simeq \T^\ast E^\ast$ corresponds to a linear bi-vector on~$E^\ast$. This is an original point of view of~\cite{JG_PU_Lie_alg_P-N_str_1997}.
\end{rem}

\begin{rem}[a local form of $\kappa$] By Remark~\ref{rem:local_kappa_a}, the local form of a linear \ZM\ corresponding to a given general algebroid structure is given by formulas \eqref{eqn:kappa_local} with $\alpha^k_i(x) = \delta^k_i$. Within this description $\rhoL\colon E\ra \T M$, the base map of $\kappa$, is given by $\rhoL\colon E\ni(\und{x}^a,\und{y}^i)\longmapsto(x^a=\und{x}^a$, $\dot x^b=Q^b_i(x)\und{y}^i)\in\T M$. Similarly, the right anchor reads as $\rhoR\colon E\ni(x^a,y^i)\longmapsto(\und{x}^a=x^a$, $\dot{\und{x}}^b=\wt{Q}^b_i(x)y^i)\in\T M$.

Take now local sections $a(x)\sim(x^a,y^j=a^i(x))$, and $b(x)\sim(x^a,y^j=b^i(x))$ of $E$. We can use formula \eqref{eqn:kappa_local} together with \eqref{eqn:bracket_from_kappa} to calculate a local expression for an algebroid bracket of these two sections. Simple calculations (which we omit here) lead to
\begin{gather*} [a,b](x)\sim\left(x^a,y^k=\frac{\pa b^k}{\pa x^a}(x)Q^a_i(x)a^i(x)-\frac{\pa a^k}{\pa x^a}(x)\wt{Q}^a_j(x)b^j(x)+Q^k_{ij}a^i(x)b^j(x)\right).\end{gather*}
\end{rem}

{\bf (Lie) algebroid morphisms as $\boldsymbol{\catZM}$-morphisms.} Next we shall show that within the interpretation of (Lie) algebroids as \Zakrzewski s, the notion of an algebroid morphism corresponds to a $\catZM$-morphism between appropriate relations.

Intuitively, a morphism between two (Lie) algebroid structures on $\sigma\colon E\ra M$ and on \mbox{$\sigma'\colon E'\ra M'$} should be a vector bundle map $\phi\colon E\ra E'$ over $\und{\phi}\colon M\ra M'$ which intertwines the anchors and the algebroid bracket on sections. This intuition, however, faces immediate problems as, in general, a VB morphism $\phi$ does not map sections of $\sigma$ to sections of~$\sigma'$. This problem is solved by passing to the pull-back bundles.

\begin{df}[morphism of (Lie) algebroids, \cite{Mackenzie_lie_2005}]\label{def:algebroid_morphism}
Let $(\sigma\colon E \ra M,[\cdot,\cdot],\rhoL,\rhoR)$ and\linebreak $(\sigma'\colon E'\ra M',[\cdot,\cdot]',\rhoL',\rhoR')$ be (Lie) algebroids. A VB morphism $\phi\colon E\ra E'$ over $\und{\phi}\colon M\ra M'$ is a~\emph{morphism} between the algebroid structures on $\sigma$ and $\sigma'$ if and only if the VB morphisms~$\phi$ and~$\T\und{\phi}$ relate the left and the right anchors of $\sigma$ and $\sigma'$, i.e.,
\begin{gather}\label{eqn:compatible_anchors}\begin{split} &
\xymatrix{E \ar[r]^{\phi}\ar[d]^{\rhoL}& E'\ar[d]^{\rhoL'}\\
\T M \ar[r]^{\T\und{\phi}}& \T M'
}\end{split} \qquad\text{and}\qquad \begin{split}&
\xymatrix{E \ar[r]^{\phi}\ar[d]^{\rhoR}& E'\ar[d]^{\rhoR'}\\
\T M \ar[r]^{\T\und{\phi}}& \T M',}\end{split}
\end{gather}
and if for every sections $a,b\in\Sec_M(E)$ such that their push-forwards $\phi_\ast a,\phi_\ast b\in \Sec_{M}(\und{\phi}^\ast E')$ can be (locally) presented as finite sums $\phi_\ast a=\sum_i f_i\cdot \und{\phi}^\ast a_i$ and $\phi_\ast b=\sum_j g_j\cdot \und{\phi}^\ast b_j$ for some functions $f_i, g_j\in \Cf(M)$ and some sections $a_i,b_j\in \Sec_{M'}(E')$ we have
\begin{gather}\label{eqn:algebroid_morphism}
\phi_\ast[a,b]=\sum_{i,j}\left(\rhoL(a)(g_j)\cdot\und{\phi}^\ast b_j-\rhoR(b)(f_i)\cdot\und{\phi}^\ast a_i+f_ig_j\cdot\und{\phi}^\ast[a_i,b_j]'\right).
\end{gather}
\end{df}
\begin{rem}[on notion of an algebroid morphism] For any $x\in M$ there is a neighbourhood~$U$ of $x$ in~$M$ and a neighbourhood~$U'$ of $\und{\phi}(x)$ in $M'$ such that for any section $a$ of $\sigma$ we have $\phi_\ast a = \sum_i f_i \und{\phi}^\ast e_i'$ where $(e_i')$ is a basis of sections of $\sigma'$ over $U'$ and $f_i$ are some functions on~$U$. This explains a local character of the above definition. There are other equivalent and more elegant formulation of the definition of an algebroid morphism -- we discus some of them in Proposition~\ref{prop:alg_morphism_eqv} (see also \cite[Definition~3]{JG_PU_Algebroids_1999}, where general algebroids are considered as a ~special type of Leibniz structures on the dual bundle and~\cite{JG_mod_class_skew_alg_rel_2012} for a generalization of this notion to an algebroid relation).
\end{rem}

It may seem unclear if condition \eqref{eqn:algebroid_morphism} is well-posed, i.e., if it does not depend on the presentation of $\phi_\ast a$ and $\phi_\ast b$ as finite sums of sections with $\Cf(M)$ coefficients. To prove that this is the case, we check that the right-hand side of \eqref{eqn:algebroid_morphism} is tensorial with respect to $a_i$ and $b_j$ (for this it is crucial that condition \eqref{eqn:compatible_anchors} holds). The essential calculations (for a simpler case of a skew algebroid) can be found in the classical book \cite{Mackenzie_lie_2005}.

It turns out that the above definition has a very elegant (and much simpler) interpretation in the language of \ZM s naturally related with the (Lie) algebroid structure.
\begin{prop}[(Lie) algebroid morphisms as $\catZM$-morphisms]\label{prop:alg_morphism}
Let $(\sigma\colon E\!\ra\! M,[\cdot,\cdot],\rhoL,\rhoR)$ and $(\sigma'\colon E' \ra M',[\cdot,\cdot]',\rhoL',\rhoR')$ be $($Lie$)$ algebroids with the corresponding \Zakrzewski s {$\kappa\colon \T\sigma\relto \tau_{E}$} and $\kappa'\colon \T\sigma'\relto \tau_{E'}$, respectively. A VB map $\phi\colon E\ra E'$ over $\und{\phi}\colon M\ra M'$ is a~morphism between the algebroid structures on~$\sigma$ and $\sigma'$ if and only if $(\T\phi,\T\phi)$ is a $\catZM$-morphism from $\kappa$ to $\kappa'$.
\end{prop}
The proof is given in Appendix~\ref{sec:app}.

{\bf Subalgebroids and algebroidal relations.} Let us now describe the notion of a subalgebroid in terms of \ZM s. Recall the following definition.

\begin{df}[subalgebroid, \cite{JG_mod_class_skew_alg_rel_2012, Mackenzie_lie_2005}]\label{def:sublagebroid}
Let $(\sigma\colon E\ra M, \rhoL,\rhoR, [\cdot, \cdot])$ be a (Lie) algebroid. A~vector subbundle $\sigma'\colon E'\ra M'$ of $\sigma$ is called a \emph{subalgebroid} of $\sigma$ if the following two conditions are satisfied
\begin{enumerate}[(i)]\itemsep=0pt
 \item \label{df:sub_algebroid_1} The restrictions of the anchors $\rhoL|_{E'}$ and $\rhoR|_{E'}$ map $E'\subset E$ to $\T M'\subset \T M$.
 \item \label{df:sub_algebroid_2} If sections $a,b\in\Sec_M(E)$ are such that $a|_{M'},b|_{M'} \in \Sec_{M'}(E')$, then $[a, b]|_{M'} \in \Sec_{M'}(E')$.
\end{enumerate}
\end{df}

The first of the above conditions assures us that the bracket operation $[\cdot, \cdot]_{E'}$ defined on sections of $\sigma'$ by $[\wt{a}, \wt{b}]_{E'}:= [a, b]|_{M'}$ does not depend on the choice of the extensions $a,b\in \Sec_M(E)$ of sections $\wt{a},\wt{b}\in\Sec_{M'}(E')$. The second condition guarantees that the section space $\Sec_{M'}(E')$ is closed with respect to this bracket.
Clearly the subbundle $\sigma'$ carries a general algebroid structure inherited from $\sigma$.

In the face of the relationship between (Lie) algebroids and \ZM s, the \ZM\ \mbox{$\kappa\colon \T\sigma\!\relto \!\tau_E$} corresponding to the algebroid structure on $\sigma$ should induce some \ZM\ $\kappa'\colon \T\sigma'\relto \tau_{E'}$ corresponding to the structure of a subalgebroid on $\sigma'$ described above. We claim that such a~$\kappa'$ is a~\emph{fine restriction} of $\kappa$ in the sense of the definition below.

\begin{df}[fine restriction] Let, for $i=1,2$, $\sigma_i'\colon E_i'\ra M_i'$ be a vector subbundle of $\sigma_i\colon E_i\ra M_i$, and let $r\colon \sigma_1\relto \sigma_2$ be a \ZM\ over a base map $\und{r}\colon M_2\ra M_1$. We say that~$r$ \emph{restricts fine to $\sigma_1'\times \sigma_2'$} if $\und{r}(M_2')\subset M_1'$ and if for any $X\in E_1$ and $Y\in E_2$ that are $r$-related and such that $X\in E_1'$ while $\sigma_2(Y)\in M_2'$ we have $Y\in E_2'$. If this is the case, relation $r'$, defined as the intersection of $r$ with $E_1'\times E_2'$, defines a \ZM\ from $\sigma_1'$ to $\sigma_2'$.
\end{df}

The equivalence of the classical notion of a subalgebroid with the notion of a fine restriction of the corresponding \ZM\ can be easily proved.

\begin{prop}[on subalgebroids]\label{p:subalgebroid} Let $(\sigma\colon\! E\!\ra\! M, \kappa)$ be a $($Lie$)$ algebroid and let \mbox{$\sigma'\colon\! E'\!\ra\! M'$} be a vector subbundle of $\sigma$. Then the following conditions are equivalent:
\begin{enumerate}[$(i)$]\itemsep=0pt
	\item\label{cond:subalg_1} The subbundle $\sigma'$ is a subalgebroid of $\sigma$.
	\item\label{cond:subalg_2} The inclusion map $\iota\colon \sigma'\hookrightarrow\sigma$ is an algebroid morphism.
	\item\label{cond:subalg_3} The \ZM\ $\kappa$ restricts fine to $\T\sigma'\times\tau_{E'}\subset \T\sigma\times\tau_{E}$.
\end{enumerate}
\end{prop}
\begin{proof}It involves some elementary diagram-chasing to check that \eqref{cond:subalg_3} is equivalent to $(\T\iota,\T\iota)$ being a $\catZM$-morphism between $\kappa'$ and $\kappa$. Due to Proposition~\ref{prop:alg_morphism} the latter condition is equivalent to~\eqref{cond:subalg_2}. Finally the equivalence of conditions \eqref{cond:subalg_1} and \eqref{cond:subalg_2} is a standard fact in the theory of Lie algebroids (see \cite[Chapter~4]{Mackenzie_lie_2005}).
\end{proof}

Following Grabowski~\cite{JG_mod_class_skew_alg_rel_2012}, we recall the concept of an algebroidal relation, which is a generalization of a morphism of algebroids. It is closely related to the notion of a subalgebroid.

\begin{df}[algebroidal relation]\label{df:alg_relation} Let $(\sigma_1\colon E_1\ra M_1,\kappa_1)$ and $(\sigma_2\colon E_2\ra M_2,\kappa_2)$ be (Lie) algebroids. A relation $r\subset E_1 \times E_2$ is called an \emph{algebroidal relation} if the graph of~$r$ is a~subalgebroid of the product algebroid $(\sigma_1\times \sigma_2,\kappa_1\times\kappa_2)$.
\end{df}
If $r$ is a vector bundle morphism, we recover the notion of an algebroid morphism (see Proposition~\ref{prop:alg_morphism_eqv}). Algebroidal relations have an elegant characterization in terms of the anchors and the algebroid brackets.

\begin{prop}[on algebroidal relations]\label{prop:algebroidal_relation} Let $(\sigma_1\colon E_1\ra M_1,\rho_{1L}, \rho_{1R},[\cdot,\cdot]_1)\simeq(\sigma_1,\kappa_1)$ and $(\sigma_2\colon E_2\ra M_2,\rho_{2L}, \rho_{2R},[\cdot,\cdot]_2)\simeq(\sigma_2,\kappa_2)$ be (Lie) algebroids and let $r\colon \sigma_1\relto \sigma_2$ be a \ZM\ covering $\und{r}\colon M_2\ra M_1$.
The following are equivalent:
\begin{enumerate}[$(i)$]\itemsep=0pt
\item\label{cond:alg_rel_1} $r$ is an algebroidal relation,
\item\label{cond:alg_rel_2}\begin{enumerate}[$(a)$]\itemsep=0pt
\item\label{sub_cond:a} For any section $s_1\in\Sec_{M_1}(E_1)$ and $I=L,R$, the vector fields $\rho_{2I} (\hat{r}(s_1))$ and $\rho_{1I}(s_1)$ are $\und{r}$-related and
\item\label{sub_cond:b} for any sections $s_1, s_1'\in \Sec_{M_1}(E_1)$ we have \begin{gather*} \hat{r}([s_1, s_1']_1) = [\hat{r}(s_1), \hat{r}(s_1')]_{2}.\end{gather*}
\end{enumerate}
\end{enumerate}
\end{prop}
The proof is given in Appendix~\ref{sec:app}.
\begin{rem} We stress that although the anchor map is uniquely determined by the bracket operation in any general algebroid, the above condition \eqref{sub_cond:a} does not follow from \eqref{sub_cond:b}. A simple counterexample is provided by the zero endomorphism in a general algebroid.
\end{rem}

The following result states that fine restrictions of algebroidal relations to subalgebroids remain algebroidal relations.
\begin{prop}[restricting algebroidal relations]\label{prop:sub_r_also_algebroidal}
Let, for $j=1,2$, $(\sigma_j'\colon E_j'\ra M'_j, \kappa_j')$ be a~subalgebroid of $(\sigma_j\colon E_j\ra M_j, \kappa_j)$. Let us assume that a \ZM
\begin{gather*}
 \xymatrix{E_1 \ar[d]^{\sigma_1} \ar@{-|>}[r]^{r} & E_2 \ar[d]^{\sigma_2} \\
M_1 & M_2 \ar[l]_{\und{r}} }
\end{gather*}
restricts fine to $r'\subset E_1'\times E_2'$. Then if $r$ is an algebroidal relation then so is $r'$.
\end{prop}
\begin{proof} We should check that $r'$ is a subalgebroid of $(\sigma_1'\times \sigma_2',\kappa_1'\times\kappa_2')$. Since $r$ restricts fine to $E_1'\times E_2'$, the relation $r'$ is a vector subbundle of $E_1'\times E_2'$. Let us take $X, Y\in \T E_1'\times \T E_2'$ such that $X,Y$ are $\kappa_1'\times \kappa_2'$-related and $X\in \T r'$ while $Y$ lies over $y\in r'\subset E_1'\times E_2'$. We have to show that $Y\in\T r'$.

First note that $Y\in \T r$, because $r$ respects algebroid structures of $E_1$ and $E_2$, and $X, Y$ are also $\kappa_1\times \kappa_2$-related as $E_1'\times E_2'$ is a subalgebroid of $E_1\times E_2$, and $X\in \T r$ while $y\in r'\subset r$. Next, by \cite[Theorem~A.5]{MJ_MR_models_high_alg_2015}, $\T r$ restricts fine to $\T E_1'\times \T E_2'$. Therefore, since $Y$ belongs also to $\T E_1'\times \T E_2'$, we get $Y\in \T r' = \T r\cap (\T E_1'\times \T E_2')$ as we claimed.
\end{proof}

{\bf Characterization of algebroids with an additional structure.} We shall now express specific conditions \eqref{cond:skew}--\eqref{cond:lie} from Definition~\ref{def:algebroid} in terms of the corresponding \ZM\ $\kappa$.

\begin{lem}[characterization of various algebroids]\label{lem:char_kappa}
Let $(\sigma\colon E\ra M,[\cdot,\cdot],\rhoL,\rhoR)$ be a general algebroid structure on $\sigma$ and let $\kappa\colon \T \sigma\relto\tau_E$ be the corresponding \Zakrzewski . Then $\sigma$ is
\begin{enumerate}[$(i)$]\itemsep=0pt
	\item\label{item:kappa_skew} A skew algebroid if and only if $\kappa$ is symmetric, i.e., if $(X,Y)\in \kappa$ then $(Y, X)\in \kappa$ $($equivalently $\kappa=\kappa^T)$.
	\item \label{item:kappa_al}An almost-Lie algebroid if and only if it is skew, and $(\T \rho, \T \rho)\colon \kappa \mZM \kappa_M$ is a morphism of \ZM s $($in the sense of Definition~{\rm \ref{def:zm_morphism})}:
\begin{gather}\label{eqn:kappa_AL}\begin{split}&
\xymatrix{
\T E \ar[d]_{\T\rho} \ar@{-|>}[rr]^\kappa && \T E\ar[d]^{\T \rho}\\
\T\T M \ar[rr]^{\kappa_M} && \T\T M.
}\end{split}
\end{gather}
Here $\rho=\rhoL=\rhoR$ is the anchor map.
	\item \label{item:kappa_lie}A Lie algebroid if and only if it is almost-Lie and $\kappa\subset\T E\times \T E$ is a subalgebroid of the product algebroid $(\T\sigma\times \tau_E, \dd_\T\kappa\times\kappa_E)$. In other words, $\kappa\colon\T\sigma\relto\tau_E$ is an algebroidal relation.
\end{enumerate}
\end{lem}
The proof is given in Appendix~\ref{sec:app}.

{\bf An application -- prolongations of AL algebroids.}
Throughout this section we argued that \ZM s provide a consistent language to describe (Lie) algebroids, alternative to the standard treatment of the topic. Besides, some known constructions in the theory of (Lie) algebroids have more evident definitions in the language of \Zakrzewski . To justify this claim we shall now give a non-standard definition of a prolongation of an AL algebroid over a~fibration~\mbox{\cite{Cortes_Inn_survey_2006, Higgins_Mackenzie_1990}}, the crucial notion in the Lagrangian and Hamiltonian formalisms for mechanics on Lie algebroids developed by Mart\'{i}nez in \cite{Martinez_geom_form_mech_lie_alg_2001,Martinez_lagr_mech_lie_alg_2001}.
 \begin{ex}[prolongation of an AL algebroid over a fibration]\label{ex:prolong} Let $\pi\colon P\ra M$ be a fibration and let $(\sigma\colon E\ra M, \kappa)$ be an almost-Lie algebroid. Denote by $\rho\colon E\ra \T M$ the related anchor map. Then $\Tau^EP := E \times_{\T M} \T P = \{(e, v)\in E\times \T P\colon \rho(e) = \T \pi (v)\}$ is a vector bundle over~$P$ which carries an almost-Lie algebroid structure defined by the restriction of the relation $\kappa\times \kappa_P\colon \T E\times \T\T P\relto \T E\times \T\T P$ to $\T (\Tau^E P)\simeq \T E\times_{\T\T M}\T\T P$:
\begin{gather*}
 \xymatrix{
 \T E \times_{\T \T M} \T\T P \ar[d] \ar@{-|>}[rr]^{\kappa\times \kappa_P} && \T E \times_{\T \T M} \T\T P \ar[d] \\
 \T P && E\times_{\T M} \T P. \ar[ll]
 }
 \end{gather*}

 Indeed, the only thing to check is that $\Tau^EP$ is a subalgebroid of the product algebroid $(\sigma\times \tau_P\colon E \times \T P \ra M \times P, \kappa\times \kappa_P)$. This is straightforward, given $X, X'\in \T E$ and $Y, Y'\in \T \T P$ such that $(X, X')\in \kappa$, $Y'=\kappa_P(Y)$, and $(X, Y)$ is tangent to $\Tau^EP$, we need to prove that $(X', Y')$ is also tangent to $\Tau^EP$, i.e., that
 $\T \rho(X') = \T \T \pi(Y')$. But $\T \rho(X') = \kappa_M(\T \rho (X))$ while $\T \T \pi(Y')= \kappa_M(\T \T \pi (Y))$ (due to the assumption that $\kappa$ is almost-Lie) and the result follows.
 \end{ex}
Another non-standard characterization of the notion of the prolongation, also emphasizing the role of the AL axiom \eqref{eqn:al_algebroid}, was recently provided by one of us in \cite[Proposition~3.1]{MJ_tt_vs_prolong}.

\section{Recollection of $\N$-graded manifolds}\label{sec:gr_bund}

{\bf \grBs\ and homogeneity structures.} \emph{Higher order algebroids} which we shall introduce and study in the next section are modeled on geometric objects which generalize the notion of a vector bundle. Such a generalization (catching, in particular, the canonical graded structure of the higher tangent bundle~$\TT{k} M$ -- a fundamental example from the point of view of physical applications) was first proposed by Voronov in~\cite{Voronov_2002}, within the framework of supergeometry, as non-negatively graded ($\N$-graded, for short) manifolds. Voronov noticed that they expand into a~tower of fibrations and suggested to see a non-negative grading as a~generalization of a linear structure. In this paper we choose to work with \emph{purely even non-negatively graded manifolds}, a particular subclass of Voronov's objects. In \cite{JG_MR_gr_bund_hgm_str_2011} they are referred to as \emph{graded bundles} (see also the introduction to~\cite{BGR}). Here we shall recall basic properties and constructions associated with these objects.

An important example is the \emph{higher $(\thh{k}$-order$)$ tangent bundle $\tau^k_M\colon \TT{k} M\ra M$} of a mani\-fold~$M$, consisting of $\thh{k}$-order tangency classes (called \emph {$k$-velocities}) of curves in~$M$. Bundle $\TT{1} M=\T M$ is just the tangent bundle of $M$, however for $k>1$, $\tau^k_M$ is no longer a vector bundle. We shall see this at the elementary level.

Given a smooth function $f$ on a manifold $M$ and an integer $\alpha=0,1,\hdots,k$ one can construct a function $f^{(\alpha)}$ on $\TT k M$, the so-called \emph{$\alpha$-lift of $f$} (see~\cite{Morimoto_Lifts}). It is defined by
\begin{gather*}
f^{(\alpha)}(\tclass{k}{\gamma}):= \left.\frac{\dd^\alpha}{\dd t^\alpha}\right|_{t=0} f(\gamma(t)),
\end{gather*}
where $\tclass{k}{\gamma}$ denotes the $k$-{jet} of a curve $\gamma\colon \R \ra M$ at zero. We shall usually write $\dot{f}$, $\ddot{f}$ instead of~$f^{(1)}$,~$f^{(2)}$, respectively. The \emph{adapted coordinates}~$(x^{a}, \dot{x}^a, \ddot{x}^a, \ldots)$ for $\TT{k} M$, induced by coordinates $(x^a)$ on~$M$, are obtained by applying the above lifting procedure to coordinate functions~$x^a$, and are naturally graded. (We simply assign \emph{weight}~$k$ to coordinates \mbox{$x^{a, (k)} := (x^a)^{(k)}$}.) In particular, on $\TT{2} M$ they transform as
\begin{gather}\label{e:T2M}
x^{a'}=x^{a'}(x), \qquad \dot{x}^{a'} = \frac{\partial x^{a'}}{\pa x^b} \dot{x}^b, \qquad
\ddot{x}^{a'} = \frac{\pa x^{a'}}{\pa x^b} \ddot{x}^b + \frac{\pa^2 x^{a'}}{\pa x^b\pa x^c}\dot{x}^b \dot{x}^c
\end{gather}
thus the weight of both left and right sides of the above equalities is the same. From a geo\-metric point of view, fibres of $\tau^k_M$ are equipped with a~special structure. Namely, we have a~canonical action of the multiplicative monoid of real numbers $(\R, \cdot)$ on these fibres defined by re-parametrizing curves representing the elements of $\TT{k} M$:
\begin{gather*}
h\colon \ \R \times \TT{k} M \ra \TT{k} M, \qquad \big(t, \tclass{k}{\gamma}\big) \mapsto \tclass{k}{\gamma_t},
\end{gather*}
where $\gamma_t(s)=\gamma(ts)$. We clearly see, that the $\alpha$-lift of a function~$f$ on~$M$ is a \emph{homogeneous} function on~$\TT{k} M$ of \emph{weight} $\alpha$, i.e.,{\samepage
\begin{gather*}
f^{(\alpha)}(h_t(v)) = t^\alpha f^{(\alpha)}(v),
\end{gather*}
where $h_t(v) := h(t, v)$. This also explains why the adapted coordinates on $\TT{k} M$ are graded.}

Properties of the higher tangent bundle $\TT{k} M$ motivate the concepts of a non-negatively graded manifold and a homogeneity structure. The definition below is equivalent to the definition of a purely even positively-graded manifold according to Voronov \cite[Definition~4.1]{Voronov_2002}. Here we reformulate the condition of ``cylindricity'' of positive even coordinates from that definition as the condition that fibers are diffeomorphic to~$\R^n$.

\begin{df}[$\N$-graded manifold \cite{Voronov_2002}]\label{def:gr_bund}
 A \emph{$\N$-graded manifold} is a smooth fibration \mbox{$\sigma^k\colon\! E^k{\ra} M$} in which we are given a distinguished class of fibre coordinates (called \emph{graded coordinates}) with non-negative integer \emph{weights} assigned. Moreover, it is assumed that these graded coordinates identify the fibres with~$\R^n$ (for some integer~$n$), and that the transition functions are multi-variable polynomials that preserve the weights. The index $k$ in $E^k$ indicates that the weights on $E^k$ are less or equal~$k$. We say that~$E^k$ has \emph{\degree}~$k$.

A \emph{homogeneity structure} (the idea of which dates back at least to \cite{Severa_2005}) is a manifold~$E$ equipped with a smooth action $h\colon \R\times E\ra E$ of the multiplicative monoid of real num\-bers~$(\R, \cdot)$. Surprisingly, both concept coincide in the smooth setting~\cite[Theorem~4.2]{JG_MR_gr_bund_hgm_str_2011}, in particular, every \grB\ $\sigma^k\colon E^k\ra M$ admits a canonical homogeneity structure \mbox{$h^{E^k}\colon \R\times E^k\ra E^k$}. For this reason we shall use the terms \grB\ and homogeneity structure interchangeably.
\end{df}

Graded coordinates on a \grB\ $\sigma^k\colon E^k \ra M$ can be denoted by $(x^a, y^i_w)$ where a (superfluous) index $w=w(i)\in \Z_{>0}$ at $y^i$ indicates that a fibre coordinate~$y^i$ is homogeneous of weight~$w$. The base coordinates $(x^a)$ are assumed to have weight zero. In such a notation the associated homogeneity structure reads as
\begin{gather*}
h\colon \ \R \times E^k \ra E^k, \qquad h\big(t, \big(x^a, y^i_w\big)\big) = \big( x^a, t^w y^i_w\big), \qquad t\in \R.
\end{gather*}
It is convenient and fruitful to encode the structure of a \grB\ by means of a~canonical vector field on $E^k$ called the \emph{weight vector field} which in graded coordinates is given by
\begin{gather*}
 \Delta = \sum_i w\, y^i_w \frac{\pa}{\pa y^i_w}.
\end{gather*}
 For example, the canonical weight vector field on $\TT{k} M$ is
$ \Delta_M^k = \sum_{\alpha = 1}^k \sum_a \alpha x^{a, (\alpha)} \pa_{x^{a, (\alpha)}}$. In fact, the weight vector field provides an equivalent characterization of the structure of a \grB. For brevity, we often refer to a \grB\ as to a pair $(E^k, \Delta)$.

 A \emph{morphism} from a \grB\ $\sigma_E\colon E^k\ra M$ to $\sigma_F\colon F^k \ra M$ is a smooth map $\phi\colon E^k\ra F^k$ commuting with the respective homogeneity structures, i.e., $\phi \circ h^E = h^F\circ \phi$. Equivalently $\phi$ relates the corresponding weight vector fields $\Delta_E$ and $\Delta_F$. \grBs\ of \degree\ $k$ form a \emph{category} denoted by $\catGB[k]$. In view of \cite[Theorem 2.4]{JG_MR_hi_VB_2009}, $\catGB[1]$ is the category of vector bundles.

Every \grB\ has a \emph{zero section} $0_M^k\colon M\ra E^k$. The image of $M$ in $E^k$ is defined by \emph{putting to zero} \cite[Example 1.10]{BGR} all graded coordinates of positive weight. More generally, one can invariantly put to zero all \emph{fibre} coordinates of weights less than a given number $j$. The obtained subset of $E^k$ is denoted by $E^k[\Delta \geq j]$ and it is, actually, a graded subbundle of~$E^k$.

Given a \grB\ $\big(E^k, \Delta\big)$ of \degree\ $k$ and an integer $0\leq j\leq k$ one can
construct a~canonical projection from $E^k$ onto a \grB\ of \degree\ $j$, denoted by \mbox{$E^k[\Delta\leq j]$}, obtained by \emph{removing all coordinates of weights greater than}~$j$ \cite[Definition~1.6]{BGR}. As transformation rules for $E^k$ of coordinates of weight $\leq j$ involve only coordinates of weights $\leq j$ the above construction is correct. It follows that
a \grB\ $\sigma^k\colon E^k\to M$ induces a~\emph{tower of affine fibrations}~\cite{Voronov_2002}
\begin{gather}\label{e:seq_aff_fib}
E^k\xrightarrow{\sigma^k_{k-1}} E^{k-1}\to\cdots \xrightarrow{\sigma^2_1} E^1\xrightarrow{\sigma^1_0} M,
\end{gather}
where $E^k[\Delta\leq j]$ is denoted shortly by $E^j$.

\begin{df}[the top core of a \grB]
Let $\sigma^k\colon E^k\ra M$ be a \grB\ of \degree\ $k$. The \emph{top core} of $E^k$, denoted by $\core{E^k}$, is the set
\begin{gather*}
\core{E^k} = \big\{ v\in E^k\colon \sigma^k_{k-1} (v) \in 0_M^{k-1}(M) \subset E^{k-1} \big\},
\end{gather*}
where $\sigma^k_{k-1}\colon E^k\ra E^{k-1}$ as in~\eqref{e:seq_aff_fib}. Locally, $\core{E^k}$ is defined by \emph{putting to zero} all fibre coordinates of weights less than $k$:
\begin{gather*}
 \core{E^k} = \big\{\big(x^A, y^a_w\big)\colon y^a_w=0 \text{ for any } y^a_w \text{ such that } 1\leq w < k\big\} \subset E^k,
 \end{gather*}
hence $\big(x^A, y^a_k\big)$ are local coordinates on $\core{E^k}$. The top core $\core{E^k}$ is naturally a vector bundle over~$M$ with the homotheties defined locally by
\begin{gather*} t.\big(x^A, y^a_k\big) = \big(x^A, t\cdot y^a_k\big)
\end{gather*}
for $t\in\R$. Moreover, the top core $\core{\cdot}$ is a functor from the category $\catGB[k]$ to the category of vector bundles.
\end{df}

\begin{ex}[split \grBs]\label{ex:split_gr_bndl} Given a sequence of vector bundles $E_j$, $j=1, \ldots, k$, over the same base manifold $M$ we can turn the Whitney sum $E =\bigoplus_{j=1}^k E_j$ into a \grB\ of \degree\ $k$ by assuming that linear coordinates on fibres of $E_j$ have weight $j$. The obtained \grB\ will be denoted by $\Cross_{j=1}^k E_j[j]$, where the notation $V[j]$ means that we assign weight $j$ to linear coordinates on a vector space~$V$. In other words, we have a~functor from the category of graded vector bundles supported in degrees $-1, -2, \ldots, -k$ to the category~$\catGB[k]$ of \grBs\ of~\degree\ $k$.

\grBs\ equipped with a \emph{splitting}, i.e., a \grB\ isomorphism $p\colon E^k \ra \Cross_{j=1}^k E_j$ as above are called \emph{split \grBs}. It is worth to remember that in the smooth real setting any \grB\ is (non-canonically) isomorphic to its \emph{split form} which is obtained from the sequence of the top core bundles $\core{E^j}$, $j=1, \ldots, k$: $E^k \simeq \Cross_{j=1}^k \core{E^j}$.
\end{ex}

{\bf Multi-\grBs.} Of particular interests are geometric objects which admit several compatible \grB\ structures. Multi-grading appeared as a powerful tool in Voronov's paper~\cite{Voronov_2012}, were it was used to study Mackenzie's double Lie algebroids~\cite{Mackenzie_1992}.
We define a \emph{$k$-tuple \grB} $(E, \Delta_1, \ldots, \Delta_k)$ to be a manifold $E$ with $k$ pair-wise commuting weight vector fields, i.e., $[\Delta_i,\Delta_j]=0$ for any $i,j=1,\hdots,k$. Equivalently, $E$ admits $k$ pair-wise commuting homogeneity structures $h^i\colon \R\times E\ra E$ with $i=1,\hdots,k$. That is, $h^i_t\circ h^j_s=h^j_s\circ h^i_t$ for any $i,j=1,2,\hdots,k$ and any $t,s\in\R$.
Interestingly, every $k$-tuple \grB\ admits local coordinates which are simultaneously graded with respect to any of its graded structures \cite[Theorem 5.2]{JG_MR_gr_bund_hgm_str_2011}.

 Double vector bundles, the tangent and cotangent bundles of a \grB\ \mbox{$\sigma^k\colon\! E^k{\ra} M$} are examples of double \grBs. There are numerous other constructions of multi-\grBs: the iterated higher tangent bundles $\T^{k_1} \cdots \T^{k_r} M$, the substructures and quotients (e.g., the construction of the linearization of a \grB, the vertical subbundle $\V E^k = \ker{\T \sigma^k} \subset T E^k$) and others (see \cite[Definitions~1.6 and~1.8]{BGR} for ``the removal of coordinates of weight greater than a given number~$l$'' and ``putting to zero coordinates of $X$-negative weights'' where $X$ is linear combination of weight vector fields $\Delta_i$).

The notion of core of a DVB has a natural generalization for multi-\grBs.
\begin{df}[the core and the ultracore of a multi-\grB]\label{df:Core} Let $(E, \Delta_1, \ldots, \Delta_k)$ be a $k$-tuple \grB\ and $k>1$. The \emph{core $\Core(E)$ of $E$} is obtained by putting to zero all fiber coordinates $y^i_w$ with multi-weight $w$ in which at least one of $w_1, \ldots, w_k$ is zero.
Formally, using the construction $E[X\geq 0]$ \cite[Definition~1.8]{BGR} of putting to zero all fiber coordinates with negative weight with respect to a vector field~$X$ being a linear combination of weight vector fields~$\Delta_i$, we may write the core of $E$ as the intersection
\begin{gather*}
\Core(E):= \bigcap_{j=1}^k E\bigg[N \Delta_j - \sum_{i\neq j} \Delta_i \geq 0\bigg],
\end{gather*}
for a sufficient large number $N$. Indeed, $E[N \Delta_j - \sum_{i\neq j} \Delta_i \geq 0]$ is the subset of $E$ obtained by putting to zero all \emph{fiber} coordinates $y^i_w$ of the total fibration $E\ra M= E[\Delta_1=\Delta_2=\cdots=\Delta_k=0]$ with multi-weight $w=(w_1, \ldots, w_k)$ in which $w_j=0$. In general, $\Core(E)$ remains a~$k$-tuple \grB.

By the \emph{ultracore} of $E$ we mean the top core $\core{E}$ of the \grB\ $(E, \Delta_1+\cdots+\Delta_k)$.
\end{df}

Note that for a double vector bundle, the core in the above sense coincides with the usual notion of the core of a DVB. What is more, our definition of the ultracore coincides with the one introduced by Mackenzie
\cite{Mackenzie_Triple_VB_2005}.

{\bf Graded-linear manifolds.}
 Let us now take a closer look at double \grBs\ with one of the homogeneity structures being linear. If $\big(E,h^1,h^2\big)$ is such a structure, we may treat it as a vector bundle $\sigma\colon E\ra M$ (say, that the VB structure on $E$ corresponds to the homogeneity structure $h^2$) equipped with a \grB\ structure encoded in the homogeneity structure $h^1\colon \R \times E\ra E$ such that the structure maps of the vector bundle $\sigma$ (vector bundle projection, zero section, addition, and scalar multiplication) are \emph{weighted} or, in other words, homogeneous with respect to~$h^1$. In particular, for each $\lambda\in \R$, the mapping $v \mapsto \lambda \cdot_{\sigma} v$, $v\in E$, should be a \grB\ morphism, or equivalently, it should commute with the homotheties~$h^2_t$ for any $t\in \R$. A double \grB\ $\big(E, h^1, h^2\big)$ in which $\big(E, h^2\big)$ is a vector bundle is called a~\emph{graded-linear manifold} or, equivalently, a \emph{weighted vector bundle}, when we want to put emphasis on the underlining linear structure. In the language of \cite{BGG}, we would speak about \emph{graded-linear bundles}. A graded-linear manifold $\sigma\colon E\ra M$ can be depicted in a diagram
\begin{gather}\label{d:gr_lin_bndl}\begin{split}&
\xymatrix{
E^k \ar[rr]^{\sigma^k_0} \ar[d]_{\sigma^{\{k\}}} && E^0 \ar[d]^{\sigma^{\{0\}}} \\
M^k \ar[rr]^{\underline{\sigma}^k_0} && M^0.}\end{split}
\end{gather}
Here $E^k = E$, $\sigma^{\{k\}} = \sigma$ and the superscript $k$ at $E$ indicates the order of the homogeneity structure. The base $M = M^k$ of the vector bundle $\sigma$ carries a \grB\ structure induced from $\sigma^k_0\colon E^k\ra E^0$. The notation in parentheses $\{ \ \}$ is used in order to distinguish the vector bundle projection $\sigma^{\{k\}}$ given by the homotheties $h^2$ from the \grB\ projection $\sigma^k_0$ given by the homogeneity structure~$h^1$.

An upcoming definition of a higher (Lie) algebroid involves two particular examples of graded-linear manifolds: the tangent bundle of a \grB\ of \degree\ $k$, and the $\thh{k}$-order tangent bundle of a vector bundle. Thus it will be important to understand the core bundles of these examples in greater detail. For this we shall show now that the core of any graded-linear manifold admits a canonical splitting into a direct sum of vector bundles. In the special case of a double vector bundle we recover the fact that the core is a vector bundle over the final base~$M$.

\begin{prop}[the core of a graded-linear manifold]\label{p:core_gr_lin_bndl} The core of any \degree-$k$ graded-linear manifold $(D, \Delta_1, \Delta_2)$ is a graded vector bundle, i.e.,
\begin{gather*}
\Core(D) = \bigoplus_{j=1}^k D_j[j]
\end{gather*}
for some vector bundles $D_j$ canonically associated with $D$. In particular cases,
\begin{gather*}
\Core\big(\TT{k} E\big) \simeq E[1]\oplus E[2] \oplus \cdots \oplus E[k], \qquad \text{and}\\
\Core(\T E^k) \simeq E^1[1] \oplus \core{E^2}[2] \oplus \cdots \oplus \core{E^k}[k].
\end{gather*}
\end{prop}
\begin{proof}Bi-graded coordinates on $D$ have the form $\big(x^A, y^a_{(i, j)}\big)$ where $i=0$ or $1$ and $0\leq j\leq k$, and $(i,j)\neq (0,0)$. As $C:= \Core(D)$ is given locally in $D$ by equations $y^a_{(1, 0)}=0$ and $y^a_{(0, j)}=0$, $1\leq j\leq k$, the transformation rules for the fibre coordinates $y^a_{(1, j)}|_{C}$, $j=1, \ldots, k$, are of the form $y^{a'}_{(1, j)}|_{C} = Q^{a'}_{j, b}\, y^b_{(1, j)}|_{C}$ for some functions $Q^{a'}_{j, b}$ on the final base $M$ of $D$. It follows that the submanifold of $D$ defined locally by vanishing all fibre coordinates except those of weight ${(1, j)}$, where $1\leq j\leq k$ is a fixed integer, form a vector bundle $D_j$ over $M$ and $C$ is just the Whitney sum of these vector bundles.
 The vector bundle $D_j$ can be recognized as the top core bundle of $(D[\Delta_2\leq j], \Delta_1+\Delta_2)$. The decomposition of the core of $\T^k E$ and $\T E^k$ are due to the isomorphisms $\core{\T^k E}\simeq E$ and $\core{\T E^k} \simeq \core{E^k}$. \end{proof}

{\bf Weighted structures and their reductions from `higher' to `lower' order.} In a~similar manner we can introduce notion of other weighted geometrical objects and structures, in particular \ZM s.

\begin{df}[weighted \ZM ]\label{def:wZM} Let $\sigma_i\colon E_i\ra M_i$, where $i=1,2$, be weighted VBs of order~$k$ with homogeneity structures $h_i$ on $E_i$. A \ZM\ $r\colon \sigma_1\relto \sigma_2$ is called a \emph{weighted \ZM} if $r\subset E_1\times E_2$ is homogeneous with respect to $h_1\times h_2$. In other words, $r$ is a graded submanifold of the product of \grBs\ $(\sigma_1)^k_0 \times (\sigma_2)^k_0\colon E_1\times E_2 \ra E_1^0\times E_2^0$. (See diagram~\eqref{d:gr_lin_bndl} for the notation concerning weighted vector bundles.) In particular, the base map $\und{r}\colon M_2\ra M_1$ is a~\grB\ morphism. (We may assume that \degree s of both the homogeneity structures are the same, as every \grB\ of \degree\ $k$ is also of \degree\ ${k'}$ for any $k'>k$.)
\end{df}
\begin{rem}[on weighted structures] Weighted (Lie) algebroids and weighted (Lie) groupoids considered in~\cite{BGG} are defined in the same spirit. One simply requests that the structure maps of a (Lie) algebroid (resp.\ a (Lie) groupoid) are graded, in some sense. Particular examples are $\VB$-algebroids and $\VB$-groupoids extensively studied in the literature. (The graded structure is then of \degree\ 1.) We shall not use these notions in our paper.
\end{rem}

For a weighted structure we may try to perform a \emph{reduction from a `higher' to a `lower' order} in the tower of affine fibrations \eqref{e:seq_aff_fib}. If the starting point is a graded-linear manifold $\big(E^k, \Delta_1, \Delta_2\big)$ of \degree\ $k$, then its $j$-order reduction $E^j = E^k[\Delta_1\leq j]$ is still a graded-linear manifold and $\sigma^k_j\colon E^k\ra E^j$ is a morphism in the category of graded-linear manifolds.
In particular, given a weighted \ZM\ $r\colon \sigma_1 \relto \sigma_2$ between weighted vector bundles $\sigma_i\colon E^k_i\ra M_i^k$ ($i=1,2$) of \degree\ $k$ we may consider~$r$ as a~$\N$-graded submanifold in the product bundle $\sigma_1^k\times \sigma_2^k\colon E^k_1 \times E^k_2 \ra E^0_1 \times E^0_2$ and then project it to get its lower-order reduction \mbox{$r^j\colon E^j_1 \relto E^j_2$}, which is not only a graded-linear submanifold, but also a weighted \ZM\ of \degree\ $j$, for $j=0, 1, 2, \ldots, k$. This and another useful observation in a similar spirit is presented below.

\begin{prop}[reduction of weighted \ZM s]\label{prop:red_h_to_l} Let $r\colon \sigma_1\relto \sigma_2$ be a weighted \ZM\ where $\sigma_i\colon E_i^k\ra M_i^k$ $(i=1, 2)$ are weighted vector bundles of \degree\ $k$, and let $0\leq j\leq k$. Then
\begin{enumerate}[$(i)$]\itemsep=0pt
 \item \label{prop:red_h_to_l:i1} the lower-order reductions $r^j\colon \sigma^{\{j\}}_1\relto \sigma^{\{j\}}_2$ remain weighted \ZM s.
 \item \label{prop:red_h_to_l:i2} If $r'\colon \sigma_1' \relto \sigma_2'$ is another weighted \ZM\ and if $(\phi_1, \phi_2)\colon r \mZM r'$, where $\phi_i\colon \sigma_i \ra \sigma_i'$ are morphism of graded-linear manifolds, is a $\catZM$-morphism, then the lower-order reductions $\big(\phi_1^j, \phi_2^j\big)\colon r^j \mZM {r'}^j$ are also $\catZM$-morphisms.
\end{enumerate}
\end{prop}

 The proof is given in Appendix~\ref{sec:app}.

\section{Higher (Lie) algebroids}\label{sec:high_alg}

{\bf Definitions and fundamental examples.} As was already suggested in the introduction, the higher tangent bundle $\tau^k_M\colon \TT{k} M\ra M$ (cf.\ the beginning of Section~\ref{sec:gr_bund}) together with the canonical isomorphism $\kappa^k_M\colon \TT{k} \T M\ra \T \TT{k} M$ will be our fundamental example of a higher (Lie) algebroid. According to our considerations from the beginning of the previous section, a choice of local coordinates $(x^a)$ on $M$ induces canonical coordinates $(x^a,\dot x^b)$ on~$\T M$, $\big(x^{a,(\alpha)}\big)_{\alpha=0,\hdots,k}$ on~$\TT{k} M$, $\big(x^{a},\dot x^{b},x^{a,(1)},\dot x^{b,(1)},\hdots,x^{a,(k)},\dot x^{b, (k)}\big)$ on $\TT{k}\T M$, and $\big(x^{a,(\alpha)},\dot x^{b,(\beta)}\big)_{\alpha,\beta=0,\hdots,k}$ on $\T\TT{k} M$. In these coordinates $\kappa^k_M$ reads as
\begin{gather*} \kappa^k_M\colon \ \big(x^{a},\dot x^{b},x^{a,(1)},\dot x^{b,(1)},\hdots,x^{a,(k)},\dot x^{b, (k)}\big)\longmapsto \big(x^{a},x^{a,(1)},\hdots,x^{a,(k)},\dot x^{b},\dot x^{b,(1)},\hdots,\dot x^{b,(k)}\big).\end{gather*}
Clearly, for $k=1$ we get $\kappa^1_M=\kappa_M$ studied in Example~\ref{ex:tang_alg}. It follows easily from the above coordinate formula that~$\kappa^k_M$ is an isomorphism in the category of graded-linear manifolds.

According to our motivating considerations from the introduction, we would now like to generalize the above example by substituting the fibration $\TT{k}M\ra M$ with an arbitrary \grB\ $E^k\ra M$ of \degree\ $k$. Instead of a graded-linear isomorphism $\kappa^k_M$ we shall take a~weighted \Zakrzewski\ $\kappa^k$ between $\TT{k} E^1$ and $\T E^k$. Such postulates are motivated by the description of (Lie) algebroids (of order one) in terms of a \ZM\ $\kappa$ (cf.\ Proposition~\ref{prop:sigma_from_kappa} and Lemma~\ref{lem:kappa_from_sigma}).

\begin{df}[general higher algebroid, HA] \label{def:higher_algebroid} A \emph{general higher} ($\thh{k}$-order) \emph{algebroid} (\emph{HA}, in short) is a \grB\ $\sigma^k\colon E^k\ra M$ of \degree\ $k$ together with a weighted \ZM\ $\kappa^k \subset \TT{k} E^1\times \T E^k$ from $\TT{k} \sigma^1$ to $\tau_{E^k}$ (covering a (graded) mapping $\rho^k\colon E^k \ra \TT{k} M$) such that the relation $\kappa^1\colon \T\sigma^1\relto\T\tau_{E^1}$ being the reduction to order one of $\kappa^k$ equips $\sigma^1\colon E^1\ra M$ with a~general algebroid structure.
 \begin{gather}\label{diag:HA}\begin{split}&
\xymatrix{\TT{k}E^1\ar[d]^{\TT{k}\sigma^1} \ar@{-|>}[rr]^{\kappa^k}&&\T E^k\ar[d]^{\tau_{E^k}}\\
\TT{k}M &&\ar[ll]_{\rho^k} E^k.
}\end{split}
\end{gather}

In addition:
\begin{enumerate}[(i)]\itemsep=0pt
\item If $\big(\sigma^1, \kappa^1\big)$ is skew, we call $\big(\sigma^k, \kappa^k\big)$ a \emph{skew} HA.
\item \label{axiom:AL_ha} If $\big(\sigma^k, \kappa^k\big)$ is skew and $\big(\TT{k}\rho^1, \T \rho^k\big)\colon \kappa^k\mZM \kappa^k_M$ is a morphism in the category of \ZM s, where $\rho_1\colon E^1\ra \T M$ is the order-1 reduction of~$\rho^k$, then we call $\big(\sigma^k, \kappa^k\big)$ an \emph{almost-Lie $($AL$)$}~HA.
\item \label{axiom:Lie_ha} A skew higher algebroid $\big(\sigma^k,\kappa^k\big)$ in which $\kappa^k$ is a subalgebroid of the product of the algebroids $\big(\TT{k} \sigma^1,\dd_{\TT{k}}\kappa^1\big)$ (the $\thh{k}$-tangent lift of $\big(\sigma^1,\kappa^1\big)$ -- see Section~\ref{ssec:higher_lie} -- and the tangent algebroid $(\tau_{E^k}, \kappa_{E^k})$), is called a \emph{Lie} HA. In fact a Lie HA must be an almost-Lie HA -- see Proposition~\ref{prop:Lie_HA}. In the language of~\cite{Cattaneo_Dherin_Weinstein_2012},
 $\big(\sigma^k,\kappa^k\big)$ is Lie if $\kappa^k$ is a Lie algebroid comorphism.
\item \label{axiom:strong_ha} A general HA in which $\kappa^k$ induces an isomorphism on the core bundles $\Core\big(\kappa^k\big)\colon \Core\big(\TT{k} E\big) \ra \Core\big(\T E^k\big)$ is called a \emph{strong} HA.
 \end{enumerate}

Analogously to our convention from Section~\ref{sec:algebroids_as_zm} we will colectively refer to the above-defined objects as to \emph{$($Lie$)$ higher algebroids}.

We define a \emph{morphism} between (Lie) higher algebroids $\big(\sigma^k_E\colon E^k\ra M, \kappa^{k,E}\big)$ and $\big(\sigma^k_F\colon F^k\ra N, \kappa^{k,F}\big)$ to be a \grB\ morphism $\phi^k\colon E^k\ra F^k$ such that $\big(\TT{k} \phi^1, \T \phi^k\big)\colon \kappa^{k,E}\mZM \kappa^{k,F}$ is a~$\catZM$-morphism. Obviously, general higher algebroids form a \emph{category} which we denote by $\catHA$, with Lie HA, AL HA, and skew HA being its full subcategories.
\end{df}

Let us remark, that the symmetric role of the left and the right anchor occurring in the first order case is no longer present for (Lie) higher algebroids. The defining relation $\kappa^k$ covers a~(graded) morphism $\rho^k\colon E^k\ra\TT{k}M$, and its reduction to order~0 induces a relation (in fact a~linear map) $\kappa^0\colon E^1\ra \T M$. For $k=1$ these were the left and the right anchors, respectively. As we see, these maps are of quite a different nature. Note, however, that for skew higher algebroids $\kappa^0=\rho^1$.

Using Proposition~\ref{prop:red_h_to_l} we easily see that reduction of a (Lie) HA $\big(E^k, \kappa^k\big)$ to a lower order $j=1,2,\hdots,k$ gives an induced (Lie) higher algebroid structure on $E^j$ which is skew (resp. AL, Lie, strong) if $\big(E^k, \kappa^k\big)$ was so. Moreover, a (Lie) HA morphism $\phi^k\colon \big(E^k, \kappa^{k,E}\big)\ra \big(F^k, \kappa^{k,F}\big)$ induces a (Lie) HA morphisms $\phi^j\colon \big(E^j, \kappa^{j,E}\big)\ra \big(F^j, \kappa^{j,F}\big)$.

Examples of (Lie) higher algebroids will be discussed later in Sections~\ref{ssec:prolongations} and~\ref{sec:examples}. For now let us take a closer look at general higher algebroids of order 2.
\begin{rem}[a local form of a general higher algebroid of order 2]\label{rem:local_order_2}
Let $(x^a, y^i, z^\mu)$ be graded coordinates on an \degree-$2$ \grB\ $E^2\ra M$.
We shall find a general form of an algebroid $\big(E^2, \kappa^2\big)$ of order $2$. As in Remark~\ref{rem:local_kappa_a} we underline coordinates in $\T E^2$ when writing equations for $\kappa^2\subset \TT{2} E^1 \times \T E^2$. Taking weights into account
\begin{gather}\label{e:T2E_TE2_weights}
\begin{tabular}{|c|c|c|c|c|c|c|}\hline
\text{weight}& $(0,0)$ & $(0,1)$ & $(0,2)$ & $(1,0)$ & $(1,1)$ & $(1,2)$ \\ \hline
$\TT{2} E^1$ & $x^a$ & $\dot{x}^a$ & $\ddot{x}^a$ & $y^i$ & $\dot{y}^i$ & $\ddot{y}^i$ \\ \hline
$\T E^2$ & $x^a$ & $\und{y}^i$ & $\und{z}^\mu$ & $\dot{\und{x}}^a$ & $\dot{\und{y}}^i$ & $\dot{\und{z}}^\mu$ \\ \hline
\end{tabular}
\end{gather}
we get
\begin{gather}\label{e:local_kappa2}
\kappa^2\colon \
\begin{cases}
\dot{x}^a = Q^a_i \und{y}^i, \\
\ddot{x}^a = \frac{1}{2} Q^a_{ij}\,\und{y}^i\und{y}^j + Q^a_\mu\,\und{z}^\mu, \qquad \text{where} \quad Q^a_{ij}=Q^a_{ji}, \\
\dot{\und{x}}^a = Q^{'a}_i y^i,\\
\dot{\und{y}}^i = \dot{y}^i + Q^i_{jk}\,\und{y}^j y^k,\\
\dot{\und{z}}^\mu = Q^\mu_i \ddot{y}^i + Q^\mu_{ij} \und{y}^i \dot{y}^j + Q^\mu_{\nu i}\und{z}^\nu y^i + \frac{1}{2} Q^\mu_{ij,k} \und{y}^i\und{y}^j y^k, \qquad \text{where} \quad Q^\mu_{ij,k}=Q^\mu_{ji,k},
\end{cases}
\end{gather}
for some structure functions $Q^{\cdots}_{\cdots}$ of base coordinates $x^a$. For a skew algebroid $\big(E^2, \kappa^2\big)$ we have $Q^{'a}_i= Q^a_i$ and $Q_{jk}^i = - Q_{kj}^i$. An algebroid is strong if the matrix $Q^\mu_i$ is invertible. The conditions of being an almost-Lie or a Lie algebroid result in more complicated systems of equations for the structure functions.
\end{rem}

 Our considerations from Section~\ref{sec:algebroids_as_zm} lead to a notion of a higher subalgebroid.
\begin{df}[subalgebroid of a (Lie) HA]\label{def:higher_sub_algebroid} Let $\big(\sigma^k\colon E^k\ra M, \kappa^k\big)$ be a (Lie) higher algebroid and let ${\sigma'}^{k}\colon {E'}^{k}\ra M'$ be a $\N$-graded submanifold of $\sigma^k$. We call ${\sigma'}^{k}$ a \emph{higher subalgebroid} of $\big(\sigma^k,\kappa^k\big)$ if the \ZM\ $\kappa^k$ restricts fine to $\TT{k}{\sigma'}^1\times\tau_{{E'}^k}\subset\TT{k}\sigma^1\times\tau_{E^k}$.
\end{df}

\subsection{On the Lie axiom}\label{ssec:higher_lie}

In order to discuss the notion of a Lie higher algebroid, we shall first take a look at higher tangent lifts of a first order (Lie) algebroid structure. We will begin our considerations by studying the lifts of sections of a vector bundle.

{\bf Higher lifts of sections of a vector bundle.} Let $\sigma\colon E\ra M$ be a vector bundle and $s\in\Sec_M(E)$ its section. The assignment $\Sec_M(E)\ni s\mapsto \TT{k} s\in\Sec_{\TT{k} M}\big(\TT{k} E\big)$ is injective, and $\R$-linear, yet it cannot be onto, since $\rank\big(\TT{k}\sigma\big)=(k+1)\cdot\rank(\sigma) > \rank(\sigma)$. It is, however, possible to characterize all sections of $\TT{k}\sigma$ in terms of sections of $\sigma$ by means of the following procedure.

First{,} note that the multiplicative action of the reals $\R\times E\ra E$ lifts to the multiplicative action $\TT{k}\R\times\TT{k} E\ra\TT{k} E$ of $\TT{k}\R$, the latter understood as an $\R$-algebra with addition and multiplication defined as $\TT{k}$-lifts of the standard addition and multiplication in $\R$. In fact, as an algebra $\TT{k}\R$ is canonically isomorphic to the truncated polynomial algebra $\R[\eps]/_{\<\eps^{k+1}>}$. (This algebra is often denoted $\D_k$ and called the \emph{algebra of $\thh{k}$-order numbers}.) From the point of view of the theory of natural bundles \cite{KMS_1993}, $\TT{k}$ is a product-preserving bundle functor associated with the \emph{Weil algebra} $\D_k$. In conclusion, on $\TT{k} E$ we have an operation of multiplication by the elements of $\D_k$, which is in fact determined by the action of the generator $\eps\in\D_k$. In coordinates $\big(x^{a,(\alpha)},y^{i,(\alpha)}\big)$ on $\TT{k}\sigma$ induced from linear coordinates $(x^a,y^i)$ on $\sigma$ it reads as
\begin{gather*}\eps\cdot \big(x^{a,(\alpha)};y^{i,(0)},y^{i,(1)},\hdots,y^{i,(k)}\big)=
\big(x^{a,(\alpha)};0, y^{i,(0)}, 2 y^{i, (1)}, \hdots,k y^{i,(k-1)}\big).\end{gather*}
This construction leads us to

\begin{df}[lifts of a section]\label{def:lift_section} Let $s$ be a section of $\sigma$ and $\alpha=0,1,\hdots,k$ an integer. By the \emph{$(k-\alpha)$-lift of $s$} we understand the section of $\TT{k}\sigma$ defined as
\begin{gather*} s^{(k-\alpha)}:= \frac{(k-\alpha)!}{k!}\eps^\alpha\cdot \TT{k}s. \end{gather*}
In particular $s^{(0)}=\frac{1}{k!} \eps^k\cdot\TT{k}s$ is called the \emph{vertical} and $s^{(k)}=\TT{k}s$ the \emph{total lift} of $s$, in agreement with the standard notions for $k=1$.
\end{df}
Using local coordinates it is easy to check that sections of the form $s^{(k-\alpha)}$ with $\alpha=0,1,\hdots,k$ span locally the full space of sections $\Sec_{\TT{k}M}\big(\TT{k}E\big)$. Recall the notion of an $(\alpha)$-lift of a smooth function introduced at the beginning of Section~\ref{sec:gr_bund}. For any function $f\in\Cf(M)$ the operation of the $(k)$-lift has the following property
\begin{gather*}(f\cdot s)^{(k)}=\TT{k} (f\cdot s)=\TT{k} f\cdot\TT{k}s\\
\hphantom{(f\cdot s)^{(k)}}{} = \left(\sum_{\alpha=0}^kf^{(\alpha)}\frac{\eps^\alpha}{\alpha !}\right)\cdot s^{(k)}= \sum_{\alpha=0}^k
\frac{f^{(\alpha)}}{\alpha!}\frac{k!}{(k-\alpha)!}s^{(k-\alpha)} =\sum_{\alpha=0}^k \binom{k}{\alpha} f^{(\alpha)}s^{(k-\alpha)}.
\end{gather*}
We recognize the standard formula for the iterated derivative. An analogous formula for \mbox{$(f\cdot s)^{(k-\alpha)}$} can be easily derived from the latter.

{\bf Higher tangent lifts of (Lie) algebroids.} It is well-known that if a VB $\sigma\colon E\ra M$ carries a (Lie) algebroid structure, then its $\thh{k}$-tangent lift $\TT{k}\sigma\colon \TT{k} E\ra\TT{k} M$ has the so-called lifted (Lie) algebroid structure (e.g.,~\cite{K_N_Tangent_lifts_2013}). It can be elegantly described in terms of the \ZM\ $\kappa$ related with the initial algebroid structure on~$\sigma$.

\begin{df}[higher tangent lift of a (Lie) algebroid]\label{def:h_t_lift_algebroid} Let $(\sigma\colon E\!\ra\! M, \rhoL,\rhoR\colon E\!\ra\! \T M, [\cdot, \cdot])$ $\simeq(\sigma,\kappa\colon \T\sigma\relto\tau_E)$ be a (Lie) algebroid. A \ZM\ $\dd_{\TT{k}}\kappa\colon \T \TT{k} \sigma\relto \tau_{\TT{k}E}$ over an anchor map $\dd_{\TT{k}}\rhoL\colon \TT{k} E\ra \T \TT{k} M$ obtained, up to natural identifications, by applying functor~$\TT{k}$ to~$\kappa$:
\begin{gather}\label{diag:kappa^Tk_E}\begin{split}&
\xymatrix{
\T \TT{k} E \ar[d]^{\T \TT{k} \sigma} \ar[rr]^{\simeq \kappa_{k,E}^{-1}} &&
\TT{k} \T E\ar@{-|>}[rr]^{\TT{k} \kappa}\ar[d]^{\TT{k} \T \sigma} &&\TT{k} \T E \ar[d]^{\TT{k} \tau_{E}}
\ar[rr]^{\simeq \kappa_{k, E}} && \T \TT{k} E \ar[d]^{\tau_{\TT{k} E}}
\\
\T\TT{k} M &&\ar[ll]_{\simeq \kappa^k_M}
 \TT{k} \T M && \TT{k} E \ar[ll]_{\TT{k} \rhoL}
 && \ar[ll]_{=}\TT{k} E,
}\end{split}
\end{gather}
i.e., $\dd_{\TT{k}}\kappa= \kappa_{k, E}^{-1}\circ\TT{k} \kappa\circ \kappa_{k, E}$, and $\dd_{\TT k} \rhoL:=\TT k\rhoL\circ\kappa_k^M$, equips $\TT{k} \sigma\colon \TT{k} E\ra \TT{k} M$ with a (Lie) algebroid structure. We call it the \emph{$\thh{k}$-tangent lift of the algebroid $(\sigma,\kappa)$}.
\end{df}

We shall now check that $\dd_{\TT{k}}\kappa$ indeed defines a (Lie) algebroid structure and then describe it in the classical terms of the anchors and the bracket operation.

\begin{prop}[properties of the lifted (Lie) algebroid]\label{p:algebroid_Tk_E} The relation $\dd_{\TT{k}}\kappa:=\kappa_{k, E}\circ\TT{k} \kappa\circ\kappa_{k, E}^{-1}$ defines a $($Lie$)$ algebroid structure on $\TT{k}\sigma$. Moreover,
\begin{enumerate}[$(i)$]\itemsep=0pt
\item The $($left and right$)$ anchor maps on $\TT{k} \sigma$ are given by $\dd_{\TT{k}}\rhoL:=\kappa^k_M\circ\TT{k}\rhoL$ and $\dd_{\TT{k}}\rhoR:=\kappa^k_M\circ\TT{k}\rhoR$, respectively.
\item \label{p:algebroid_Tk_E_item_bracket} The algebroid bracket $[\cdot,\cdot]_{\dd_{\TT{k}}}$ on $\TT{k}\sigma$ satisfies
\begin{gather}\label{p:algebroid_Tk_E_item_bracket_formula}
\left[ \frac{k!}{(k-\alpha)!} s_1^{(k-\alpha)}, \frac{k!}{(k-\beta)!} s_2^{(k-\beta)}\right]_{\dd_{\TT{k}}} = \frac{k!}{(k-\alpha-\beta)!} ([s_1, s_2]_{\sigma})^{(k-\alpha-\beta)},
\end{gather}
for any integers $\alpha,\beta=0,1,\hdots,k$ such that $\alpha + \beta \leq k$ and any sections $s_1,s_2\in\Sec_M(E)$, and $\big[s_1^{(k-\alpha)}, s_2^{(k-\beta)}\big]_{\dd_{\TT{k}}} = 0$ if $\alpha+\beta>k$. These properties fully determine the bracket $[\cdot,\cdot]_{\dd_{\TT{k}}}$ on the space of sections of $\TT{k}\sigma$.
\item \label{p:algebroid_Tk_E_item_Tk_E_is_Lie} If $(\sigma,\kappa)$ is a skew/AL/Lie algebroid, then so is $\big(\TT{k} \sigma,\dd_{\TT{k}}\kappa\big)$.
\end{enumerate}
\end{prop}
The proof is given in Appendix~\ref{sec:app}.

\begin{rem}[characterization of lifts by the canonical pairings]\label{rem:lifts_pairings} The pairing $\<\cdot, \cdot>_\sigma\colon E^\ast\times_M E\ra \R$ lifts to the pairing $\TT{k} \<\cdot, \cdot>_\sigma\colon \TT{k} E^\ast \times_{\TT{k} M}\TT{k} E \ra \R$ which enables us to identify sections of the vector bundle $\TT{k} \sigma$ with linear functions on $\TT{k} E^\ast$. As $\TT{k} E^\ast$ is, in particular, a \grB\ of \degree\ $k$, the space of sections of $\TT{k} \sigma$ is naturally graded.

Consider a section $s$ of $\sigma$. It is easy to see that its $(k-\alpha)$-lift $s^{(k-\alpha)}$ (considered as a~linear function on $\TT{k} E^\ast$) coincides with the $(k-\alpha)$-lift $\tilde{s}^{(k-\alpha)}$ of the linear function $\tilde{s}\colon E^{\ast}\ra \R$ canonically associated with the section $s$. Taking into account the \grB\ structure of $\TT{k} E^\ast$, a section of the form
$f^{(\alpha)} s^{(k-\beta)}$, where $f\in \Cf(M)$ and $s\in \Sec(E)$, would have weight $\alpha + k- \beta$. However, in a presence of a (Lie) algebroid structure on $\sigma$, it is reasonable to shift this gradation by $-k$, so that the section $f^{(\alpha)} s^{(k-\beta)}$ has weight $\alpha-\beta$. Then formula~\eqref{p:algebroid_Tk_E_item_bracket_formula} shows that homogeneous sections of $\TT{k} \sigma$ form a graded Lie algebra concentrated in degrees $-k, -k+1, \ldots$. It has a Lie subalgebra $\Sec_{\TT{k}M, \leq 0} \big(\TT{k} E\big)$ consisting of sections of non-positive weights which is of finite rank (over~$\Cf(M)$).
\end{rem}

{\bf Algebroid lifts of sections and the Lie axiom.} By Remark~\ref{rem:zm_sections} a \ZM\ between two vector bundles induces a canonical map between the associated spaces of sections of these bundles. On the other hand, in Definition~\ref{def:lift_section} we introduced natural lifts of sections of a vector bundle to sections of its higher tangent lift. Given a (Lie) higher algebroid structure, we can combine these two constructions to arrive at
\begin{df}[algebroid lift of a section]\label{def:alg_lifts_sections}
Let $\big(\sigma^k\colon E^k\ra M,\kappa^k\colon \TT{k}\sigma^1\relto\tau_{E^k}\big)$ be a (Lie) higher algebroid. Given a section $s\in\Sec_M\big(E^1\big)$ and an integer $\alpha=0,1,\hdots,k$ we define the \emph{$(k-\alpha)$-algebroid lift of $s$} as
\begin{gather*} s^{[k-\alpha]}:=\wh{\kappa^k}\big(s^{(k-\alpha)}\big)\in \X\big(E^k\big).\end{gather*}
\end{df}
The above concept naturally generalizes the notion of a \emph{vertical lift} and the \emph{algebroid lift} (for $(k,\alpha)=(1,0)$ and $(k,\alpha)=(1,1)$, respectively) known in the literature~\cite{JG_PU_Algebroids_1999}.

In \cite{JG_PU_Lie_alg_P-N_str_1997} Grabowski and Urba\'nski observed that the Jacobi identity can be elegantly reformulated in terms of the algebroid lifts. Their results easily generalize to the setting of Lie higher algebroids.
\begin{prop}[characterization of the Lie axiom]\label{prop:Lie_HA}
 The map $\wh{\kappa^k}\colon \Sec_{\TT{k}M} \big(\TT{k} E\big)\ra \X\big(E^k\big)$ preserve gradation, i.e., for a homogeneous section $s$ of $\TT{k} \sigma$ the vector field $\wh{ \kappa^k}(s)$ has the same weight as $s$. Moreover, a
higher algebroid $\big(\sigma^k\colon E^k\ra M,\kappa^k\colon \TT{k}\sigma^1\relto\tau_{E^k}\big)$ is Lie if and only if it is almost-Lie and if the associated algebroid lift satisfies
\begin{gather}\label{eqn:Lie_HA}
\left[\frac{k!}{(k-\alpha)!} s_1^{[k-\alpha]}, \frac{k!}{(k-\beta)!}s_2^{[k-\beta]}\right]_{E^k}=\frac{k!}{(k-\alpha-\beta)!}([s_1,s_2]_{\sigma^1})^{[k-\alpha-\beta]}
\end{gather}
for any sections $s_1,s_2\in\Sec_M\big(E^1\big)$ and any integers $\alpha,\beta=0,1,\hdots,k$. Above $[\cdot,\cdot]_{E^k}$ denotes the standard Lie bracket of vector fields on $E^k$, and $[\cdot,\cdot]_{\sigma^1}$ the algebroid bracket on the reduced $($Lie$)$ algebroid $\big(\sigma^1,\kappa_1\big)$.

Equivalently, an AL higher algebroid $\big(E^k, \kappa^k\big)$ is Lie if and only if
\begin{gather}\label{eqn:Lie_HA_leq0}
\wh{\kappa^k}\colon \ \Sec_{\TT{k} M, \leq0}\big(\TT{k} E\big) \ra \X_{\leq0}\big(E^k\big)
\end{gather}
is a graded Lie algebra morphism, where $\Sec_{\TT{k} M, \leq0}\big(\TT{k} E\big)$ $($resp.~$\X_{\leq0}\big(E^k\big))$ are Lie algebras generated by the homogeneous sections of $\TT{k}\sigma$ $($resp.\ vector fields on $E^k)$ of non-positive weight.
\end{prop}
\begin{proof} Let $s\in \Sec_{\TT{k} M} \big(\TT{k} E\big)$ be a homogeneous section, denote $X := \wh{ \kappa^k}(s)$ and consider the corresponding functions $\tilde{s}\in \Cf\big(\TT{k} E^\ast\big)$, and $\tilde{X} \in \Cf\big(\T^\ast E^k\big)$ on the dual vector bundles (cf.\ Remark~\ref{rem:lifts_pairings}). They are related by $\tilde{X} = \tilde{s}\circ \eps_k$, where $\eps_k\colon \T^\ast E^k \ra \TT{k} E^\ast$ is a (weighted) vector bundle morphism dual to~$\kappa^k$. Since~$\kappa^k$ is bi-homogeneous, so is~$\eps_k$, hence weights of $\tilde{X}$ and $\tilde{s}$ are the same, and our first assertion follows.

By definition, $\big(\sigma^k,\kappa^k\big)$ is a Lie HA if and only if it is skew and if $\kappa^k$ is an algebroidal relation between $\big(\TT{k}\sigma^1,\dd_{\TT{k}}\kappa^1\big)$ and $(\tau_{E^k},\kappa_{E^k})$. Proposition~\ref{prop:algebroidal_relation} provides a useful characterization of algebroidal relations. It follows that $\big(\sigma^k,\kappa^k\big)$ is a Lie HA if and only if for any sections $s,s'\in \Sec_{\TT{k} M}\big(\TT{k} E^1\big)$ we have (note that both $\big(\TT{k}\sigma^1,\dd_{\TT{k}}\kappa^1\big)$ and $(\tau_{E^k},\kappa_{E^k})$ are skew algebroids and their anchor maps are $\kappa^k_M\circ \TT{k}\rho^1$ and $\id_{\T E^k}$, respectively):
\begin{gather}\label{eqn:kappa_anchor}
\T \rho^k\big(\wh{ \kappa^k}(s)\big)=\kappa^k_M\big(\TT{k}\rho^1(s)\big) =\dd_{\TT k}\rho (s)
\end{gather} and
\begin{gather} \label{eqn:kapppa_bracket}
\wh{\kappa^k}([s,s']_{\dd_{\TT{k}}})=\big[\wh{\kappa^k}(s),\wh{\kappa^k}(s')\big]_{E^k}.
\end{gather}
The first of these conditions means that for any section $s\in \Sec\big(\T^k \sigma\big)$ the vector fields $\wh{\kappa^k}(s)\in \VF\big(E^k\big)$ and $\kappa^k_M \circ \T^k \rho^1 (s) = \dd_{\TT{k}}\rho(s) \in \VF\big(\T^k M\big)$ are $\rho^k$-related and it is equivalent to the commutativity of the diagram
\begin{gather*} \xymatrix{\TT{k} E^1\ar@{-|>}[rr]^{\kappa^k}\ar[d]^{\TT{k}\rho^1}&&\T E^k\ar[d]^{\T\rho^k}\\
\TT{k}\T M\ar[rr]^{\simeq\kappa^k_M}&& \T\TT{k}M,}\end{gather*} i.e., the axiom of an AL algebroid.
The second condition for lifts $s=s_1^{(k-\alpha)}$ and $s'=s_2^{(k-\beta)}$, with $s_1,s_2\in\Sec_M\big(E^1\big)$, gives equality \eqref{eqn:Lie_HA}.

Assume now that the equation \eqref{eqn:Lie_HA} holds and $\big(E^k,\kappa^k\big)$ is an almost-Lie higher algebroid. The latter condition implies~\eqref{eqn:kappa_anchor}, and from~\eqref{eqn:Lie_HA} it follows that equation \eqref{eqn:kapppa_bracket} holds for sections of the form $s=s_1^{(k-\alpha)}$ and $s'=s_2^{(k-\beta)}$. To prove that $\big(E^k,\kappa^k\big)$ is Lie, we need to show that~\eqref{eqn:kapppa_bracket} holds for any sections $s, s'\in \Sec\big(\T^k \sigma\big)$.

 Assuming that \eqref{eqn:kapppa_bracket} holds for given sections $s$ and $s'$, we shall show that it also holds for sections $s:=\phi s$ and $s'$ for any $\phi\in \Cf\big(\T^k M\big)$. Indeed, we have
\begin{gather*}\begin{split}&
\wh{\kappa^k}([\phi s,s']_{\dd_{\TT{k}}}) \overset{\eqref{eqn:bracket}}= \wh{\kappa^k}\big[\phi [ s,s']_{\dd_{\TT{k}}} - {\dd_{\TT{k}}}\rho(s')(\phi) s)\big] \\
& \hphantom{\wh{\kappa^k}([\phi s,s']_{\dd_{\TT{k}}})}{} \overset{\eqref{e:hat_r_axiom}}= \big(\rho^k\big)^\ast (\phi) \wh{\kappa^k}([s, s']_{\dd_{\TT{k}}}) -
\big(\rho^k\big)^\ast ({\dd_{\TT{k}}}\rho(s')(\phi))\wh{\kappa^k}(s) ,\end{split}	
\end{gather*}
while
\begin{gather*}
\big[\wh{\kappa^k}(\phi s),\wh{\kappa^k}(s')\big]_{E^k} \overset{\eqref{e:hat_r_axiom}}= \big[\big(\rho^k\big)^\ast (\phi)\wh{\kappa^k}(s),\wh{\kappa^k}(s')\big]_{E^k}\\
\hphantom{\big[\wh{\kappa^k}(\phi s),\wh{\kappa^k}(s')\big]_{E^k}}{} \overset{\eqref{eqn:bracket}}= \big(\rho^k\big)^\ast (\phi)\big[\wh{\kappa^k}(s),\wh{\kappa^k}(s')\big]_{E^k} - \wh{\kappa^k}(s')\big(\big(\rho^k\big)^\ast(\phi)\big) \wh{\kappa^k}(s)\\
\hphantom{\big[\wh{\kappa^k}(\phi s),\wh{\kappa^k}(s')\big]_{E^k}}{} = \big(\rho^k\big)^\ast (\phi)\big[\wh{\kappa^k}(s),\wh{\kappa^k}(s')\big]_{E^k} - \big(\rho^k\big)^\ast\big(\wh{\kappa^k}(s')(\phi)\big) \wh{\kappa^k}(s)\\
\hphantom{\big[\wh{\kappa^k}(\phi s),\wh{\kappa^k}(s')\big]_{E^k}}{} \overset{\eqref{eqn:kappa_anchor}}= \big(\rho^k\big)^\ast (\phi)\big[\wh{\kappa^k}(s),\wh{\kappa^k}(s')\big]_{E^k} - \big(\rho^k\big)^\ast(\dd_{\TT k}\rho(s')(\phi)) \wh{\kappa^k}(s) .
\end{gather*} Comparison of the above equalities proves our claim.
Now, since all $(k-\alpha)$-lifts (with $\alpha\geq 0$) of sections from $\Sec_M\big(E^1\big)$ span the whole space $\Sec_{\TT{k} M}\big(\TT{k} E^1\big)$ as the $\Cf\big(\TT k M\big)$-module, equation \eqref{eqn:kapppa_bracket} holds for every sections $s,s'\in \Sec_{\TT{k} M}\big(\TT{k} E^1\big)$. This ends the proof.
\end{proof}

\begin{rem} Note that $\Sec_{\TT{k} M, \leq0}\big(\TT{k} E\big)$ and $\X_{\leq0}\big(E^k\big)$ are $\Cf(M)$-modules of finite rank. It follows that the Lie axiom for HA can be reduced to a finite number of equations of the form~\eqref{eqn:Lie_HA} with $s_1=e_i$, $s_2=e_j$ where $(e_i)$ is a local basis of sections of~$\sigma_1$.
\end{rem}

{\bf Subalgebroids of Lie higher algebroids.} We shall end our considerations concerning the Lie axiom for HA by showing that it is inherited by subalgebroids. First we will need the following properties of higher tangent lifts.
\begin{prop}[properties of the lifted algebroid] \label{prop:sub_algebroidal_rel} Let $k$ be a positive integer. Consider vector bundles $\sigma\colon E\ra M$, $\sigma_1\colon E_1\ra M_1$ and $\sigma_2\colon E_2\ra M_2$ and their subbundles $\sigma'\colon E'\ra M'$, $\sigma'_1\colon E'_1\ra M'_1$ and $\sigma'_2\colon E'_2\ra M'_2$, respectively.
\begin{enumerate}[$(i)$]\itemsep=0pt
 \item \label{item:T^kr restricts_fine} Let
 $r\colon \sigma_1\relto \sigma_2$ be a \ZM
	\ that restricts fine to a \ZM\ $r'\colon \sigma_1'\relto \sigma_2'$. Then $\TT{k} r\colon \TT{k}\sigma_1\mbox{$\relto$}$ $\TT{k}\sigma_2$ is a \ZM\ which restricts fine to a \ZM\ $\TT{k} r'\colon \TT{k} \sigma_1'\relto\TT{k} \sigma_2'$.
 \item \label{item:T^kE_as_subalgeroid} Let $(\sigma,\kappa)$ be a $($Lie$)$ algebroid and $(\sigma',\kappa')$ its subalgebroid. Then the $\thh{k}$-tangent lift algebroid $\big(\TT{k}\sigma',\dd_{\TT{k}}\kappa'\big)$ is a subalgebroid of $\big(\TT{k}\sigma,\dd_{\TT{k}}\kappa\big)$.
 \item \label{item:T^kr_also_algebroidal} Let $(\sigma_1, \kappa_1)$ and $(\sigma_2,\kappa_2)$ be $($Lie$)$ algebroids and let $r\colon \sigma_1\relto \sigma_2$ be an algebroidal relation. Then $\TT{k} r\colon \TT{k} \sigma_1\relto \TT{k} \sigma_2$ is also an algebroidal relation between the $\thh{k}$-tangent lift algebroids $\big(\TT{k}\sigma_1,\dd_{\TT{k}}\kappa_1\big)$ and $\big(\TT{k}\sigma_2,\dd_{\TT{k}}\kappa_2\big)$.
\end{enumerate}
\end{prop}
The proof is given in Appendix~\ref{sec:app}.

 \begin{prop}[the Lie axiom and subalgebroids] \label{prop:sub_HA_is_Lie} A higher subalgebroid of a Lie $($resp.\ almost-Lie$)$ higher algebroid is Lie $($resp.\ almost-Lie$)$.
\end{prop}
\begin{proof} Assume that $\big(\sigma^k_F\colon F^k\ra N, \kappa^{k,F}\big)$ is a higher subalgebroid of an AL higher algebroid $\big(\sigma^k_E\colon E^k\ra M, \kappa^{k,E}\big)$. To prove that $\big(\sigma^k_F, \kappa^{k,F}\big)$ is also AL we must verify that if two elements $X\in \TT{k} F^1$ and $Y\in \T F^k$ are $\kappa^{k,F}$-related then their images $\TT{k} \rho^1_F (X)$ and $\T\rho^k_F(Y)$ are $\kappa^k_{ N}$-related. Now as $\kappa^{k,F}$ is a fine restriction of $\kappa^{k,E}$, elements~$X$ and~$Y$ as above are $\kappa^{k,E}$-related, and thus, since $\big(\sigma^k_E\colon E^k\ra M, \kappa^{k,E}\big)$ was almost-Lie, elements $\TT{k} \rho^1_E (X)$ and $\T\rho^k_E(Y)$ are $\kappa^k_M$-related. The assertion follows from a~simple observation that the anchor { maps} $\rho^k_F\colon F^k\ra \TT{k} N$ and $\rho^1\colon F^1\ra \T M$ are the restrictions $\rho^k_F=\rho^k_E|_{F^k}$ and $\rho^1_F=\rho^1_E|_{F^1}$, respectively.

Now assume that $\big(\sigma^k_E, \kappa^{k,E}\big)$ is Lie. This means that $\kappa^{k,E}$ is an algebroidal relation between $\TT{k}\sigma^1_E$ and $\tau_{E^k}$. Moreover, by definition, $ \kappa^{k,F}$ is a fine restriction of $\kappa^{k,E}$. In particular, $\big(\sigma^1_E, \kappa^{1, F}\big)$ is a subalgebroid of $\big(\sigma^1_E, \kappa^{1, E}\big)$, hence by Proposition~\ref{prop:sub_algebroidal_rel}\eqref{item:T^kE_as_subalgeroid}, $\TT{k}\sigma^1_F$ is a subalgebroid of $\TT{k}\sigma^1_E$. Note also that $\tau_{F^k}$ is a subalgebroid of $\tau_{E^k}$. Thus, by Proposition~\ref{prop:sub_algebroidal_rel}\eqref{item:T^kr_also_algebroidal}, $\kappa^{k,F}$ is also an algebroidal relation, i.e., $\big(\sigma^k_F\colon F^k\ra N, \kappa^{k,F}\big)$ is a Lie higher algebroid.
\end{proof}

\subsection{Prolongations of an AL algebroid}\label{ssec:prolongations}
In our previous publication \cite{MJ_MR_models_high_alg_2015} we studied a class of objects crafted to play the role of prototypes of (Lie) higher algebroids. Our construction was motivated by the procedure of reduction of a~higher tangent bundle of a~Lie groupoid. In fact, as we shall see shortly, these objects provide natural examples of (Lie) higher algebroids in the sense of Definition~\ref{def:higher_algebroid}. Let us now recall their construction.

{\bf Construction of the prolongations.} Let $(\sigma\colon E\ra M, \kappa)$ be an AL algebroid. For each $k=1,2,\hdots$, we will construct bundles $\sigma^{[k]}\colon \E{k}\ra M$ and relations $ \kappa^{[k]}\subset\TT{k} E\times \T \E k$. We shall call the pair $\big(\sigma^{[k]}, \kappa^{[k]}\big)$ the \emph{$\thh{k}$-prolongation of the algebroid $(\sigma,\kappa)$}. The total spaces of bundles $\sigma^{[k]}$ are defined inductively:
\begin{gather*}
\E{1}:= E, \qquad \E{2}:=\{A\in \T E\colon \T\sigma(A)=\rho\circ\tau_E(A)\},\\
\E{k+1}:=\T \E{k}\cap\TT{k} E, \qquad \text{for}\quad k\geq 2,
\end{gather*}
where $\rho\colon E\ra \T M$ is the anchor map of $(\sigma,\kappa)$. By the inductive hypothesis, $\E{k}$ is considered, as a subset of $\TT{k-1}E$, hence both $\T \E{k}$ and $\TT{k} E$ can be understood here as subsets of $\T \TT{k-1} E$, and $\sigma^{[k]}$ is simply the projection $\E k\subset\T\TT{k-1}E\ra E\overset{\sigma}\ra M$. (Recall that for any manifold $M$, $\TT{k} M$ can be consider a subset of $\T \TT{k-1} M$ in a natural way.)

The bundle $\sigma^{[k]}\colon \E{k}\ra M$ has a very interesting algebraic structure encoded in a relation~$\kappa^{[k]}$ which can be defined directly as
 $\kappa^{[1]}=\kappa$, $\kappa^{[k]}=\big(\kappa^{k-1}_E\circ\TT{k-1}\kappa\big)\cap\big(\TT{k} E\times\T \E{k}\big)$ (see \cite[Proposition~4.6]{MJ_MR_models_high_alg_2015}):
\begin{gather}\label{eqn:diag_prolong_kappa}\begin{split}&
\xymatrix{
 \TT{k-1}\T E\ar@{-|>}[rr]^{\TT{k-1}\kappa}&& \TT {k-1}\T E \ar[rr]^{\kappa^{k-1}_E}&& \T\TT{k-1}E \\
\TT {k} E\ar@{_{(}->}[u]\ar@{--|>}[rrrr]^{\kappa^{[k]}}&&&& \T \E{k}.\ar@{_{(}->}[u]}\end{split}
\end{gather}
Since $\kappa$ is a \ZM, it is easy to see that so is $\kappa^{k-1}_E\circ\TT{k-1}\kappa$. We stress that the fact that the latter restricts fine to a \ZM\ from $\TT {k} E$ to $\T \E{k}$ is a non-trivial result which strongly depends on the fact that $\sigma$ has an AL algebroid structure.

In particular, if we start from the tangent algebroid $(\sigma,\kappa)=(\tau_M,\kappa_M)$, then $\sigma^{[k]}$ is just the higher tangent bundle $\tau^k_M\colon \TT{k}M\ra M$ and $\kappa^{[k]}$ is the canonical isomorphism $ \kappa^k_M$. In general, $\sigma^{[k]}\colon \E{k}\ra M$ is a \grB\ of rank $(r, r, \ldots, r)$ where $r$ is the rank of $E$ (see \cite[Theorem~4.5]{MJ_MR_models_high_alg_2015}). If $(\sigma, \kappa)$ is the Lie algebroid of a Lie groupoid $\mathcal{G}$ (i.e., it is a reduction of the tangent algebroid $\T\mathcal{G}$ by the action of $\mathcal{G}$), then its $\thh{k}$-prolongation $\A^k(\G)$ can be interpreted as a~reduction of the higher tangent bundle $\TT{k}\mathcal{G}$ by the action of $\G$. More precisely, $\A^k(\G) =\TT k \G^\alpha\big|_{M}$, where $M$ is the base of $\G$ and $\alpha\colon \G\ra M$ denotes the source map~\cite{MJ_MR_models_high_alg_2015}. That is, bundle~$\A^k(\G)$ consists of $k$-jets of curves in $\G$ which are tangent to the $\alpha$-fibres at the base.

To our best knowledge the $\thh{k}$-order prolongations of Lie algebroids were for the first time studied by Saunders~\cite{Saunders_2004}, who introduced bundles $\A^k(\G)$ discussed above. He also proposed an abstract construction of the second prolongation of a Lie algebroid which is not necessarily integrable. Later Colombo and De Diego~\cite{Col_deDiego_seminar_2011} generalized his ideas defining bundle $\E {k+1}$ (under the name of \emph{higher order Lie algebroid}) as the space of $\thh{k}$-tangent lifts of admissible curves in a~Lie algebroid $E$. Their construction was later adopted by Martinez~\cite{Martinez_2015} to develop higher-order variational calculus on Lie algebroids. Later in Remark~\ref{rem:var_calc} we briefly discuss his and other applications of the notion of a prolongation in variational calculus.

Let us remark that in all the above-cited sources prolongations are constructed out of Lie algebroids, whereas our results from \cite{MJ_MR_models_high_alg_2015} recognize the almost-Lie class as fundamental to define a prolongation. Moreover, in all these sources prolongations are understood only as bundles $\sigma^{[k]}\colon \E{k}\ra M$, and the prominent role of the \ZM\ $\kappa^{[k]}$, being the generalization of the algebroid bracket operation, remains unrecognized.

{\bf Properties of the prolongations.} Prolongations of AL algebroids provide an example of strong AL higher algebroids. Moreover, prolongations of Lie algebroids are Lie higher algebroids.
\begin{prop}[properties of the prolongations]\label{prop:E_k_as_AL_HA}
Let $(\sigma\colon E\ra M,\kappa)$ be a Lie $($resp.\ almost-Lie$)$ algebroid. Then $\big(\sigma^{[k]}\colon \E{k}\ra M, \kappa^{[k]}\big)$ is an example of a strong Lie $($resp.\ almost-Lie$)$ higher algebroid.
\end{prop}
The proof is given in Appendix~\ref{sec:app}.

\begin{ex}[a local form of a second order prolongation] \label{ex:prolong_2}
Let $\sigma\colon E\ra M$ be a vector bundle with local coordinates $(x^a,y^i)$ and consider an AL algebroid $(\sigma,\kappa)$ given locally by structure function $Q^b_i(x)=\wt Q^b_i(x)$ and $Q^i_{jk}(x)$ (cf.\ Remark~\ref{rem:local_kappa_a}). On $\TT 2 E$ consider induced coordinates $(x^a,y^i,\dot x^b,\dot y^j,\ddot x^c,\ddot y^k)$, and coordinates $\big(x^a,\und{y}^i,\und{y}^{j,(1)},\dot{\und{x}}^b,\dot{\und{y}}^k,\dot{\und{y}}^{l,(1)}\big)$ on $\T \E 2$. The inductive formula~\eqref{eqn:diag_prolong_kappa} allows to calculate the expression for $\kappa^{[2]}\colon \TT 2 E\relto \T \E 2$ (recall a general local form of a second-order algebroid given in Remark~\ref{rem:local_order_2}) which reads as
\begin{gather}\label{eqn:kappa_prolong_2}
\kappa^{[2]}\colon
\begin{cases}
\dot{x}^a = Q^a_i(x) \und{y}^i, \\
\ddot{x}^a = \dfrac{\pa Q^a_i}{\pa x^b}(x)Q^b_j(x)\und{y}^i\und{y}^j + Q^a_i(x) \und{y}^{i,(1)}, \\
\dot{\und{x}}^a = Q^{a}_i(x)y^i,\\
\dot{\und{y}}^i = \dot{y}^i + Q^i_{jk}(x) \und{y}^j y^k,\\
\dot{\und{y}}^{i,(1)} = \ddot{y}^i + Q^i_{jk}(x)\und{y}^j \dot{y}^k +Q^i_{jk}(x)\dot{\und{y}}^j {y}^k+ \dfrac{\pa Q^i_{jk}}{\pa x^a}(x)Q^a_l(x)\und{y}^l\und{y}^j {y}^k .
\end{cases}
\end{gather}
Observe that the equation for $\ddot{x}^i$ can be obtained from the formula $\dot{x}^a = Q^a_i(x) \und{y}^i$ by a formal derivation, assuming that $(\und{y}^i)^{\dot{}}=\und{y}^{i,(1)}$. In a similar manner, formula $\dot{\und{y}}^{i,(1)}$ can be derived from the equation $\dot{\und{y}}^i = \dot{y}^i + Q^i_{jk}(x) \und{y}^j y^k$. Thus it is instructive to think about $\kappa^{[2]}$ as of first-order differential consequences of $\kappa=\kappa^{[1]}$. We will use this property later in Section~\ref{ssec:ex_var_calc}.
\end{ex}

\section{Variational calculus on (Lie) higher algebroids}\label{sec:var_calc}

In this section we shall discuss a geometric formalism of the $\thh{k}$-order variational calculus based on the notion of a (Lie) higher algebroid $\big(\sigma^k\colon E^k\ra M,\kappa^k\big)$. Actually the geometry here is exactly the same as in our previous study on prolongations of AL algebroids~\cite{MJ_MR_prototypes}. The idea is that the relation~$\kappa^k$ is responsible for the construction of admissible variations of potential extremals of the system (the role of $\kappa$ is to change a `jet of curves' into a `curve of jets'~-- cf.\ our discussion in Section~\ref{sec:intro}). By dualizing $\kappa^k$ we may represent the variation of an action functional as a pairing of a certain curve of covectors with the $\thh{k}$-tangent lift of the curve of virtual displacements. Then all that is left to do is to perform an integration by parts according to the recipe from~\cite{MJ_MR_higher_var_calc_2013}.

Actually, it is straightforward to observe that the geometric formalism described above works well for even more general structures which we shall call \prealg s. We believe that they may have some potential applications in the theory of reductions. In fact, in the last paragraph of Section~\ref{ssec:var_calc} we provide a rather trivial example of that sort. From the point of view of complexity it is not interesting, yet still it is an example of a non-standard reduction~-- cf.\ Remark~\ref{rem:pre_alg}. On the other hand, more complex examples including the derivation of higher-order Euler--Poincar\'e and Hummel equations within our formalism can be found in \cite{MJ_MR_prototypes}.

\subsection{The formalism of variational calculus}\label{ssec:var_calc}

{\bf \Prealg s.}
\begin{df}[\prealg]\label{def:pre_algebroid} By a \emph{\prealg\ of order $k$} we understand a triple $\big(\sigma^k,\tau, \kappa^k\big)$ consisting of a \grB\ $\sigma^k\colon E^k\ra M$ of \degree\ $k$, a vector bundle $\tau\colon F\ra M$ (over the same base), and a weighted \ZM\ $\kappa^k\colon \TT{k}\tau\relto \tau_{E^k}$ inducing the identity on~$M$:
\begin{gather}\label{diag:pre_alg}\begin{split}&
\xymatrix{\TT{k}F\ar[d]^{\TT{k}\tau} \ar@{-|>}[rr]^{\kappa^k}&&\T E^k\ar[d]^{\tau_{E^k}}\\
\TT{k}M &&\ar[ll]_{\rho^k} E^k.}\end{split}
\end{gather}
The reduction of the relation $\kappa^k$ to order zero defines a vector bundle morphism $\rho^F = \kappa^0\colon F\ra \T M$. We call a \prealg\ $\big(\sigma^k, \tau, \kappa^k\big)$ \emph{almost-Lie} if $\big(\TT{k}\rho^F, \T \rho^k\big)\colon \kappa^k\mZM \kappa^k_M$ is a $\catZM$-morphism.
\end{df}
The difference with Definition \ref{def:higher_algebroid} is that we allow the weighted \ZM\ $\kappa^k$ to match the tangent bundle $\tau_{E^k}\colon \T E^k\ra E^k$ with the lifted bundle $\TT{k}\tau\colon \TT{k} F\ra\TT{k}M$ of an arbitrary VB $\tau\colon F\rightarrow M$, a priori not related with $\sigma^k$. By taking $\tau=\sigma^1$ and assuming that $\kappa^1$ gives the identity on the core bundles we recover the definition of a general HA. \Prealg s of order one have an elegant description in terms of a certain bracket operation. We discuss it briefly in the concluding Section~\ref{sec:final}.

It is easy to see that the construction of the algebroid lift can be straightforwardly extended to \prealg s. This time to a section $s\in\Sec_M(F)$ of $\tau$ and a number $\alpha=0,1,\hdots,k$ we assign a vector field $s^{[k-\alpha]}:=\wh{\kappa^k}\big(s^{(k-\alpha)}\big)\in\X\big(E^k\big)$.

\begin{rem}\label{rem:pre_alg}We believe that examples of \prealg s may appear in the theory of reductions. Let $G$ be a Lie group (or more generally a Lie groupoid). In the standard reduction of a~$\thh{k}$-order variational problem on $G$ one divides the higher tangent bundle $\TT{k}G$ by the natural action of $G$. In fact, here $G$ acts on curves in~$G$ by, say, left multiplication and the action on~$\TT{k}G$ is the differential consequence of this action on $\thh{k}$-jets. However, one can consider more general actions by a subgroup $H<\TT{k} G$ which properly contains $G$, and thus which would not be induced by the action of~$G$ on curves. In proper circumstances, the quotient $\TT{k} G/H$ will have a structure of a \grB\ (space)~$E^k$. However, in general the knowledge of~$E^1$ may not be enough to define admissible variations, if for example~$E^1$ is trivial but the entire~$E^k$ is not. Thus a reduction of $\kappa_G^k$ (if it can be properly defined) may lead to a natural example of a \prealg .

A very simple example in this direction is provided in the last paragraph of this subsection.
\end{rem}

{\bf Admissible paths and admissible variations.} Consider now a \prealg\ $\big(\sigma^k,\tau,\kappa^k\big)$. Let $\gamma\colon \R\ra E^k$ be a curve over $\und{\gamma}:=\sigma^k( \gamma)\colon \R\ra M$. Choose another curve $a\colon \R\ra F$ and consider its $\thh{k}$-order tangent lift $\tclass{k}{a}\colon \R\ra \TT{k} F$. Along $\gamma$ we would like to construct a vector field $\del_a\gamma$ by the formula
\begin{gather}\label{eqn:adm_variation}
 \delta_a\gamma (t):=\big(\kappa^k\big)_{\gamma(t)}\big( \tclass{k}{a}(t)\big)\in \T_{\gamma (t)}E^k.
\end{gather}
Note that, by the properties of $\kappa^k$, this definition is correct if and only if $\gamma$ and $a$ share the same base path, i.e., $\sigma^k(\gamma)=\und{\gamma}=\tau(a)$ and if $\tclass{k}{\und{\gamma}}=\rho^k(\gamma)$. The latter equation is known as the (left) \emph{admissibility} of $\gamma$. If these conditions hold, we call $\delta_a\gamma$ an \emph{admissible variation} along $\gamma$ and $a$ its \emph{generator} or a \emph{virtual displacement}. The set of all admissible curves on the considered \prealg\ will be denoted by $\Adm\big(E^k\big)$, while the set of all admissible variations by $\VA\big(E^k\big)$.

Let us explain a relation between the notions of admissible variations and \emph{variations} as understood in the standard variational calculus. In a general setting, a variation $\delta \gamma$ of an $E^k$-valued path $\gamma$ is a vector field along $\gamma$, i.e., $\delta\gamma\colon t\mapsto \T_{\gamma(t)} E^k$. Hence a variation $\delta \gamma$ can be represented by a family of paths $\gamma_s\colon \R\ra E^k$, where $\gamma_0 = \gamma$ and $\delta\gamma(t) = \tclass 1{(s\mapsto \gamma_s(t))}$. In a~very general, but intuitive sense, a variation $\delta \gamma$ is tangent to a subset $\mathcal{A}$ of paths in $E^k$ if it can be represented by a family $\gamma_s$ lying in $\mathcal{A}$. Results of \cite{Martinez_2008} and \cite[Theorem~3]{GG_var_calc_gen_alg_2008} show that for an algebroid of order one the subspace $\Adm\big(E^1\big)$
of admissible curves is a~Banach submanifold in the space of all $C^1$-paths in~$E^1$. Moreover, the condition $\T \Adm\big(E^1\big) =\VA\big(E^1\big)$ is equivalent to the AL axiom, i.e., for AL algebroids admissible variations are precisely variations tangent to the space of admissible curves.

 In the higher-order case we can prove the following two Lemmas in this direction. Actually, the presented reasonings simplify the proofs available in the literature for the order-one case.

\begin{lem}[tangent space for admissible variations] \label{lem:adm1} Let $\delta \gamma \in \T_\gamma E^k$ be a vector field along an admissible path $\gamma$ in a \prealg\ $\big(\sigma^k \colon E^k \ra M,\tau\colon F\ra M, \kappa^k\big)$. If $\delta \gamma$ is tangent to $\Adm\big(E^k\big)$ then{\samepage
\begin{gather}\label{eqn:Adm_var}
\big(\T \rho^k\big) (\delta \gamma) = \kappa^k_M\big(\tclass{k}{\und{\delta\gamma}}\big),
\end{gather}
where $\tclass{k}{\und{\delta\gamma}}$ denotes the $\thh{k}$-tangent lift of the curve $\und{\delta\gamma} = \T \sigma^k (\delta \gamma)$.}
\end{lem}
\begin{proof}Assume that $\delta \gamma$ is a variation of an admissible path $\gamma\in\Adm\big(E^k\big)$ tangent to $\Adm\big(E^k\big)$, i.e., $\delta \gamma(t) = \tclass{1}{[s\mapsto \gamma(t,s)]}$ and paths $t\mapsto \gamma(t,s)$ are admissible for each $s$. Then
\begin{gather*}
\T \rho^k (\delta\gamma)(t) = \tclass{1}{\big[s\mapsto \rho^k(\gamma(t, s))\big]} = \tclass{1}{[s\mapsto \tclass{k}{[t\mapsto \und{\gamma}(t, s)]}]} \in \T \TT{k} M,
\end{gather*}
while $ \tclass{k}{\und{\delta \gamma}} = \tclass{k}{\big[t\mapsto\tclass{1}{[s\mapsto \und{\gamma}(t, s)]}\big]}$. Due to the definition of $\kappa^k_M$, formula \eqref{eqn:Adm_var} follows.
\end{proof}

\begin{rem}[problems with the Banach-manifold structure]\label{rem:banach_problem} According to~\cite{Martinez_2008}, for $k=1$ the converse of above lemma is also true, i.e., equation~\eqref{eqn:Adm_var} describes the space $\T\Adm\big(E^1\big)$. For $k>1$ some additional conditions are necessary to assure that the subspace of admissible curves in~$E^k$ is a Banach submanifold. Such conditions may follow directly from the geometry of the considered problem. For example in our preprint~\cite{MJ_MR_prototypes} additional assumptions imposed on admissible paths were motivated by the reduction procedure.
\end{rem}
\begin{lem}[a property of AL \prealg s]\label{lem:adm2} If a \prealg\ $\big(\sigma^k\colon E^k \ra M,\tau\colon F\ra M, \kappa^k\big)$ is almost-Lie, $\gamma$ is an admissible path, and $a$ is a curve in $F$ such that $\tau(a)=\und{\gamma}$, then the admissible variation $\delta_a\gamma$ satisfies equation~\eqref{eqn:Adm_var}.
\end{lem}
\begin{proof}Consider a diagram
\begin{gather*}
\xymatrix{\TT{k} F\ar[d]^{\TT{k}\rho^F} \ar@{-|>}[rr]^{\kappa^k} &&\T E^k\ar[d]^{\T \rho^k}\\
\TT{k} \T M \ar[rr]^{\kappa^k_M} &&\T \TT{k} M,}
\end{gather*}
which is commutative as our \prealg\ $\big(\sigma^k, \tau, \kappa^k\big)$ is almost-Lie. As $\big(\tclass{k}{a}, \delta_a\gamma\big)\in \kappa^k$, we get $\T \rho^k (\delta_a \gamma) = \kappa^k_M\big(\TT{k} \rho^F\big(\tclass{k}{a}\big)\big) = \kappa^k_M\big(\tclass{k}{\rho^F(a)}\big) = \kappa^k_M\big(\tclass{k}{\und{\delta_a \gamma}}\big)$. The last equality follows from the commutativity of the diagram
\begin{gather*}
\xymatrix{\TT{k} F\ar[d]^{\tau^k_F} \ar@{-|>}[rr]^{\kappa^k}&&\T E^k\ar[d]^{\T \sigma^k}\\
F \ar[rr]^{\rho^F} && \T M, }
\end{gather*}
from which we get $\rho^F(a) = \und{\delta_a \gamma}$ as $\tclass{k}{a}\sim_{\kappa^k} \delta_a\gamma$.
\end{proof}

{\bf Variational calculus.} We understand Lagrangian mechanics on a $\thh{k}$-order \prealg\ $\big(\sigma^k,\tau, \kappa^k\big)$ as the study how a functional, obtained by integrating the Lagrangian function $L\colon E^k\ra \R$ along an admissible curve $\gamma\colon \R\ra E^k$, behaves under movement in the direction of admissible variations. By construction
\begin{gather}\label{eqn:var_formula_0}
\<\dd L(\gamma),\del_a\gamma>_{\tau_{E^k}}=\big\langle \dd L(\gamma),\big(\kappa^k\big)_{\gamma}\big(\tclass{k}{a}\big)\big\rangle_{\tau_{E^k}}=\big\langle { \big(\kappa^k\big)}^\ast(\dd L(\gamma)), \tclass{k}{a}\big\rangle_{\TT{k}\tau},
\end{gather}
where ${\big(\kappa^k\big)}^\ast\colon \TT\ast E^k\ra \TT{k} F^\ast$ is a morphism of graded-linear bundles dual to $ \kappa^k$. We have thus arrived at the pairing of a $\TT{k} F^\ast$-valued curve with the $\thh{k}$-tangent lift $\tclass{k}{a}\colon \R\ra\TT{k} F$ of the generator $a$ of $\del_a\gamma$. Using our earlier results on the $\thh{k}$-order geometric integration by parts -- see \cite{MJ_MR_higher_var_calc_2013} -- we may present this pairing as a sum of a complete derivative, and a paring $\<\cdot,\cdot>_{\tau}\colon F^\ast\times F\ra \R$ evaluated on the generator $a$
\begin{gather}\label{eqn:var_formula}\big\langle {\big(\kappa^k\big)}^\ast(\dd L(\gamma)), \tclass{k}{a}\big\rangle_{\TT{k}\tau}=\big\langle \mathcal{EL}_k\big(\tclass{k}{\gamma}\big),a\big\rangle_{\tau}+\dbydt{\big\langle \mathcal{P}_k\big( \tclass{k-1}{\gamma}\big),\tclass{k-1}a\big\rangle_{\TT{k-1}\tau}}.
\end{gather}
 Here $\mathcal{EL}_k(\cdot)$ denotes the Euler--Lagrange operator and $\mathcal{P}_k(\cdot)$ the momentum operator associated with the Lagrangian (see \cite{MJ_MR_higher_var_calc_2013} for the precise definitions). An admissible curve $\gamma$ satisfies \emph{Euler--Lagrange equations} \mbox{$\mathcal{EL}_k\big( \tclass k {\gamma}\big)=0$} (i.e., is a \emph{trajectory} of the considered Lagrangian system) if and only if for every virtual displacement $a$ we have
\begin{gather}\label{eqn:var_calc}
\<\dd L(\gamma),\del_a\gamma>_{\tau_{E^k}}=\dbydt{\big\langle \mathcal{P}_k\big( \tclass{k-1}{\gamma}\big), \tclass{k-1}a\big\rangle_{\TT{k-1}\tau}},
\end{gather}
or in a more familiar integral version
\begin{gather*}\int_0^T\<\dd L(\gamma),\del_a\gamma>_{\tau_{E^k}}\dd t=\big\langle\mathcal{P}_k\big( \tclass{k-1}{\gamma}\big) ,\tclass{k-1}a\big\rangle_{\TT{k-1}\tau}\big|^T_0.\end{gather*}
The presented formalism generalizes the construction of geometric Lagrangian mechanics on a~general algebroid from \cite{GG_var_calc_gen_alg_2008}.

\begin{rem}\label{rem:var_calc}In the literature there have been a few attempts to develop higher-order mechanics on generalizations of Lie algebroids. With the sole exception of \cite{BGG_2015}, all trials known to us used the structure of a prolongation of a Lie algebroid $\big(\E{k},\kappa^{[k]}\big)$ (see Section~\ref{ssec:prolongations}), naturally appearing in reduction of higher tangent bundles of Lie groupoids.

As we already pointed out when discussing the notion of an AL algebroid prolongation for the first time, the role of the \ZM\ $\kappa^{[k]}$ in the structure of an algebroid prolongation remained unrecognised in all the primary sources \cite{Col_deDiego_seminar_2011, Martinez_2015, Saunders_2004}. On the other hand, in the theory presented above, $\kappa^{[k]}$ is responsible for the construction of admissible variations. Mart\`{i}nez in \cite{Martinez_2015} constructed the $\thh{k}$-algebroid lift $s^{[k]}$ of a section $s\in\Sec_M(E)$ directly form the first-order data $(E,\kappa)$ (which should not be surprising as the prolongation is completely determined by its base object). Then he defined an admissible variation generated by $s$ as a restriction of $s^{[k]}$ to an admissible trajectory $\gamma$. Apparently, his approach is equivalent to ours as $s^{[k]}|_\gamma=\delta_a\gamma$ for $a=s|_{\und{\gamma}}$ (see the discussion of symmetries and conservation laws below). Summing up, on the prolongation of a Lie algebroid the geometry of Mart\`{i}nez's formalism is the same as ours, despite the presence of $\kappa^{[k]}$ not being directly observed.

The case of $k=2$ prolongations was intensively studied by Colombo and his collaborators (see \cite{Abrunheiro_Colombo_2018, Colombo_2017} and the references therein). In his formalism the second prolongation $\E 2$ is treated as a subset in the prolongation of $E$ along the projection $\tau\colon E\ra M$ (see Example~\ref{ex:prolong}), i.e., $\E 2\subset \mathcal{T}^EE$. Now the initial problem on $\E 2$ can be treated as a constrained problem on $ \mathcal{T}^EE$, the latter being a first-order algebroid. Equations of motion are derived by using the first-order formalism of \cite{Martinez_geom_form_mech_lie_alg_2001} applied to $\mathcal{T}^EE$. This method seems to be quite complicated (prolongations of prolongations are required) in comparison to our approach, the presence of $\kappa^{[2]}$ is hidden somewhere inside the inclusion $\E 2\subset \mathcal{T}^EE$, and the direct relation with variational calculus is hard to observe.

The approach of \cite{BGG_2015} is conceptually similar to the one of Colombo, but much more general. Given a \grB\ $F^k$, the authors consider a first-order algebroid structure on its \emph{linearisation} $D\big(F^k\big)$. Again we have a canonical inclusion $F^k\subset D\big(F^k\big)$, and the $\thh{k}$-order dynamics on $F^k$ is derived as the first-order constrained dynamics on the algebroid $D\big(F^k\big)$. The relation of the methods of \cite{BGG_2015} to the ones of this work is so far unclear, although in the case of prolongation of algebroids both approaches give consistent results.
\end{rem}

{\bf Symmetries and conservation laws.} As explained above, Euler--Lagrange equations on higher \prealg s are derived by studying changes of the Lagrangian $L(\gamma)$ in the direction of every admissible variation $\del_a\gamma$. By contrast, conservation laws are related with special properties of a particular admissible variation. Below we will sketch basic concepts standing behind conservation laws in higher-order variational calculus.

Let $\gamma\colon \R\ra E^k$ be an admissible curve over $\und{\gamma}\colon \R\ra M$. The idea is to look for a generator $a\colon \R\ra F$ of an admissible variation $\del_a\gamma$ such that $\<\dd L(\gamma (t)),\del_a\gamma (t)>=\dbydt f (t)$ for some smooth function $f\colon \R\ra \R$. Now if $\gamma$ is a solution of the Euler--Lagrange equations, by \eqref{eqn:var_calc} we have $\dbydt f=\dbydt{\<\mathcal{P}_k\big(\tclass{k-1}{\gamma}\big), \tclass{k-1}a>_{\TT{k-1}\tau}}$, and hence the difference $f-\<\mathcal{P}_k\big( \tclass{k-1}{\gamma}\big), \tclass{k-1}a>_{\TT{k-1}\tau}$ is constant along $\gamma$.

In practice such a method of finding constants of motion is not very effective, as $a$ is defined only along $\und{\gamma}$ which is unknown until we actually solve (at least partially) the Euler--Lagrange equations $\mathcal{EL}_k\big(\tclass k{\gamma}\big)=0$. Instead it is better to look for a generator universal for every possible~$\gamma$ by considering $a=s|_{\und{\gamma}}$, where $s\in\Sec_M(F)$ is some section of $\tau$. Note that in this case \mbox{$s^{[k]}|_{\gamma}=\del_a\gamma$}, i.e., the admissible variation $\del_a\gamma$ is the restriction of the $\thh{k}$-order algebroid lift~$s^{[k]}$ of~$s$.
Note that $\<\dd L,s^{[k]}>_{\tau_{E^k}}$ is a smooth function on $E^k$. If $\<\dd L,s^{[k]}>_{\tau_{E^k}}=\dd f \circ \rho^k$ for some smooth function $f\in \Cf(\TT{k-1}M)$ (note that $\dd f$ may be regarded as a map $\dd f\colon \TT{k}M\subset\T\TT{k-1}M\ra \R$), we call the section $s\in\Sec_M(F)$ a \emph{generator of the symmetry of $L$}. If this is the case then
\begin{gather*}
\<\dd L(\gamma),\del_a\gamma>=\big\langle\dd L|_{\gamma},s^{[k]}|_{\gamma}\big\rangle_{\tau_{E^k}}=\dd f\big(\rho^k(\gamma)\big)=\dd f\big( \tclass {k}{\und{\gamma}}\big)=\dbydt{f\big(\tclass{k-1}{\und{\gamma}}\big)},
\end{gather*}
 i.e., $\<\dd L(\gamma),\del_a\gamma>$ is a total derivative regardless of the choice of an admissible curve $\gamma$, as intended.
If follows that
$f\big( \tclass{k-1}{\und{\gamma}}\big)-\<\mathcal{P}_k\big(\tclass{k-1}{\gamma}\big), {\T^{k-1}{s}\circ \rho^{k-1}(\gamma)}>_{\TT{k-1}{\tau}}$ is constant, i.e., we derived a~\emph{conservation law}
related with the symmetry of~$L$.

{\bf A simple example of a prealgebroid.} Let us end this part with a very simple example of a~variational problem on a \prealg . Our starting point is a standard second-order variational problem on the Euclidean plane $\R^2$, constituted by a Lagrangian function \mbox{$L\colon \TT{2}\R^2\ra \R$}. In this case the HA in question is simply the second tangent bundle $\tau^2_M\colon \TT{2} \R^2\ra \R^2$ equipped with the standard higher algebroid structure $\kappa^2_{\R^2}\colon \TT{2}\T\R^2\overset{\simeq}{\lra}\T\TT{2}\R^2$. Formula \eqref{eqn:var_formula} leads to the standard Euler--Lagrange equations:
\begin{gather}\label{eqn:EL1}
\dbydtk{2}\left(\frac{\pa L}{\pa \ddot x}\right)-\dbydt \left(\frac{\pa L}{\pa \dot x}\right)+\frac{\pa L}{\pa x}=0\qquad\text{and}\qquad \dbydtk{2}\left(\frac{\pa L}{\pa \ddot y}\right)-\dbydt \left(\frac{\pa L}{\pa \dot y}\right)+\frac{\pa L}{\pa y}=0,
\end{gather}
where $(x,y)$ are standard coordinates on $\R^2$ and $(x,y,\dot x,\dot y,\ddot x,\ddot y)$ the adapted coordinates on~$\TT{2}\R^2$.

Consider now the following action of $(a,b,c)\in \R^3$ on (the germs at $t=0$ of) curves in $\R^2$:
\begin{gather*}(a,b,c)\circ (x(t),y(t))=(x(t)+a+\dot x(0)bt, y(t)+c).\end{gather*}
Obviously this action reduces to the $\R^3$-action on $\TT{2}\R^2$ given by
\begin{gather*}(a,b,c)\circ (x,y,\dot x,\dot y,\ddot x,\ddot y)=(x+a,y+c,\dot x(1+b),\dot y,\ddot x,\ddot y).\end{gather*}
If $L$ is invariant under this action, then \eqref{eqn:EL1} reduces to
\begin{gather}\label{eqn:EL2}
\dbydtk{2}\left(\frac{\pa L}{\pa \ddot x}\right)=0\qquad\text{and}\qquad \dbydtk{2}\left(\frac{\pa L}{\pa \ddot y}\right)-\dbydt \left(\frac{\pa L}{\pa \dot y}\right)=0.
\end{gather}

On the other hand, the latter equations can be obtained within the framework of variational calculus on \prealg s by an easy reduction procedure. Namely, note that the quotient of~$\TT{2}\R^2$ by the action of~$\R^3$ is naturally a graded space (i.e., a \grB\ over a point) $\sigma^2\colon E^2=\R[1]\oplus\R^2[2]\ra\{\pt\}$. On $E^2$ we can introduce graded coordinates $(y_1,x_2,y_2)$ induced by $\dot y$, $\ddot x$ and $\ddot y$, respectively.

It is clear that an invariant Lagrangian $L$ induces a function $l\colon E^2\ra\R$ such that $L=p\circ l$ for $p\colon \TT{2}\R^2\ra E^2=\TT{2}\R^2/\R^3$ being the quotient map.

We can further equip $\sigma^2$ with a \prealg\ structure
\begin{gather*}\xymatrix{\TT{2}\R^2\ar[d]^{\TT{2}\tau}\ar@{-|>}[rr]^{\kappa^2}&&\T E^2\ar[d]^{\tau_{E^2}}\\
\{\pt\}&& E^2\ar[ll]}\end{gather*}
as follows
\begin{gather*}\big(\kappa^2\big)_{(y_1,x_2,y_2)}\colon \ (a,b,\dot a,\dot b,\ddot a,\ddot b)\mapsto (y_1,x_2,y_2,\dot y_1=\dot b,\dot x_2=\ddot a,\dot y_2=\ddot b).\end{gather*}

For this structure the admissibility condition is empty, since $\tau\colon \R^2\ra\{\pt\}$ has a trivial base. However, an additional condition
\begin{gather}\label{eqn:adm_ex}
\dot y_1=y_2
\end{gather}
should be imposed on admissible curves if we want to maintain the correspondence of admissible curves in $E^2$ with the admissible curves in $\TT{2}\R^2$ under the reduction procedure (cf.\ Remark~\ref{rem:banach_problem}).

It is now an easy exercise to show that equations \eqref{eqn:EL2} are the Euler--Lagrange equations for the variational problem on $\big(\sigma^2,\tau,\kappa^2\big)$ constituted by function $l\colon E^2\ra\R$ and the admissibility condition \eqref{eqn:adm_ex}.

The above example suggests that \prealg s may naturally appear in reductions of standard variational problems by actions which are non-trivial on higher jets. (Note that in the standard Lie groupoid--Lie algebroid reduction \cite{Mackenzie_lie_2005} and its generalization to higher jets \cite{MJ_MR_models_high_alg_2015} the action of the groupoid on higher jets is merely a consequence of its action on points -- see Remark~\ref{rem:pre_alg}.)

\subsection{Examples}\label{ssec:ex_var_calc}

Now we shall provide some basic examples of higher-order Euler--Lagrange equations obtained using the formalism introduced in Section~\ref{ssec:ex_var_calc}. For simplicity we will reduce our attention to the case $k=2$.

{\bf Variational problem on a prolongation of an AL algebroid.} In Example~\ref{ex:prolong_2} we derived the local form of the second prolongation $\big(\E 2,\kappa^{[2]}\big)$ of an AL algebroid structure $(E,\kappa)$. Now we can use it to obtain a coordinate formula of the second-order Euler--Lagrange equations. Consider an admissible curve $\gamma(t)=\big(x^a(t),y^i(t),y^{j,(1)}(t)\big)\in \E 2$. From~\eqref{eqn:kappa_prolong_2} the admissibility conditions for $\gamma(t)$ are $\dot{x}^a = Q^a_i(x) y^i$ and $\ddot{x}^a = \frac{\pa Q^a_i}{\pa x^b}(x)Q^b_j(x)y^i y^j + Q^a_i(x) y^{i,(1)}$. We can deduce the latter equation from the former under an additional assumption that $y^{i,(1)}(t)=\dot y^i(t)$, which can be easily justified by the reduction procedure (see~\cite{MJ_MR_prototypes}). A~generator $a(t)=(x^a(t),\xi^i(t))\in E_{\gamma(t)}$ produces the following admissible variation
\begin{gather*}
\delta_a\gamma(t)\colon \
\begin{cases}
\delta{x}^a = Q^{a}_i(x) \xi^i(t),\\
\delta{y}^i = \dot{\xi}^i + Q^i_{jk}(x)\,y^j \xi^k,\\
\delta y^{i,(1)} = \ddot{\xi}^i + Q^i_{jk}(x)y^j \dot \xi^k +Q^i_{jk}(x)\dot{y}^j \xi^k+ \dfrac{\pa Q^i_{jk}}{\pa x^a}(x)Q^a_l(x)y^l y^j \xi^k= \big(\delta y^i\big)^{\cdot}.
\end{cases}
\end{gather*}
Now, for the Lagrangian $L\colon \E 2\ra\R$, the left hand-side of \eqref{eqn:var_formula_0} reds as
\begin{gather*}\<\dd L(\gamma),\del_a\gamma>=\frac{\pa L\ }{\pa x^a}\big(x,y,y^{(1)}\big)\delta x^a+\frac{\pa L\ }{\pa y^i}\big(x,y,y^{(1)}\big)\delta y^i+\frac{\pa L\ \ }{\pa y^{i,(1)}}\big(x,y,y^{(1)}\big)\delta y^{i,(1)}.\end{gather*}
Expressing it in terms of $\xi^i$'s and integrating by parts leads us to
\begin{gather*}\<\dd L(\gamma),\del_a\gamma> = \<\left(\delta^k_i\frac{\dd\ }{\dd t}-Q^k_{ij}(x)y^j\right)\left(\frac{\dd\ }{\dd t}\frac{\pa L\ \ }{\pa y^{k,(1)}}-\frac{\pa L\ }{\pa y^k}\right)+Q^a_i(x)\frac{\pa L\ }{\pa x^a},\xi^i >\\
\hphantom{\<\dd L(\gamma),\del_a\gamma> =}{} +\frac{\dd\ }{\dd t}\left(\<\frac{\pa L\ }{\pa y^{i,(1)}},\dot \xi^i>-\<\left(\delta^k_i\frac{\dd\ }{\dd t}-Q^k_{ij}(x)y^j\right)\frac{\dd\ }{\dd t}\frac{\pa L\ }{\pa y^{k,(1)}}-\frac{\pa L\ }{\pa y^i},\xi^i>\right) .
\end{gather*}
And hence (cf.~\eqref{eqn:var_formula}) the \emph{second-order Euler--Lagrange equations} become
\begin{gather*}
\left(\delta^k_i\frac{\dd\ }{\dd t}-Q^k_{ij}(x)y^j\right)\left(\frac{\dd\ }{\dd t}\frac{\pa L\ }{\pa y^{k,(1)}}-\frac{\pa L\ }{\pa y^k}\right)+Q^a_i(x)\frac{\pa L\ }{\pa x^a}=0\qquad\text{and}\qquad
\dot{x}^a = Q^a_i(x) \,y^i.
\end{gather*}

{\bf Euler--Poincar\'{e} equations.} Let $(\g,[\cdot,\cdot])\simeq(\g,\kappa_{\g})$ be a Lie algebra, and let $l\colon \T\g\simeq \g\times\g\ra \R$ be a Lagrangian function its the second prolongation $\big(\g^{[2]}=\T\g,\kappa_{\g}^{[2]}\big)$. By setting in the previous example $Q^a_i(x)=0$ and $Q^i_{jk}(x)=c^i_{jk}$, where $c^i_{jk}$ denote the structural constants of~$\g$ in a given basis~$\{e_j\}$, we easily arrive at the \emph{second-order Euler--Poincar\'{e} equations} for a~curve $(a(t),\dot a(t))\in \T\g$:
\begin{gather}\label{eqn:EP}
\left(\frac{\dd\ }{\dd t}-\operatorname{ad}^\ast_{a(t)}\right)\left(\frac{\dd \ }{\dd t}\frac{\pa l}{\pa \dot a}-\frac{\pa l}{\pa a}\right)(a(t),\dot a(t))=0.
\end{gather}
An alternative, more geometric derivation is possible directly from \eqref{eqn:adm_variation} and \eqref{eqn:var_formula_0} using formula~\eqref{eqn:kappa_l_GG_expd}~-- see~\cite{MJ_MR_prototypes} for the treatment of the general case. The obtained results are consistent with these of Colombo and De Diego~\cite{Colombo_Diego_2014}.

\section[Further examples -- substructures and quotients of $\TT{k} G/G$]{Further examples -- substructures and quotients of $\boldsymbol{\TT{k} G/G}$}\label{sec:examples}

Another interesting class of examples of (Lie) higher algebroids is obtained from the reduction $\TT{k} G / G$ of the $\thh{k}$-tangent bundle of a Lie group $G$. The resulting space $E^k := \TT{k}_e G$, consisting of $k$-velocities in $G$ based at the identity element $e\in G$, is a \grB\ over a single point $e\in G$ (i.e., a \emph{graded space} -- see \cite{JG_MR_gr_bund_hgm_str_2011}). It can be equipped with the canonical HA structure~$\kappa^k_\g$ which, in fact, can be identified with the $\thh{k}$-prolongation (in the sense of Section~\ref{ssec:prolongations}) of~$(\g,[\cdot,\cdot])$~-- the Lie algebra of the group~$G$.

Throughout this part we shall describe all higher subalgebroids of $\TT{k} G / G$, and give examples of HA quotients of $\TT{k} G / G$ by which we understand higher algebroids $(F^k, \kappa^k)$ obtained from $\TT{k} G / G$ by means of surjective HA morphism onto $F^k$.

{\bf A (Lie) higher algebroid structure on $\boldsymbol{\TT{k}_eG}$.}
 We shall begin by describing the higher algebroid structure on $\TT{k}_eG$. First of all, $\TT{k}_e G$ is a split graded space concentrated in weights $1,2, \ldots, k$ in which each homogeneous component is identified with $\g$, the Lie algebra of $G$. This canonical identification is possible thanks to the group structure on $G$ and is obtained by means of the local diffeomorphism $\exp\colon \g\ra G$ inducing an isomorphism $\TT{k}_0 \exp\colon \TT{k}_0\g \xrightarrow{\simeq} \TT{k}_e G$ of graded spaces. (Note that in general, for $k\geq 2$, the $\thh{k}$-tangent space of a manifold at a given point has no canonical vector space structure.) Next, $\TT{k}_0 \g$ (and so $\TT{k}_e G$) has a canonical identification with $\TT{k-1} \g$ which is, on the other hand, a graded Lie algebra (see Proposition~\ref{p:algebroid_Tk_E}), yet the latter gradation is shifted by $-1$ with respect to the former. Summing up, $\TT{k}_e G$ is canonically equipped with two structures: of a graded space (which is split) and of a graded Lie algebra. Both structures are clearly recognized after canonical identifications $\TT{k}_e G = \TT{k-1}\g = \g\otimes \R[t]/\ideal{t^k}$.

Now we would like to recall a \ZM\ $\kappa^k_\g\colon \TT{k}\g\relto \T\TT{k-1}\g$ constituting a HA structure on $\TT{k}_eG\simeq\TT{k-1}\g$ (see \cite[Section~6]{MJ_MR_models_high_alg_2015} for details). An element $X\in\TT{k}\g$ represented by a curve $t\mapsto \sum_{j=0}^k X_j t^j\in\g$ can be identified with a $(k+1)$-tuple $(X_0,X_1,\hdots,X_k)$ of elements of $\g$. In a~similar manner a vector $Y \in \T \TT{k-1} \g\simeq \TT{k-1}\g\times\TT{k-1}\g$ (note that $\T\TT{k-1}\g$ is a tangent bundle of a~vector space), can be identified with a $2k$-tuple $(Y_0,\hdots,Y_{k-1},\dot Y_0,\hdots,\dot Y_{k-1})$ of elements of~$\g$. Now $(X,Y)\in\kappa^k_\g$ if and only if
\begin{gather}\label{eqn:kappa_l_GG_expd}
\dot Y_{l} = (l+1) X_{l+1} - \sum_{i+j=l} [X_i, Y_j],\qquad\text{for every} \quad l=0,1,\hdots,k-1.
\end{gather}
The details of this relation can be found in \cite[Proposition~6.3]{MJ_MR_models_high_alg_2015}.

{\bf Higher subalgebroids of $\boldsymbol{\TT{k}_e G}$.} Let $F^k$ be a split graded subspace of $\T_e^k G \simeq \g[t]/\ideal{t^k}$, so there are linear subspaces $V_j\subset \g$ ($0\leq j\leq k-1$) such that $F^k = \big\{\sum_{j=0}^{k-1} v_j t^j +\ideal{t^k}\colon v_j\in V_j\big\}$. We recall that homogeneous functions of weight $j+1$ on $\T_e^k G$ correspond to linear functions on the summand $t^j \g$.

\begin{prop}[subalgebroids of $\TT{k}_eG$]\label{prop:sub_alg_TkG} A split graded subbundle $F^k \subset \TT{k}_e G$ is a higher subalgebroid if and only if
\begin{gather}\label{eqn:subTkg}
 \TT{k-1} V_0 \subset F^k \qquad\text{and}\qquad [V_0,V_i]\subset V_j\qquad\text{for each} \quad i=0,1,\hdots,k-1 \quad \text{and} \quad j \geq i.\!\!\!
\end{gather}
 In particular, $V_0$ should be a Lie subalgebra of $\g$ and additionally each subspace $V_i$ should contain~$V_0$ and be preserved by the action of the latter subspace.
\end{prop}
\begin{proof} We should basically check if $\kappa^k_\g$ restricts fine to a \ZM\ from $\TT{k}V_0$ to $\T F^k$.
Take any $\und{Y}=(Y_0, \hdots, Y_{k-1})\in F^k \subset \TT{k-1}\g$ and $X= (X_0,\hdots,X_k)\in \TT{k} V_0 \subset\TT{k}\g$. So all the components~$X_j$ are in~$V_0$ and $Y_i\in V_i$ for $i=0,\hdots,k-1$. Let $Y := (\und{Y}, \dot{\und{Y}}) = \big(\kappa^k_\g\big)_{\und{Y}}(X)$ be the element in $\T_{\und Y}\TT{k-1}\g$ which is $\kappa^k_\g$-related to $X$. The components of $\dot{\und Y}=(\dot Y_0,\hdots,\dot Y_{k-1})$ are determined by~\eqref{eqn:kappa_l_GG_expd}. We should examine when $\dot{\und Y}\in F^k$, i.e., when $\dot Y_i\in V_i$ for each $i=0,1,\hdots,k-1$. This happens if and only if
\begin{gather*}
 V_0 + \sum_{i=0}^j [V_0, V_i] \subset V_j
\end{gather*}
for each $j=0,1,\hdots,k-1$. This gives \eqref{eqn:subTkg}.
\end{proof}

{\bf Quotients of $\boldsymbol{\TT{k}_e G}$.} General quotients of Lie algebroids and the concept of an ideal in a~Lie algebroid are quite involved subjects (see~\cite{Mackenzie_lie_2005}). For this reason in the following definition of a quotient (Lie) higher algebroid structure we shall restrict our attention only to ``quotient'' maps covering the identity on the base. We call a \grB\ morphism $\phi^k\colon E^k\ra F^k$ \emph{surjective} if for each $1\leq j\leq k$ the induced top core vector bundle morphisms $\core{\phi^j}\colon \core{E^j}\ra \core{F^j}$ are fibre-wise surjective linear maps. In particular, such a morphisms always is a surjective mapping $E^k\ra F^k$ but not vice versa: a \grB\ morphism which is a surjective mapping need not induce surjective mapping between the core bundles.

\begin{df}[a quotient of a (Lie) HA]\label{df:quotient} We shall say that a $\thh{k}$-order algebroid $\big(\sigma^k_F\colon F^k\ra M, \kappa^{k,F}\big)$ is a \emph{quotient} of a (Lie) higher algebroid $\big(\sigma^k_E\colon E^k\ra M, \kappa^{k,E}\big)$ if there is a surjective \grB\ morphism $\phi^k\colon E^k \ra F^k$ covering the identity on the base manifolds such that $\big(\TT{k} \phi^1, \T\phi^k\big)$ is a $\catZM$-morphism $ \kappa^{k,E} \mZM \kappa^{k,F}$. In other words, elements $X'\in \TT{k} F^1$ and $Y'\in \T F^k$ are $ \kappa^{k,F}$-related if and only if they are images under projections $\TT{k} \phi^1$ and $\T \phi^k$, respectively, of some $ \kappa^{k,E}$-related elements $X\in \TT{k} E^1$ and $Y\in \T E^k$:
\begin{gather}\label{diag:df_quotient_kappa_k}\begin{split} &
\xymatrix{
\TT{k} E^1 \ar[d]_{T^k\phi^1} \ar@{-|>}[rr]^{ \kappa^{k,E}} && \T E^k\ar[d]^{\T \phi^k}\\
\TT{k} F^1 \ar@{--|>}[rr]^{ \kappa^{k,F}} && \T F^k.}\end{split}
\end{gather}
\end{df}

\begin{ex}[a class of examples]A slight generalization of the example $\big(\TT{k-1}\g, \kappa^k_{\g}\big)$ is possible. Let $E^k = \bigoplus_{i=0}^{k-1} \g_i$ be a graded Lie algebra. In particular, $\g_0$ has a Lie algebra structure and{, consequently, $\TT{k} \g_0$ also} is graded Lie algebra. Let $\alpha\colon \TT{k-1}\g_0 \ra E^k$ be a graded Lie algebra homomorphism such that $\alpha_0 = \id_{\g_0}$, where $\alpha_0$ is the restriction of~$\alpha$ to the subalgebra~$\g_0$ being the component in degree $0$ of the graded Lie algebra $\TT{k-1}\g_0$. For $X\in \TT{k} \g_0$ let $(X_{\bar{0}}, X_{\bar{1}}) \in \T \TT{k-1}\g_0 \simeq \TT{k-1}\g_0 \times \TT{k-1}\g_0$ be the image of $X$ under the canonical embedding $\TT{k} \g_0 \subset \T \TT{k-1}\g_0$. We consider~$E^k$ as a (split) graded space defined by the weight vector field on~$E^k$ given by $\Delta = \sum_{i=0}^{k-1} (i+1)\Delta_{\g_i}$, where $\Delta_{\g_i}$ is the Euler vector field on $\g_i$.
Then the formula
\begin{gather}\label{eqn:from_GLA_to_HA}
\big(\kappa^k\big)_y(X) = \alpha(X_{\bar{1}}) + [y, \alpha(X_{\bar{0}})]_{E^k},
\end{gather}
where $y\in E^k$, turns $E^k$ into a higher Lie algebroid. The proof is left to the reader. (One should notice that vector fields $F_{(a, b)}\in\VF(\g)$, $F_{(a,b)}\colon y\mapsto [y, a]_{\g} + b$, where $\g$ is an arbitrary Lie algebra and $y, a, b\in \g$, form a Lie subalgebra of $\VF(\g)$ isomorphic with $\g \oplus \eps \g$.) We shall show that $\big(E^k,\kappa^k\big)$ is a quotient of $\big(\T^{k-1} \g_0, \kappa^k_{\g_0}\big)$ if $\alpha$ is surjective. Let $J := \ker \alpha = \oplus_{i=0}^{k-1} J_i$, so $\g_i = \g_0/J_i$. Consider the diagram \eqref{diag:df_quotient_kappa_k} in our case:
\begin{gather*}
\xymatrix{
\TT{k} \g_0 \ar[d]_{T^k\alpha_0 = \id_{\T^k \g_0}} \ar@{-|>}[rr]^{\kappa^{k}_{\g_0}} && \T \T^{k-1}\g_0\ar[d]^{\T \alpha}\\
\TT{k} \g_0 \ar@{--|>}[rr]^{ \kappa^{k}} && \T E^k.
}
\end{gather*}
Take $y\in E^k$ and $X\in \T^k \g_0$ as above. Let $\tilde{y} \in y + J$ be any pre-image of $y$ under the projection $\alpha$. Using \eqref{eqn:kappa_l_GG_expd} we find that the vector $\big(\kappa^k_{\g_0}\big)_{\tilde{y}} X \in \T_{\tilde{y}} \T^{k-1}\g_0$ is represented by the curve $t\mapsto \tilde{y} + t (X_{\bar{1}} + [\tilde{y}, X_{\bar{0}}])\in \T^{k-1}\g_0$. Therefore $\T \alpha \big(\big(\kappa^k_{\g_0}\big)_{\tilde{y}} X\big) \in \T E^k$ does not depend on the choice of $\tilde{y}$ and is given by the formula~\eqref{eqn:from_GLA_to_HA} as $\alpha(\tilde{y}) = y$ and $\alpha([\tilde{y}, X_{\bar{0}}]) = [y, \alpha(X_{\bar{0}})]_{E^k}$.
\end{ex}

\section{Final remarks}\label{sec:final}
The main idea of this paper was to use the framework of \Zakrzewski s in order to provide a proper language to describe higher analogs of (Lie) algebroids, having in mind potential applications in variational calculus and geometric mechanics. Our studies suggest a~few (in our opinion interesting) directions of future research, which we discuss below.

{\bf Left-twisted algebroids.} In Lemma~\ref{lem:kappa_core} we studied linear \ZM s $\kappa\colon \T\sigma\relto \tau_E$ which induce the identity on the core. By modifying the latter condition we can define generalization of the concept of an algebroid. For example if $\kappa$ is a linear \ZM\ which induces the identity on the base $\kappa\cap(M\times M)=\graf(\id_M)\subset M\times M$ then the left and right anchors $\rhoL,\rhoR\colon E\ra \T M$ are still well-defined. Moreover, such a $\kappa$ induces a VB endomorphism on the cores $\phi\colon \sigma\ra \sigma$ over $\id_M$. By an analogous argument to the one used in the proof of Proposition~\ref{prop:sigma_from_kappa} such a $\kappa$ defines a \emph{left-$\phi$-twisted algebroid} structure which satisfies a modified equation~\eqref{eqn:bracket}
\begin{gather*}[f\cdot a,g\cdot b]=f\rhoL(a)(g)\cdot\phi(b)-g\rhoR(b)(f)\cdot a+fg\cdot[a,b].\end{gather*}
The converse implication is also true (the argument used in the proof of Lemma \ref{lem:kappa_from_sigma} works without any change): a left-$\phi$-twisted algebroid structure on $\sigma$ uniquely determines a linear \ZM\ $\kappa\colon \T\sigma\relto \tau_E$ which induces an endomorphism $\phi$ on the cores.

For example, a linear connection on the tangent bundle $\T M$ is an $\R$-bi-linear operator $\nabla\colon \X(M)\times\X(M)\ra \X(M)$ satisfying $\nabla_X(f\cdot Y)=f\cdot\nabla_XY+X(f)\cdot Y$ and $\nabla_{f\cdot X}Y=f\cdot\nabla_X Y$ for any vector fields $X,Y\in\X(M)$ and any smooth function $f\in\Cf(M)$. Now $[X,Y]:=-\nabla_YX$ is a left-$\phi$-twisted algebroid structure with both anchors being the identities $\id_{\T M}${,} and $\phi\colon \T M\ra\T M$ being the null map. The related \ZM\ $\kappa$ maps a given element $A\in \T\T M$ to the $\nabla$-horizontal lift of $\tau_{\T M}(A)$ at $\T\tau_M(A)$.

{\bf \Prealg s and reduction.} As we have seen in Section \ref{sec:var_calc} for many interesting applications, we do not need the full structure of a (higher) algebroid, but it is enough to have a~less rigid structure of a (higher) \prealg . It seems to us that such objects can na\-tu\-rally appear as reductions of higher tangent bundles, while considering (pseudo) group actions on the space of smooth curves on manifolds in the spirit of \cite{Kruglikov_Lychagin_2015} (see an example of such a situation in the last paragraph of Section~\ref{sec:var_calc}). Thus it would be interesting to ini\-tiate a systematic study of such objects. For example a structure of a \prealg\ of order one $(\sigma\colon E\ra M, \tau\colon F\ra M, \kappa)$ can be equivalently characterized as a pair of anchor maps $\rho^E\colon E\ra \T M $ and $\rho^F\colon F\ra \T M$ a VB morphism $\phi\colon F\ra E$ over the identity on~$M$ and a bi-linear bracket operation $[\cdot, \cdot]\colon \Sec(E)\times \Sec(F)\ra \Sec(E)$ satisfying, for every sections $a\in\Sec(E)$, $b\in\Sec(F)$ and functions $f,g\in\Cf(M)$, a Leibniz-like rule
\begin{gather*}[f\cdot a,g\cdot b]=f\rho^E(a)(g)\cdot\phi(b)-g\rho^F(b)(f)\cdot a+fg\cdot[a,b].\end{gather*}
In particular, left-twisted algebroids provide examples of \prealg s (with $\tau=\sigma$).

{\bf Applications to variational problems.} In Section~\ref{sec:var_calc} we only sketched some possibilities of applications of the theory of higher (pre)algebroids in geometric mechanics. We postpone a more detailed study to a separate publication. Let us only mention here the problem of developing a constrained higher-order Lagrangian dynamics in this setting. Following the elegant treatment of the first-order case in \cite{GG_var_calc_gen_alg_2008}, for a Lagrangian system constituted by a function $L\colon E^k\ra\R$ on a higher algebroid $\kappa^k\colon \T^k E^1\relto \T E^k$, there are basically two different concepts of constraints. \emph{Vakonomic constraints} should be understood as a graded subbundle $D^k\subset F^k$ and the related dynamics is generated by these admissible variations which are valued in $\T D^k$. On the other hand, \emph{nonholonomic constraints} are constituted by restricting the set of generators to a subbundle $D\subset E^1$. That is, admissible variations are obtained from the restriction $\kappa^k|_{\T^k D}\colon \T^k D\relto \T E^k$.

{\bf Structure of (Lie) higher algebroids.} One of us initiated a study of the internal structure of (Lie) higher algebroids. It turns out that, at least in \degree\ 2, (Lie) higher algebroids have a geometric description in terms of a collection of bundle maps and differential operators which should satisfy certain compatibility conditions. The details can be find in a forthcoming publication \cite{Rotkiewicz_aha}.

\appendix

\section{Appendix -- proofs of technical results}\label{sec:app}

{\bf A proof of Lemma~\ref{lem:kappa_from_sigma}.} 
Note that, according to Lemma \ref{lem:kappa_core}\eqref{cond:core_action}, every relation satisfying~\eqref{cond:kappa_biliner} and~\eqref{cond:core} must respect the natural action of the core $C\simeq E$ on $\T E$. Thus if the assertion holds, the following generalization of formula \eqref{eqn:kappa_from_bracket}
\begin{gather}\label{eqn:kappa_from_bracket_core}
\kappa_a\left[\T b(\rhoL(a))\plus c\right]:=\T a(\rhoR(b))\plus [a,b]\plus c
\end{gather}
is valid for every sections $a,b,c\in\Sec_M(E)$. We shall show that this generalized formula defines a differential relation $\kappa$ satisfying conditions~\eqref{cond:kappa_biliner} and~\eqref{cond:core}. This will end the proof.

Denote $F(a,b):=\T b(\rhoL(a))$ and $G(a,b):=\T a(\rhoR(b))\plus [a,b]$. Leaving aside (for a moment) the problem of correctness of the definition of $\kappa$, note that formula \eqref{eqn:kappa_from_bracket_core} relates an element $F(a,b)\plus c$, whose $\T\sigma$- and $\tau_E$-projections are, respectively, $\rhoL(a)$ and $b$, with an element $G(a,b)\plus c$, for which these projections are $\rhoR(b)$ and $a$, respectively. We clearly see that the \ZM\ $\kappa$ (if correctly defined) covers the left anchor $\rhoL$ and projects to the graph of the right anchor $\rhoR$ under $\tau_E\times\T\sigma$.

We want to prove that $\kappa$ is a \ZM\ over $\rhoL$. Note that every element $B$ of the $\T\sigma$-fibre $(\T E)_{\rhoL(a)}$ for any $a\in E_x$ can be represented as $B=\T b|_x(\rhoL(a))\plus c|_x$ for some sections $b,c\in \Sec_M(E)$, and thus \eqref{eqn:kappa_from_bracket_core} gives us the value of $\kappa_a$ on every element of $(\T E)_{\rhoL(a)}$. (Note that if $\rhoL(a)=0$, then also $\T b|_x(\rhoL(a))$ is a null vector regardless of the choice of section $b$ and hence the initial formula \eqref{eqn:kappa_from_bracket} is not sufficient to determine the values of $\kappa_a$ for every possible element of $(\T E)_{\rhoL(a)}$. This justifies the need of using the extended formula \eqref{eqn:kappa_from_bracket_core}.) It is enough to show that the value $\kappa_a(B)$ does not depend on the chosen presentation $B=F(a,b)|_x\plus c|_x${,} and that the resulting differential relation is bi-linear.

It follows directly from \eqref{eqn:bracket} that for any smooth function $f\in \Cf(M)$ we have $G(f\cdot a,b)=f\cdot_{\T \sigma}G(a,b)$, i.e., $G(a,b)$ is tensorial with respect to $a$, and hence the value of $\kappa_a(B)$ given by \eqref{eqn:kappa_from_bracket_core} (for $B$ as above) does not depend on the particular choice of the section $a$, but only on the value~$a(x)$. Now assume that we present $B=F(a,b)|_x \plus c|_x$ in a different way as $B=F(a,b')|_x\plus c'|_x$. We shall show that formula \eqref{eqn:kappa_from_bracket_core} gives the same value of $\kappa_a(B)$ for both presentations. Clearly, $b|_x=b'|_x$ and hence we may present $b'-b=f\cdot b''$ for some smooth function~$f$ vanishing at~$x$ and some (local) section $b''\in\Sec_M(E)$. By our assumption,
\begin{gather*}
B =F(a,b)|_x\plus c|_x=\T b|_x(\rhoL(a))\plus c|_x=\T b'|_x(\rhoL(a))\plus c'|_x=\T(b+f\cdot b'')|_x(\rhoL(a))\plus c'|_x\\
\hphantom{B}=F(a,b)|_x+f|_x\cdot F(a,b'')|_x\plus\rhoL(a)(f)\cdot b''|_x\plus c'|_x=F(a,b)|_x\plus\rhoL(a)(f)\cdot b''|_x\plus c'|_x,
\end{gather*}
and we conclude that $\rhoL(a)(f)\cdot b''|_x=(c-c')|_x$. Now we can use this fact to get
\begin{gather*}
G(a,b)|_x\plus c|_x =\T a|_x(\rhoR(b))\plus[a,b]|_x\plus c|_x=\T a|_x(\rhoR(b'))\plus[a,b-f\cdot b'']|_x\plus c|_x\\
\hphantom{G(a,b)|_x\plus c|_x}{} =G(a,b')\minus f|_x\cdot[a,b'']|_x\minus \rhoL(a)(f)\cdot b''|_x\plus c|_x\\
\hphantom{G(a,b)|_x\plus c|_x}{} =G(a,b')\minus \rhoL(a)(f)\cdot b''|_x\plus c|_x =G(a,b')|_x\plus c'|_x.
\end{gather*}
It follows that $\kappa_a(B)$ is indeed well-defined.

Finally to show that $\kappa$ defined by formula \eqref{eqn:kappa_from_bracket_core} is bi-linear we need to check that if $(B,A)\in \kappa$ then also $(f\cdot_{\T\sigma}B,f\cdot_{\tau_E}A)\in \kappa$ and $(f\cdot_{\tau_E}B,f\cdot_{\T\sigma}A)\in \kappa$. This can be easily done using the following properties od $F$ and $G$:
\begin{gather}
G(f\cdot a,b)=f\cdot_{\T\sigma} G(a,b), \qquad F(f\cdot a,b)=f\cdot_{\tau_E} F(a,b),\nonumber\\
G(a,f\cdot b)=f\cdot_{\tau_E} G(a,b)\plus\rhoL(a)(f)b,\qquad F(a,f\cdot b)=f\cdot_{\T\sigma} F(a,b)\plus\rhoL(a)(f)b,\label{eqn:rescalling}
\end{gather}
which follow directly from \eqref{eqn:bracket}. 

{\bf A proof of Proposition~\ref{prop:sigma_from_kappa}.} 
From the results of Lemma~\ref{lem:kappa_core} we already know that every linear \ZM\ inducing the identity on the cores gives rise to a well-defined left and right anchor maps. Thus we need only to check if formula \eqref{eqn:bracket_from_kappa} defines an algebroid bracket compatible with these anchors. First note that given any sections $a,b\in \Sec_M(E)$ the right-hand side of~\eqref{eqn:bracket_from_kappa} is a difference of two elements of $\T E$ with the same $\tau_E$-projection $a$ and the same $\T\sigma$-projection~$\rhoR(b)$, thus a vertical vector. That is, $[a,b]\in\Sec_M(E)$ is well-defined. To check that this bracket satisfies the Leibniz rule \eqref{eqn:bracket} we have to study the behaviour of formula \eqref{eqn:bracket_from_kappa} under rescaling $a\mapsto f\cdot a$ and $b\mapsto g\cdot b$. The calculations are basically the same as in the proof of Lemma~\ref{lem:kappa_from_sigma}~-- see formulas \eqref{eqn:rescalling}. 

{\bf A proof of Proposition~\ref{prop:alg_morphism}.} 
Note first that if we have a $\catZM$-morphism \mbox{$(\T \phi,\T\phi)\colon \kappa\mZM\kappa'$}, then (see diagram \eqref{d:ZMtoZM}) the commutativity at the level of base maps (the left anchors) gives us precisely the commutativity of the first diagram in \eqref{eqn:compatible_anchors}. The commutativity of the se\-cond diagram follows from an observation that $(\T \phi,\T\phi)$ is also a $\catZM$-morphism between~$\kappa^\T$ and~${\kappa'}^\T$, whose base maps are the right anchors (cf.\ Lemma~\ref{lem:kappa_core} and Remark~\ref{rem:kappa_T}). This is an immediate consequence of Proposition~\ref{p:ZM-morphism} where the characterisation of being a $\catZM$-morphism does not depend on the direction of the considered relations, in contrast to the formulation of Definition~\ref{def:zm_morphism}. From now on we can thus assume that \eqref{eqn:compatible_anchors} holds.

Choose any two sections $a,b\in\Sec_M(E)$ and let $\phi_\ast a=\sum_i f_i\cdot \und{\phi}^\ast a_i$ and $\phi_\ast b=\sum_j g_i\cdot \und{\phi}^\ast b_j$ as in Definition~\ref{def:algebroid_morphism}. In the following calculations we understand a section $s\in\Sec_M(\und{\phi}^\ast E')$ as a map $s\colon M\ra E'$ such that $s(x)\in E'_{\und{\phi}(x)}$ and thus $\V_{s}{s'}$, where $s, s'\in \Sec_M(\und{\phi}^\ast E')$, is a map $M \ra \T E'$, $x\mapsto \V_s{s'}(x)$, where $\V_s{s'}(x)$ is represented by the curve $t\mapsto s(x)+ t s'(x) \in E'_{\und{\phi}(x)}$. We have
\begin{gather*}
\T\phi [\T a(\rhoR(b)) ]= \sum_i [\rhoR(b)(f_i)\cdot \V_{\phi_\ast a}\und{\phi}^\ast a_i+f_i\cdot_{\T \sigma'} \T a_i \T \und{\phi}(\rhoR(b)) ]\\
\hphantom{\T\phi [\T a(\rhoR(b)) ]}{} \overset{\eqref{eqn:compatible_anchors}}= \sum_i [\rhoR(b)(f_i)\cdot\V_{\phi_\ast a} \und{\phi}^\ast a_i+f_i \cdot_{\T\sigma'} \T a_i \, \rhoR'(\phi_\ast b) ]\\
\hphantom{\T\phi [\T a(\rhoR(b)) ]}{}=\sum_i\rhoR(b)(f_i)\cdot \V_{\phi_\ast a} \und{\phi}^\ast a_i +\sum_{i,j}f_i\cdot_{\T\sigma'}g_j\cdot_{\tau_{E'}} \T a_i \rhoR'(\und{\phi}^\ast b_j).
\end{gather*}

Similarly,
\begin{gather*}\T\phi [\T b(\rhoL(a)) ]=
\sum_j\rhoL(a)(g_j)\cdot \V_{\phi_\ast b}\und{\phi}^\ast b_j+\sum_{i,j}g_j\cdot_{\T\sigma'} f_i \cdot_{\tau_{E'}}\T b_j \rhoL'(\und{\phi}^\ast a_i)\end{gather*}
and we conclude that
\begin{gather*}
\kappa'_{\phi(a)} [\T\phi [\T b(\rhoL(a)) ] ]-\T\phi [\T a(\rhoR(b)) ] \\
\qquad{}= \V_{\phi_\ast a}\bigg[\sum_j\rhoL(a)(g_j)\cdot \und{\phi}^\ast b_j-\sum_i\rhoR(b)(f_i)\cdot
 \und{\phi}^\ast a_i +\sum_{i,j}f_ig_j\cdot\und{\phi}^\ast[a_i,b_j]'\bigg].
\end{gather*}
Applying the tangent map $\T\phi$ to formula \eqref{eqn:bracket_from_kappa} leaves us with
\begin{gather*}\V_{ \phi_\ast a}\phi_\ast[a,b]:=\T\phi [\kappa_a [\T b(\rhoL(a)) ] ]-\T\phi [\T a(\rhoR(b)) ].\end{gather*}
The comparison of the latter two equations shows that~\eqref{eqn:algebroid_morphism} holds if and only if $\T\phi [\kappa_a\! [\T b(\rhoL(a)) ] ]\!\!$ $=\kappa'_{\phi(a)} [\T\phi [\T b(\rhoL(a)) ] ]$. The latter equality (for every~$a$ and~$b$) is equivalent to $(\T\phi,\T\phi)$ being a~$\catZM$-morphism between~$\kappa$ and~$\kappa'$. 

{\bf Equivalent characterizations of algebroid morphisms.}
\begin{prop}\label{prop:alg_morphism_eqv} Let $\sigma\colon E\ra M$ and $\sigma'\colon E'\ra M'$ be vector bundles carrying $($Lie$)$ algebroid structures as in Definition~{\rm \ref{def:algebroid_morphism}}. Let
 $\phi\colon E\ra E'$ over $\und{\phi}\colon M\ra M'$ be a VB morphism. The following are equivalent:
\begin{enumerate}[$(i)$]\itemsep=0pt
\item \label{cond:alg_morph_Mac} $\phi$ is an algebroid morphism $($in the sense of Definition~{\rm \ref{def:algebroid_morphism})}.
\item \label{cond:alg_morph_graph} The graph of $\phi$ is a subalgebroid of the product $($Lie$)$ algebroid on $\sigma\times \sigma'$.
\item \label{cond:alg_morph_Leibniz} The graph of $\phi^\ast$ is a coisotropic submanifold of the Leibniz manifold~{\rm \cite{JG_PU_Algebroids_1999}} $(E^\ast\times {E'}^\ast, -\Lambda_E + \Lambda_{E'})$.
\end{enumerate}
\end{prop}
\begin{proof}Recall \cite{JG_PU_Algebroids_1999} that a \emph{Leibniz manifold} $(M, \Lambda)$ is a manifold $M$ equipped with a contravariant 2-tensor $\Lambda$, and that \emph{linear} contravariant 2-tensors on the dual $\sigma^\ast\colon E^\ast\ra M$ of a vector bundle $\sigma\colon E \ra M$ are in one-to-one correspondence with general algebroid structures on $\sigma$. This correspondence is completely parallel to that between linear Poisson tensors and Lie algebroid structures.

The equivalence between \eqref{cond:alg_morph_graph} and \eqref{cond:alg_morph_Leibniz} is proved in \cite[Theorem~5.1]{JG_mod_class_skew_alg_rel_2012} for skew algebroids, however, the proof can be directly rewritten for general algebroids. Thus we need only to show the equivalence between~\eqref{cond:alg_morph_Mac} and~\eqref{cond:alg_morph_graph}.

Let $F\subset E\times E'$, $N\subset M\times M'$ be the graphs of $\phi$ and $\und{\phi}$, respectively, and assume that $\sigma\times \sigma'|_F\colon F\ra N$ is a subalgebroid of $\sigma\times \sigma'$. As $\phi$ is a vector bundle morphism, the base map~$\und{\phi}$ is smooth. Because $(\rhoL, \rhoL')$ and $(\rhoR, \rhoR')$ map $F$ to $\T N \subset \T(M\times M')$ we have commuting diagrams in~\eqref{eqn:compatible_anchors}. To prove~\eqref{eqn:algebroid_morphism} take sections $a,b\in\Sec_M(E)$, write $\phi\circ a = \sum_i f_i \und{\phi}^\ast a_i$ and similarly, $\phi\circ b = \sum_j g_j \und{\phi}^\ast b_j$
and notice that $(x, \und{\phi}(x))\mapsto (a(x), \phi \circ a(x))$ is a section of $F$ whose extension to a section of $\sigma\times \sigma'$ can be written as follows
\begin{gather*}
\tilde{a}(x, y) = \bigg(a(x), \sum_i f_i(x) a_i(y)\bigg),
\end{gather*}
and similarly for section $b$. We clearly have
\begin{gather*}
[\tilde{a}, \tilde{b}]_{\sigma\times \sigma'}(x, y) \\
\qquad{} = \bigg([a, b](x), \sum_{i,j} f_i(x) g_j(x) [a_i, b_j]'(y) + \rhoL(a)(g_j)(x)b_j(y) - \rhoR(b)(f_i)(x) a_i(y) \bigg),
\end{gather*}
hence $[\tilde{a}, \tilde{b}](x, \und{\phi}(x)) \in F$ implies~\eqref{eqn:algebroid_morphism}. This reasoning can be inverted proving the equivalence of conditions \eqref{cond:alg_morph_Mac} and \eqref{cond:alg_morph_graph}.
\end{proof}

{\bf A proof of Proposition~\ref{prop:algebroidal_relation}.} 
By definition{,} condition \eqref{cond:alg_rel_1} is equivalent to $\kappa$ being a~subalgebroid of the product (Lie) algebroid structure on $\sigma_1\times\sigma_2$. Note that for this structure the left and right anchors are, respectively, $\rhoL:=(\rho_{1L},\rho_{2L})$ and $\rhoR:= (\rho_{1R},\rho_{2R})$, whereas the algebroid bracket reads as
\begin{gather*}[(s_1,s_2),(s_1',s_2')]:=([s_1,s_1']_1,[s_2,s_2']_2)\end{gather*}
for any sections $s_1,s_1'\in\Sec_{M_1}(E_1)$ and $s_2,s_2'\in\Sec_{M_2}(E_2)$. It is thus enough to prove that conditions \eqref{sub_cond:a} and \eqref{sub_cond:b} are equivalent to conditions~\eqref{df:sub_algebroid_1} and~\eqref{df:sub_algebroid_2} of Definition~\ref{def:sublagebroid} taken for $\sigma':=\sigma_1\times\sigma_2|_r\colon E_1\times E_2\supset r=:E'\ra M':=\graf(\und{r})\subset M_1\times M_2$, $\sigma:=\sigma_1\times \sigma_2$ and $\kappa:=\kappa_1\times\kappa_2$.

First observe that for sections $s_1\in \Sec_{M_1}(E_1)$ and $s_2\in \Sec_{M_2}(E_2)$ the restriction $(s_1, s_2)|_{M'}$ is a section of the vector subbundle $\sigma'\subset \sigma_1\times \sigma_2$ if and only if $s_2=\hat{r}(s_1)$ and, moreover, every section of $\sigma'$ can be presented as such a restriction. Thus condition \eqref{sub_cond:a} is equivalent to the fact that $\sigma_I=(\sigma_{1I},\sigma_{2I})$ maps $E'\subset E_1\times E_2$ to $\T M'=\T\graf(\und{r})\subset\T M_1\times \T M_2$ for $I=L,R$. That is precisely condition \eqref{df:sub_algebroid_1} of Definition~\ref{def:sublagebroid} for $\sigma'$, $\sigma$ and $\kappa$ as above.

Next note that condition \eqref{sub_cond:b} means that the algebroid bracket on $\sigma_1\times\sigma_2$ is closed with respect to sections of the form $(s_1,\wh{r}(s_1))$, where $s_1\in\Sec_{M_1}(E_1)$. As we already observed such sections are $\sigma_1\times\sigma_2$-extensions of all possible sections of $\sigma'$. Since, by the remark following Definition~\ref{def:sublagebroid}, condition \eqref{df:sub_algebroid_1} of this definition guarantees that the induced bracket defined on~$\sigma'$ does not depend on the choice of extensions of sections of~$\sigma'$, we conclude that if~\eqref{sub_cond:a} holds then~\eqref{sub_cond:b} is equivalent to condition \eqref{df:sub_algebroid_2} of Definition~\ref{def:sublagebroid} for $\sigma'$, $\sigma$ and $\kappa$ as above. This ends the proof. 

{\bf A proof of Lemma~\ref{lem:char_kappa}.} 
The relation between the algebroid bracket and \ZM\ $\kappa$ is given by formula~\eqref{eqn:kappa_from_bracket}. As has been observed in Remark~\ref{rem:kappa_T}, $\kappa^T$ corresponds to another general algebroid structure with the bracket $[a,b]^T:=-[b,a]$ (note that this passage is possible due to the fact that $\kappa$ is bi-linear and that it induces the identity on the core). Consequently the symmetry condition $\kappa=\kappa^T$ is equivalent to the antisymmetry of the bracket $[a,b]=-[b,a]$. This proves~\eqref{item:kappa_skew}.

Assume now that the algebroid structure on $\sigma$ is skew. To prove \eqref{item:kappa_al} apply $\T\rho$ to~\eqref{eqn:bracket_from_kappa} to get
\begin{gather*}\V_{\rho(a)}\rho([a,b])=\T\rho (\kappa_a [\T b(\rho(a)) ] )-\T \rho(a)(\rho(b)).\end{gather*}
On the other hand, formula~\eqref{eqn:bracket_from_kappa} for the commutator of vector fields $\rho(a),\rho(b)\in \Sec_M(\T M)$ gives us
\begin{gather*}\V_{\rho(a)}[\rho(a),\rho(b)]_{\T M}:=(\kappa_{\T M})_{\rho(a)} [\T\rho(b)(\rho(a)) ]-\T \rho(a)(\rho(b)).\end{gather*}
By comparing the above two formulas it is clear that \eqref{eqn:al_algebroid} is equivalent to~\eqref{eqn:kappa_AL}.

 Point \eqref{item:kappa_lie} is equivalent to the fact that the dual map $\kappa^\ast\colon \T^\ast E\ra \T E^\ast$ (cf.\ Remark~\ref{rem:kappa_dual}) is a~linear Poisson map, which is a~well-known equivalent characterization of the Jacobi identity~\cite{JG_PU_Algebroids_1999}. By the results of~\cite{JG_mod_class_skew_alg_rel_2012} the latter condition is equivalent to $\kappa$ being the algebroidal relation. 

{\bf A proof of Proposition~\ref{prop:red_h_to_l}.}
\begin{enumerate}[(i)]\itemsep=0pt
\item Let $\big(x^a_{w}, X^A_{w}\big)$, $\big(y^i_{w}, Y^I_{w}\big)$ be graded coordinates on $E_1^k$ and $E_2^k$, respectively, where the capital letters indicate linear coordinates, i.e., $X^A_{w}$ and $Y^I_{w}$ are of weight $(1, w)$. A \ZM\ $r$ is given locally in $E_1^k \times E_2^k$ by a system of equations of the form
\begin{gather}\label{e:local_r}
r\colon \
\begin{cases}
x^a_{w} = f^a_{w}(y), \\
Y^I_{w} = \sum_A X^A_{w'} g_{A, w-w'}^I(y),
\end{cases}
\end{gather}
where $f^a_{w}$ and $g_{A, w}^I$ are (local) functions on $M_2^k$ of weight $w$. Equations for $r^{j} \subset E_1^j\times E_2^j$ are obtained from \eqref{e:local_r} by removing all equations in which at least one coordinate has weight greater than $j$. Clearly, such a system defines a \ZM\ as $r^j$ is a vector subspace of $E_1^j\times E_2^j \ra M_1^0\times M_2^0$ with the property: given $(y^i_{w})\in M_2^j$ and $\big(x_{w}^a, X_{w}^A\big) \in E_1^j$ such that $x^a_{w} = f^a_{w}(y)$ there is a unique $\big(y^i_{w}, Y^I_{w}\big) \in E_2^j$ (namely, $Y^I_{w}=\sum_A X^A_{w'} \, g_{A, w-w'}^I(y)$) satisfying the equations for $r^j$.

\item In view of Proposition \ref{p:ZM-morphism}, we should check whether $\big(A^j, B^j\big)\in r^j$ implies $\big(\phi_1^j\big(A^j\big), \phi_2^j\big(B^j\big)\big)$ $\in (r')^j$. This is straightforward: there exist $\big(A^k, B^k\big) \in r^k$ which project to $\big(A^j, B^j\big)$. We have $\big(\phi_1\big(A^k\big), \phi_2\big(B^k\big)\big) \in (r')^k$ as $(\phi_1, \phi_2)\colon r \mZM r'$ is a $\catZM$-morphism. But the latter element projects under the bundle reduction map $(\sigma_1)^k_j \times (\sigma_2)^k_j$ to $\big(\phi_1^j\big(A^j\big), \phi_2^j\big(B^j\big)\big)$ which is therefore an element of~$r^j$. 
\end{enumerate}

{\bf A proof of Proposition~\ref{p:algebroid_Tk_E}.} 
By Proposition~\ref{prop:sigma_from_kappa}, to prove that $\dd_{\TT{k}}\kappa$ defines a general algebroid structure on $\TT{k}\sigma$ it is enough to show that it is a linear \ZM\ and that it induces the identity on the cores. By the results of \cite[Theorem~A.5]{MJ_MR_models_high_alg_2015}, the higher tangent lift~$\TT{k}\kappa$ has these two desired properties. So do~$\kappa^k_E$ and $(\kappa^k_E)^{-1}$, which are not only \ZM s but, in fact, even isomorphisms of graded-linear bundles. We conclude that $\dd_{\TT{k}}\kappa$, which is a composition of these three \ZM s, also shares these properties, i.e., $\dd_{\TT{k}}\kappa$ indeed defines a graded algebroid structure. In fact, all the considered \ZM s are also $\TT{k}\R$-linear, with respect to the natural $\TT{k}\R$-action on the $\TT{k}$-lift of a vector bundle introduced in at the beginning of Section~\ref{ssec:higher_lie}. It follows that $\dd_{\TT{k}}\kappa$ is also $\TT{k}\R$-linear. We will use this property shortly.

The formula for anchor maps follows immediately from the definition of $\dd_{\TT{k}}\kappa$ -- we simply need to calculate the relevant projections of this \ZM\ to the legs of the considered DVBs.

\looseness=-1 Consider now any two sections $s_1,s_2\in\Sec_M(E)$. By applying the functor $\TT{k}$ to \eqref{eqn:kappa_from_bracket}, we get
\begin{gather*}\TT{k}\kappa_{\TT{k} s_1}\big[\TT{k}\T s_2\big(\TT{k}\rhoL\big(\TT{k} s_1\big)\big)\big]:=\TT{k}\T s_1\big(\TT{k}\rhoR\big(\TT{k}s_2\big)\big)\plus \TT{k}[s_1,s_2].\end{gather*}
Note that in the above formula element $\TT{k}[s_1,s_2]$ belongs to the core. Now let us compose this equality with maps $\kappa^k_E$ and $\big(\kappa^k_E\big)^{-1}$, which are core-identities, and use the definition of $\dd_{\TT{k}}\kappa$ to get
 \begin{gather*}(\dd_{\TT{k}}\kappa)_{\TT{k} s_1}\big[\T\TT{k} s_2\big(\dd_{\TT{k}}\rhoL\big(\TT{k} s_1\big)\big)\big]:=\T\TT{k} s_1\big(\dd_{\TT{k}}\rhoR\big(\TT{k}s_2\big)\big)\plus \TT{k}[s_1,s_2].\end{gather*}
By comparing the above with formula \eqref{eqn:kappa_from_bracket} for the algebroid structure $\big(\TT{k}\sigma, \dd_{\TT{k}}\kappa\big)$, we may conclude that the algebroid bracket $[\cdot,\cdot]_{\dd_{\TT{k}}}$ on $\TT{k}\sigma$ satisfies
\begin{gather*}\big[\TT{k}s_1,\TT{k} s_2\big]_{\dd_{\TT{k}}}=\TT{k}[s_1,s_2],\qquad \text{i.e.,}\qquad \big[s_1^{(k)},s_2^{(k)}\big]_{\dd_{\TT{k}}}=[s_1,s_2]^{(k)}\end{gather*}
for any sections $s_1,s_2\in\Sec_M(E)$. What is more, since, as we have observed before, $\dd_{\TT{k}}\kappa$ is $\TT{k}\R$-linear, the resulting bracket operation is $\TT{k}\R$-bilinear, and hence~\eqref{p:algebroid_Tk_E_item_bracket_formula} holds:
\begin{gather*}
\left[ \frac{k!}{(k-\alpha)!}s_1^{(k-\alpha)}, \frac{k!}{(k-\beta)!}s_2^{(k-\beta)}\right]_{\dd_{\TT{k}}}= \big[\eps^\alpha\cdot s_1^{(k)},\eps^\beta\cdot s_2^{(k)}\big]_{\dd_{\TT{k}}}\\
\hphantom{\left[ \frac{k!}{(k-\alpha)!}s_1^{(k-\alpha)}, \frac{k!}{(k-\beta)!}s_2^{(k-\beta)}\right]_{\dd_{\TT{k}}}}{} =\eps^{\alpha+\beta}\cdot[s_1,s_2]^{(k)}= \frac{k!}{(k-\alpha-\beta)!} [s_1,s_2]^{(k-\alpha-\beta)}.
\end{gather*}
This formula completely determines the algebroid bracket on $\TT{k}\sigma$ since $(k-\alpha)$-lifts of sections of~$\sigma$ for all possible~$\alpha$'s span the module of sections of~$\TT{k}\sigma$ (cf.\ the remark following Definition~\ref{def:lift_section}).

Finally, the fact that the lifted algebroid structure preserves the properties of being \linebreak skew/AL/Lie follows directly from the anchor and the bracket formulas. 

{\bf A proof of Proposition~\ref{prop:sub_algebroidal_rel}.} 
 To prove \eqref{item:T^kr restricts_fine} note first that, by \cite[Theorem~A.5]{MJ_MR_models_high_alg_2015}, the relations $\TT{k} r$ and $\TT{k} r'$ are \ZM s. We have to prove that $\TT{k} r$ restricts fine to a subbundle $\TT{k} E_1'\times \TT{k} E_2'$. Take a $k$-jet $\und{v}_2^k\in \TT{k} M_2'$ which is represented by a curve $\und{\gamma}_2(t)\in M_2'$ and denote $\und{\gamma}_1(t):=\und{r}(\und{\gamma}_2(t))$. Since $\und{r}(M_2')\subset M_1'$, the curve $\und{\gamma}_1(t)$ lies in $M_1'$ . Let us take any $k$-jet $v_1^k\in \TT{k} E_1'$ which lies over $\und{v}_1^k :=\tclass{k}{\und{\gamma}_1} \in \TT{k} M_1'$. We can choose a representative $\gamma_1(t)\in E_1$ of $v_1^k$ such that~$\gamma_1(t)$ projects to $\und{\gamma}_1(t)$. Then we may write
\begin{gather*}
\big(\TT{k} r\big)_{\und{v}_2^k}\big(v_1^k\big) = \tclass{k}{ r_{\und{\gamma}_2(t)}(\gamma_1(t))}\in \TT{k} E_2',
\end{gather*}
which ends the proof of point \eqref{item:T^kr restricts_fine}.

 To show \eqref{item:T^kE_as_subalgeroid} observe that by our hypothesis, $\kappa$ restricts fine to $\T E'\times \T E'$, hence by \eqref{item:T^kr restricts_fine} $\TT{k} \kappa\colon \TT{k} \T \sigma\relto \TT{k} \tau_{E}$ restricts fine to $\TT{k}\T E'\times \TT{k} \T E'$. Since $\dd_{\TT{k}}\kappa$ is obtained by composing~$\TT{k} \kappa$ with the vector bundle isomorphisms $\kappa^k_E$ and its inverse, the relation $\dd_{\TT{k}}\kappa$ also restricts fine to $\T \TT{k} E'\times \T \TT{k} E'$, hence $\big(\TT{k} \sigma',\dd_{\TT{k}}\kappa'\big)$ is a subalgebroid of $\big(\TT{k} \sigma, \dd_{\TT{k}}\kappa'\big)$, as was claimed.

 Finally, point \eqref{item:T^kr_also_algebroidal} follows easily from~\eqref{item:T^kE_as_subalgeroid}. Since $r$ is a subalgebroid of $(\sigma_1\times\sigma_2,\kappa_1\times\kappa_2)$, by~\eqref{item:T^kE_as_subalgeroid} $\TT{k} r$ is a subalgebroid of $\big(\TT{k}(\sigma_1\times \sigma_2), \dd_{\TT{k}}(\kappa_1\times\kappa_2)\big) \simeq \big(\TT{k} \sigma_1\times \TT{k}\sigma_2,\dd_{\TT k}\kappa_1\times\dd_{\TT k}\kappa_2\big)$, hence $\TT{k} r$ is an algebroidal relation. 

{\bf A proof of Proposition~\ref{prop:E_k_as_AL_HA}.} 
 Assume first that $(\sigma,\kappa)$ is an AL algebroid. From the results of \cite[Theorem~4.5(ix)]{MJ_MR_models_high_alg_2015} we already know that $\kappa^{[k]}$ is a \ZM . We shall prove first that it is a weighted \ZM . For this it is enough to check that $\kappa^{[k]}$ is a $\N$-graded submani\-fold of $\tau^k_E\times \T\tau^{[k]}\colon \TT{k} E\times \T\E{k} \ra E\times \T M$. A general idea is to consider the latter as a $\N$-graded submanifold of $\TT{k-1}\T E\times \T\TT{k-1} E\ra E\times \T M$ and notice that although $\kappa$ interchanges the two homogeneity structures on $\T E$, the homogeneity structure in our concern (which is defined by the sum of weight vector fields of degrees $1$ and $k-1$) is respected by both $\TT{k-1}\kappa$ and $\kappa^{k-1}_E$, thus also by $\kappa^{[k]}$ (which is a fine restriction of the composition of these relations). In detail,
$\T\TT{k-1} E$ is a triple-\grB\ of $3$-\degree\ $(1, k-1, 1)$ with legs $\T\TT{k-1} \tau\colon \T\TT{k-1} E\ra \T\TT{k-1} M$, $\T\tau^{k-1}_E\colon \T\TT{k-1} E \ra \T E$, and $\tau_{\TT{k-1} E}\colon \T\TT{k-1} E\ra \TT{k-1} E$, respectively. The multi-\grB\ structure on $\T\TT{k-1} E$ is defined by the weight vector fields $\T\TT{k-1}\Delta_E$, $\T\Delta^{k-1}_E$, and~$\Delta^1_{\TT{k-1} E}$. Since $\kappa$ is homogeneous with respect to $\big(\Delta^1_E, \T\Delta_E\big)\in\VF(\T E \times \T E)$ and $\big(\T\Delta_E, \Delta^1_E\big)$, its $\thh{(k-1)}$-tangent lift $\TT{k-1}\kappa$ is homogeneous with respect to
\begin{gather*}\big(\TT{k-1}\Delta^1_E, \TT{k-1}\T\Delta_E\big), \big(\TT{k-1}\T\Delta_E, \TT{k-1}\Delta^1_E\big), \big(\Delta^{k-1}_{\T E}, \Delta^{k-1}_{\T E}\big) \in \VF\big(\TT{k-1} \T E \times \TT{k-1} \T E\big).\end{gather*}
Similarly, $\kappa^{k-1}_E$ is homogeneous with respect to
\begin{gather*}
\big(\TT{k-1}\T\Delta_E, \T\TT{k-1} \Delta_E\big), \big(\Delta^{k-1}_{\T E}, \T\Delta^{k-1}_E\big), \big(\TT{k-1} \Delta^1_E, \Delta^1_{\TT{k-1} E}\big) \in \VF\big(\TT{k-1}\T E\times \T \TT{k-1} E\big).
\end{gather*}
In particular, $\TT{k-1}\kappa$ is homogeneous with respect to $\big(\TT{k-1}\Delta^1_E+\Delta^{k-1}_{\T E}, \TT{k-1}\T\Delta_E + \Delta^{k-1}_{\T E}\big)$ while $\kappa^{k-1}_E$ is homogeneous with respect to $\big(\TT{k-1}\T\Delta_E + \Delta^{k-1}_{\T E}, \T\TT{k-1} \Delta_E+\T \Delta^{k-1}_E\big)$ thus the composition $\kappa^{k-1}_E\circ \TT{k-1}\kappa$ is homogeneous with respect to $\big(\TT{k-1}\Delta^1_E+\Delta^{k-1}_{\T E}, \T\TT{k-1} \Delta_E+\T \Delta^{k-1}_E\big)$. The result follows because the \degree-$k$ \grB\ structures on $\TT{k} E \ra E$ and $\T\E{k}\ra \T M$ are encoded by $\TT{k-1}\Delta^1_E+\Delta^{k-1}_{\T E}$ and $\T\TT{k-1} \Delta_E+\T \Delta^{k-1}_E$, respectively (see \cite[Theorem~4.5(i)]{MJ_MR_models_high_alg_2015}).

The commutativity of the diagram
\begin{gather}\label{d:ZMrr}\begin{split}&
\xymatrix{
& \TT{k} E\ar[ld]_{\TT{k} \rho} \ar[dd]^<<<<<{\TT{k} \tau} \ar@{-|>}[rr]^{ \kappa^{[k]}} && \T \E{k} \ar[ld]_{\T \rho^{[k]}}\ar[dd]^{\tau_{\E{k}}} \\
\TT{k}\T M\ar[dd]_{\TT{k} \tau_ M} \ar@{-|>}[rr]^<<<<<<<<<<{\kappa^k_M} && \T \TT{k} M \ar[dd]^<<<<<<<{\tau_{\TT{k} M}} & \\
& \TT{k} M\ar[ld]_{=} && \E{k} \ar[ld]_{ \rho^{[k]}} \ar[ll]_>>>>>>>>>{ \rho^{[k]}} \\
\TT{k} M && \TT{k} M \ar[ll]_{=} &
}\end{split}
\end{gather}
follows immediately from \cite[Theorem~4.5(xi)]{MJ_MR_models_high_alg_2015}, while the reduction of $\kappa^{[k]}$ from \degree\ $k$ to \degree~$1$ coincides obviously with the relation $\kappa$ defining AL algebroid structure on $E$.

Now if $(\sigma,\kappa)$ is Lie, $\kappa$ is an algebroidal relation between $\T \sigma\colon \T E\ra \T M$ and $\tau_E\colon \T E\ra E$. If follows from Proposition~\ref{prop:sub_algebroidal_rel}\eqref{item:T^kr_also_algebroidal} that $\TT{k-1}\kappa$ is an algebroidal relation between the $\st{(k-1)}$-tangent lift algebroids on $\TT{k-1}\T \sigma$ and $\TT{k-1}\tau_E$. What is more, $\kappa^{k-1}_E$ is an algebroid isomorphism between $\TT{k-1}\tau_E$ and $\tau_{\TT{k-1}E}$, and thus the composition $\kappa^{k-1}_E\circ \TT{k-1}\kappa$ is an algebroidal relation between $\TT{k-1}\T \sigma$ and $\tau_{\TT{k-1}E}$. By \cite[Proposition~4.6]{MJ_MR_models_high_alg_2015}, $\kappa^{k-1}_E\circ \TT{k-1}\kappa\subset\TT{k-1}\T E\times\T\TT{k-1} E$ restricts fine to $\kappa^{[k]}\subset\TT{k} E\times \T E^{[k]}$, and so, in light of Proposition~\ref{prop:sub_r_also_algebroidal}, $\kappa^{[k]}$ is also an algebroidal relation, i.e., $\big(\sigma^{[k]},\kappa^{[k]}\big)$ is Lie.

It remains to prove that $\big(\sigma^{[k]}, \kappa^{[k]}\big)$ is strong. We shall proceed by induction on $k$ and assume that $\kappa^{[k-1]}$ is the identity on the core bundle. To show the same for $\kappa^{[k]}$ we observe that the reduction from order $k$ to $k-1$ of $\kappa^{[k]}$ is $\kappa^{[k-1]}$, hence it is enough to show that $ \kappa^{[k]}$ is the identity on the ultracore bundle. But the inclusions $\TT{k} E\subset \TT{k-1}\T E$ and $\T \E{k}\subset \T \TT{k-1} E$ of graded-linear bundles give identity on the ultracores. The same is true for relations~$\TT{k-1} \kappa$ and~$\kappa^{k-1}_E$ what completes our proof. 

\subsection*{Acknowledgments}
The authors are grateful to the anonymous referees who put a lot of effort in improving the quality and clarity of the manuscript. This research was supported by the Polish National Science Center under the grant DEC-2012/06/A/ST1/00256.

\LastPageEnding

\end{document}